\documentclass[11pt]{amsart}
\usepackage{amscd,amssymb,amsthm,amsmath,amssymb,mathrsfs,enumerate}
\usepackage[matrix,arrow,curve]{xy}
\usepackage[margin=2cm]{geometry}

\newtheorem{theorem}{Theorem}[section]

\newtheorem{lemma}[theorem]{Lemma}
\newtheorem*{lemma*}{Lemma}
\newtheorem{corollary}[theorem]{Corollary}

\newtheorem*{maintheorem*}{Theorem~\ref{theorem:2}}
\newtheorem*{corollary*}{Corollary}

\newtheorem{mainthm}{Theorem}

\theoremstyle{definition}
\newtheorem{example}[theorem]{Example}
\newtheorem*{example*}{Example}

\theoremstyle{remark}
\newtheorem{remark}[theorem]{Remark}

\makeatletter\@addtoreset{equation}{section} \makeatother

\title{K-stability of pointless del Pezzo surfaces and Fano 3-folds}

\author{Hamid Abban, Ivan Cheltsov, Takashi Kishimoto, Fr\'ed\'eric Mangolte}

\let\origmaketitle\maketitle
\def\maketitle{
  \begingroup
  \def\uppercasenonmath##1{} 
  \let\MakeUppercase\relax 
  \origmaketitle
  \endgroup
}

\begin{document}

\begin{abstract}
We explore connections between existence of $\Bbbk$-rational points for Fano varieties defined over $\Bbbk$, a subfield of $\mathbb{C}$, and existence of K\"ahler-Einstein metrics on their geometric models.
First, we show that geometric models of del Pezzo surfaces with at worst quotient singularities
defined over $\Bbbk\subset\mathbb{C}$ admit (orbifold) K\"ahler--Einstein metrics if they do not have $\Bbbk$-rational points. Then we prove the same result for smooth Fano 3-folds with 8 exceptions. Consequently, we explicitly describe several families of pointless Fano 3-folds whose geometric models admit K\"ahler-Einstein metrics. In particular, we obtain new examples of prime Fano 3-folds of genus $12$ that admit K\"ahler--Einstein metrics. Our result can also be used to prove existence of rational points for certain Fano varieties, for example for any smooth Fano 3-fold over $\Bbbk\subset\mathbb{C}$ whose geometric model is strictly K-semistable.
\end{abstract}

\address{\emph{Hamid Abban}\newline
\textnormal{University of Nottingham, Nottingham, England
\newline
\texttt{hamid.abban@nottingham.ac.uk}}}

\address{ \emph{Ivan Cheltsov}\newline
\textnormal{University of Edinburgh, Edinburgh, Scotland
\newline
\texttt{i.cheltsov@ed.ac.uk}}}

\address{\emph{Takashi Kishimoto}\newline \textnormal{Saitama University, Saitama, Japan
\newline
\texttt{kisimoto.takasi@gmail.com}}}

\address{\emph{Fr\'ed\'eric Mangolte}\newline
\textnormal{Aix Marseille University, CNRS, I2M, Marseille, France}
\newline
\textnormal{\texttt{frederic.mangolte@univ-amu.fr}}}

\maketitle

\tableofcontents

Throughout this paper, all varieties are assumed to be projective, normal and geometrically irreducible.

\section{Introduction}
\label{section:intro}

The study of K\"ahler--Einstein metrics is a half-century old problem in complex geometry, in which the existence was proved for flat manifolds \cite{Yau} and for manifolds with negative curvature \cite{Aubin}. The positive curvature part, the so-called Fano manifolds, do not always admit K\"ahler--Einstein metrics, a phenomenon that lead to the notion of K-stability and finally in the resolution of the Yau--Tian--Donaldson conjecture: a complex Fano manifold admits a K\"ahler--Einstein metric if and only if it is K-polystable \cite{CDS-1,CDS-2,CDS-3}.
Another well-studied problem in algebraic geometry is the study of the existence of $\Bbbk$-rational points on (Fano) varieties defined over arbitrary fields $\Bbbk$.
The aim of this paper is to make a connection between K-stability and the geometry of Fano varieties defined over subfields of the complex numbers.

In dimension one, there is only one complex Fano manifold, $\mathbb{P}^1$, which admits a K\"ahler--Einstein metric and can be defined over $\mathbb{Q}$.
Two dimensional Fano manifolds are known as \emph{del Pezzo surfaces}.
They form $10$ deformation families: $\mathbb{P}^1\times\mathbb{P}^1$ and blow up of $\mathbb{P}^2$ in at most $8$ points in general position.
In \cite{Tian1990}, Tian proved that the only non-K\"ahler--Einstein (smooth) del Pezzo surfaces are
\begin{enumerate}
\item the blowup of $\mathbb{P}^2$ at one point, the first Hirzebruch surface denoted by $\mathbb{F}_1$, and
\item the blowup of $\mathbb{P}^2$ at two points, the del Pezzo surface of degree $7$ denoted by $S_7$.
\end{enumerate}
On the other hand, both surfaces $\mathbb{F}_1$ and $S_7$ can be defined over $\mathbb{Q}$.
Moreover, $\mathbb{F}_1$ has only one form over $\mathbb{Q}$, and every form of the surface $S_7$ has a rational point.
Therefore, if a (smooth) del Pezzo surface defined over $\Bbbk\subset\mathbb{C}$ does not have $\Bbbk$-rational points,
then its geometric model admits a K\"ahler--Einstein metric.
In this paper, we show that a similar result also holds for del Pezzo orbifolds, i.e., for del Pezzo surfaces with at most quotient singularities.

\begin{mainthm}
\label{theorem:1}
Let $S$ be a del Pezzo surface with quotient singularities defined over a subfield $\Bbbk$ of $\mathbb{C}$. Assume the geometric model of $S$ does not admit an orbifold K\"ahler--Einstein metric. Then $S$ has a $\Bbbk$-rational point.
\end{mainthm}

However, this result does not mean that every non-K\"ahler--Einstein del Pezzo orbifold has a smooth $\Bbbk$-rational point. For example, consider $S=\{x_1^2+x_2^2+x_3^2=0\}\subset\mathbb{P}^3$, a degree $8$ del Pezzo surface over $\mathbb{Q}$, for which $S_{\mathbb{C}}$ is K-unstable and $S(\mathbb{Q})$ consists of only $[0:0:0:1]$, the unique singular point of $S$.

In conclusion, two-dimensional Fano orbifolds defined over a subfield $\Bbbk\subset\mathbb{C}$ whose geometric models are non-K\"ahler--Einstein always have $\Bbbk$-rational points.
In higher-dimensions, this phenomenon has a more complicated nature even for three-dimensional Fano manifolds (Fano 3-folds).
For instance, let $X$ be a Fano 3-fold defined over a subfield $\Bbbk\subset\mathbb{C}$.
Even if its geometric model $X_{\mathbb{C}}$ does not admit a K\"ahler--Einstein metric, we cannot always conclude that $X$ has a $\Bbbk$-rational point.
Indeed, if $X=C\times\mathbb{F}_1$ or $C\times S_7$, where $C$ is a pointless conic (that is, $C(\Bbbk)=\varnothing$ and $C_{\mathbb{C}}\cong \mathbb{P}^1$), then $X_{\mathbb{C}}$ is not K\"ahler--Einstein while $X$ is pointless.
Surprisingly, there are not many other exceptions as indicated in our second result:

\begin{mainthm}
\label{theorem:2}
Let $X$ be a smooth Fano 3-fold defined over a subfield $\Bbbk\subset\mathbb{C}$ such that its geometric model is not K\"ahler--Einstein. Then $X$ has a $\Bbbk$-rational point unless
$X=C\times\mathbb{F}_1$ or $X=C\times S$ for a pointless conic $C$ and a ${\Bbbk}$-form $S$ of $S_7$, or the 3-fold $X$ is one of the following:
\begin{enumerate}
\item[$(\mathrm{1})$] the blowup of a pointless quadric in $\mathbb{P}^4$ along a  quartic elliptic curve;

\item[$(\mathrm{2})$] the blowup of a pointless quadric cone in $\mathbb{P}^4$ at its vertex;

\item[$(\mathrm{3})$] the blowup of the product $\mathbb{P}^1\times Q$ along a curve $C$ such that $\pi_1(C)$ is a point, and $\pi_2(C)$ is a conic, where $Q$ is a pointless quadric in $\mathbb{P}^3$,
and $\pi_i$ is the projection to the $i$-th factor;

\item[$(\mathrm{4})$] the blowup of a pointless $\Bbbk$-form of $\mathbb{P}^3$ along a curve of anticanonical degree $4$;

\item[$(\mathrm{5})$] the blowup of a Fano 3-fold described in $(\mathrm{4})$ along a curve of anticanonical degree $2$;

\item[$(\mathrm{6})$] the blowup of a Fano 3-fold described in $(\mathrm{4})$ along a disjoint union of a curve of anticanonical degree $2$ and a geometrically irreducible curve of anticanonical degree $4$.
\end{enumerate}
\end{mainthm}

The geometric models of the Fano 3-folds described in the eight exceptional cases in Theorem~\ref{theorem:2} are not K\"ahler--Einstein.
Moreover, all of them are K-unstable, so we have the following consequence.

\begin{corollary}
\label{corollary:semistable}
If $X$ is a smooth Fano 3-fold defined over a subfield $\Bbbk\subset\mathbb{C}$
such that its geometric model is strictly K-semistable, then $X(\Bbbk)\ne\varnothing$.
\end{corollary}

Note also that pointless forms of the eight exceptional cases in Theorem~\ref{theorem:2} exist over many subfields of $\mathbb{C}$ as, for instance, one can construct relevant examples over $\mathbb{R}$.
Such constructions are easy to obtain in the cases $(\mathrm{1})$, $(\mathrm{2})$, $(\mathrm{3})$. The example below provides pointless constructions in the remaining three cases.

\begin{example}
\label{example:Severi--Brauer}
Let $U$ be a three-dimensional Severi--Brauer variety defined over $\mathbb{R}$ with $U\not\simeq\mathbb{P}^3_{\mathbb{R}}$.
By \cite{GilleSzamuely,Kollar2016}, $U$ exists uniquely and has no real points.
Crucially, $U$ contains a zero-dimensional irreducible subscheme $Z$ of degree $2$ for which $Z_{\mathbb{C}}$ is a union of two complex conjugate points in $U_{\mathbb{C}}\simeq\mathbb{P}^3$. It follows that $U$ also contains a unique  curve $L$ of anticanonical degree $4$ containing $Z$.
In \cite{Kollar2016}, the curve $L$ is called \emph{twisted line} as $L_{\mathbb{C}}$ is a line in $U_{\mathbb{C}}$ containing both points of $Z_{\mathbb{C}}$.
Now, let $f\colon X\to U$ be the blowup of the curve $L$, and let $E$ be the $f$-exceptional divisor. Then $X$ is a pointless real Fano 3-fold described in $(\mathrm{4})$ in Theorem~\ref{theorem:2}, and $E\simeq L\times L$. Next, set $C=f^{-1}(Z)$. Then $C$ is an irreducible geometrically reducible smooth curve in $X$ with $-K_X\cdot C=2$. Moreover, every curve of anticanonical degree $2$ in $X$ can be described in this way. By blowing up $X$ along the curve $C$ we obtain a pointless real Fano 3-fold, which corresponds to $(\mathrm{5})$ in Theorem~\ref{theorem:2}.
Finally, let $L^\prime$ be another \emph{twisted line} in $U$, and let $C^\prime$ be its strict transform on $X$.
Then $L\cap L^\prime=\varnothing$ so that $-K_X\cdot C^\prime=4$. Furthermore, every geometrically irreducible curve of anticanonical degree $4$ in $X$ is a strict transform of a \emph{twisted line} in $U$ that is disjoint from $L$.
By blowing up $X$ along the curves $C$ and $C^\prime$, we obtain a pointless real Fano 3-fold described in $(\mathrm{6})$ in Theorem~\ref{theorem:2}.
\end{example}

\begin{remark}
\label{remark:applications}
Theorem~\ref{theorem:2} can be used to explicitly produce (new) examples of K\"ahler--Einstein Fano 3-folds.
For instance, real pointless smooth Fano 3-folds of Picard rank $1$ and anticanonical degree $22$ are classified in \cite{Kollar-Olaf},
and all of them are K-polystable by Theorem~\ref{theorem:2}.
This provides new K\"ahler--Einstein Fano 3-folds in Family 1.10.
For another example of a K-stable smooth Fano 3-fold in this deformation family, see \cite{CheltsovShramov2012}.
Note that also that Family 1.10 contains non-K\"ahler--Einstein smooth Fano 3-folds that can be defined over $\mathbb{R}$ \cite{Book,Tian1997,Donaldson2008}.
\end{remark}

\begin{proof}[\bf Proof of Theorem~\ref{theorem:2}]
Unlike the proof of Theorem~\ref{theorem:1}, the proof of Theorem~\ref{theorem:2} is somewhat classification based.
Recall that smooth Fano 3-folds have been classified into 105 deformation families by Iskovskikh \cite{Is77,Is78,IsPr99} and Mori--Mukai \cite{MoMu81,MoMu03,MoMu83,MoMu85,Mukai89,Mukai}.
We follow the Mori--Mukai numbering of the 105 families, written as ``Family \textnumero $m.n$'', in which $m$ is the rank of the Picard group of the 3-fold, ranging from 1 to 10, and $n$ is simply a list number.
If $X$ is a smooth Fano 3-fold defined over a subfield $\Bbbk\subset\mathbb{C}$,
then we will say that $X$ is in a given family if $X_{\mathbb{C}}$ is contained in this family.
To prove Theorem~\ref{theorem:2}, we partition 105 deformation families of smooth Fano 3-folds into the following three sets:
\begin{enumerate}[(i)]
\item $52$ families in which all smooth elements are known to be K\"ahler--Einstein;
\item $27$ families where all smooth members are non-K\"ahler--Einstein;
\item $26$ families in which only general members are known to be K\"ahler--Einstein.
\end{enumerate}
It is often a difficult task to verify whether a given Fano variety is K\"ahler--Einstein.
Recent invention of K-stability methods have enabled such studies although, as one expects,
explicit K-stability verification requires a detailed study of the geometry of the given Fano variety.
With much effort in recent years it has been verified that all smooth Fano 3-folds in the following $52$ deformation families are K\"ahler--Einstein:
\begin{itemize}
\item Families \textnumero 1.1, \textnumero  1.2, \textnumero 1.3, \textnumero 1.4, \textnumero 1.5, \textnumero 1.6, \textnumero 1.7, \textnumero 1.8 \cite[Theorem\,5.1]{AbbanZhuangSeshadri};

\item Families \textnumero 1.11, \textnumero 1.12, \textnumero 1.13, \textnumero 1.14, \textnumero 1.15, \textnumero 1.16,
\textnumero 1.17, \textnumero 2.25, \textnumero 2.27, \textnumero 2.29, \textnumero 2.32, \textnumero 2.34, \textnumero 3.1, \textnumero 3.9, \textnumero 3.15,
\textnumero 3.17, \textnumero 3.19, \textnumero 3.20, \textnumero 3.25, \textnumero 3.27, \textnumero 4.2, \textnumero 4.3, \textnumero 4.4,
\textnumero 4.6, \textnumero 4.7, \textnumero 5.1, \textnumero 5.3, \textnumero 6.1, \textnumero 7.1, \textnumero 8.1, \textnumero 9.1, \textnumero 10.1 \cite{Book};

\item Families \textnumero 2.1, \textnumero 2.2, \textnumero 2.3, \textnumero 2.4, \textnumero 2.6, \textnumero 2.7 \cite{CheltsovDenisovaFujita};

\item Family \textnumero 2.8 \cite{Liu};

\item Family \textnumero 2.15 \cite{LTN};

\item Families \textnumero 2.18 and \textnumero 3.4 \cite{CheltsovFujitaKishimotoPark};

\item Family \textnumero 3.3 \cite{CheltsovFujitaKishimotoOkada};

\item Family \textnumero 4.1 \cite{BelousovLoginov}.
\end{itemize}
These families are irrelevant for the proof of Theorem~\ref{theorem:2} --- we listed them with appropriate referencing for completeness of exposition.
Similarly, the following is the list of 27 Fano 3-fold families without any K\"ahler--Einstein smooth members:
\textnumero 2.23, \textnumero 2.26, \textnumero 2.28, \textnumero 2.30, \textnumero 2.31, \textnumero 2.33, \textnumero 2.35, \textnumero 2.36, \textnumero 3.14,
\textnumero 3.16, \textnumero 3.18, \textnumero 3.21, \textnumero 3.22, \textnumero 3.23, \textnumero 3.24, \textnumero 3.26, \textnumero 3.28, \textnumero 3.29,
\textnumero 3.30, \textnumero 3.31, \textnumero 4.5,  \textnumero 4.8,  \textnumero 4.9,  \textnumero 4.10, \textnumero 4.11, \textnumero 4.12, \textnumero 5.2 \cite{Book,CheltsovPrzyjalkowskiShramov,Fujita2019}.
In Section~\ref{Section:k-point} we show that for each element in $19$ of these families every member has $\Bbbk$-points when defined over $\Bbbk$,
hence producing the $8$ families appearing in Theorem~\ref{theorem:2} as exceptional cases. The varieties appearing in Theorem~\ref{theorem:2} are the only pointless members in each of those $8$ families. This is clear in (1), (2), (3), and for $C\times\mathbb{F}_1$ and $C\times S_7$. Other cases are also easy to see. For example, consider a pointless variety $X$ for which $X_{\mathbb{C}}$ is the blowup of $\mathbb{P}^3$ in a line. 
Then $X$ is the blowup of a non-trivial $\Bbbk$-form $U$ of $\mathbb{P}^3$ along a curve of anticanonical degree 4 so that its geometric model is a line in $\mathbb{P}^3$, as in Example~\ref{example:Severi--Brauer}.
The remaining $26$ families are Families \textnumero 1.9,  \textnumero 1.10, \textnumero 2.5,  \textnumero 2.9,  \textnumero 2.10, \textnumero 2.11, \textnumero 2.12, \textnumero 2.13, \textnumero 2.14,
\textnumero 2.16, \textnumero 2.17, \textnumero 2.19, \textnumero 2.20, \textnumero 2.21, \textnumero 2.22, \textnumero 2.24, \textnumero 3.2,  \textnumero 3.5,  \textnumero 3.6,  \textnumero 3.7,  \textnumero 3.8,  \textnumero 3.10, \textnumero 3.11, \textnumero 3.12, \textnumero 3.13, \textnumero 4.13.
Among these we treat $8$ families (Families \textnumero 2.9, \textnumero 2.11, \textnumero 2.14, \textnumero 2.17, \textnumero 2.20, \textnumero 2.22,
\textnumero 3.8, \textnumero 3.11) in Section~\ref{Section:k-point} by showing that every member has $\Bbbk$-points when defined over $\Bbbk$,
and for the remaining $18$ families we prove in Section~\ref{section:pointless-3-folds} that the $\Bbbk$-pointless elements are K-polystable. As the reader expects, the latter is the main bulk of the proof. Note that some of these families contain smooth complex non-K\"ahler--Einstein Fano 3-folds.
\end{proof}

\paragraph{\textbf{Structure of the paper}} Section~\ref{section:preliminaries} contains some preliminary technical results that we will use in the article. In Section\,\ref{section:dP-surfaces} we prove Theorem\,\ref{theorem:1}.
In Section\,\ref{Section:k-point} we prove existence of $\Bbbk$-rational points for all smooth members in 27 families of Fano 3-folds as explained above, which includes 8 families containing K-polystable objects where K-(poly)stability is unknown (or sometimes known not to hold) for all smooth elements. There remain 19 other families with that property, and in Section\,\ref{section:pointless-3-folds} we prove that any smooth elements in those 19 families for which there exists a $\Bbbk$-form with no $\Bbbk$-points is K-polystable. In Section\,\ref{section:pointless} we produce pointless examples for each of those 19 families, to illustrate the relevance of the proof of Theorem\,\ref{theorem:2}.

 \section{Preliminaries}
\label{section:preliminaries}
In this section, we collect some known technical results that we will be using in this article. The reader may skip them and only consult them as they are referred to.

We first state the following classical result which is valid over any field $\Bbbk$. We will use this theorem  repeatedly throughout the article.

\begin{lemma}[Lang--Nishimura Lemma]
\label{lemma:nonempty}
Let $V$ and $W$ be projective integral varieties defined over a field $\Bbbk$ such that there exists a rational map $V \dasharrow W$.
If $V$ admits a smooth ${\Bbbk}$-rational point, then $W({\Bbbk}) \ne\varnothing$.
\end{lemma}

\begin{proof}
See, for example, \cite[Theorem 3.6.11]{Poonen}.
\end{proof}

The next two elementary lemmas about Severi--Brauer varieties will be used frequently.

\begin{lemma}
\label{lemma:SB}
Let $U$ be a Severi--Brauer variety of dimension $n$ over $\Bbbk$. If $U$ contains a divisor defined over $\Bbbk$ whose degree is coprime to $n+1$, then $U$ is isomorphic to $\mathbb{P}_{\Bbbk}^n$.
\end{lemma}

\begin{proof}
This is well known to experts. See, for example, \cite[Theorem 5.1.3]{GilleSzamuely} or \cite{Kollar2016}.
\end{proof}

\begin{corollary}[{cf. \cite{Alvaro}}]
\label{corollary:ICMS}
Let $X$ be a variety of dimension $\geqslant 2$ such that $X_{\mathbb{C}}$ is a hypersurface in $\mathbb{P}^n$ of degree $d$ such that 
$(n,d)\ne (2,3)$, $(n,d)\ne (3,4)$, and $n+1$ and $d$ are coprime. 
Then $X$ is a hypersurface in $\mathbb{P}^n$ of degree $d$.
\end{corollary}

\begin{proof}
It follows from \cite[Chapter~7]{Collio} that $X$ can be embedded in a $\Bbbk$-form of $\mathbb{P}^n$ as twisted hypersurface of degree $d$. 
Hence, it follows from Lemma~\ref{lemma:SB} that this $\Bbbk$-form of $\mathbb{P}^n$ is actually isomorphic to $\mathbb{P}^n$, so the result follows.
\end{proof}

\begin{lemma}
\label{lemma:plane}
Let $U$ be a Severi--Brauer variety of dimension three over $\Bbbk$ and $C$ an irreducible curve in $U$ such that
$C_{\mathbb{C}}$ is contained in a plane in $U_{\mathbb{C}}\simeq\mathbb{P}^3$, but $C_{\mathbb{C}}$ is not a line.
Then $U\cong\mathbb{P}^3$.
\end{lemma}

\begin{proof}
Let $H$ be the plane in $U_{\mathbb{C}}$ that contains $C_{\mathbb{C}}$.
Then $H$ is defined over $\Bbbk$ as otherwise there will be at least two planes containing $C_{\mathbb{C}}$, which implies
that $C_{\mathbb{C}}$ is a line. Hence $U\cong \mathbb{P}^3$ by Lemma~\ref{lemma:SB}.
\end{proof}

We will also need the following classification based result.
\begin{lemma}
\label{lemma:Prokhorov}
Let $X$ be a smooth Fano 3-fold defined over $\Bbbk$ with base extension $X_{{\mathbb C}}$ being of Picard rank $\rho(X_{\mathbb{C}})=2$ and not contained in the families \textnumero 2.12 and \textnumero 2.21. If $X_{{\mathbb C}}$ admits an extremal birational contraction $\pi\colon X_{\mathbb{C}}\to V$, then there exists a morphism $p\colon X \to W$ defined over $\Bbbk$ such that the base extension $p_{{\mathbb C}}\colon X_{{\mathbb C}} \to W_{\mathbb C}$ coincides with $\pi$.
\end{lemma}
\begin{proof}
The required assertion is well-known. See for example \cite[Theorem~1.2]{Prokhorov2013}.
If $\rho(X)=1$, then two extremal rays of the Mori cone $\mathrm{NE}(X_{\mathbb{C}})$ would be permuted by the Galois group $\mathrm{Gal}({\mathbb{C}}/{\Bbbk})$. However, it follows from the description of these extremal rays \cite{MoMu81}
that this is impossible, unless if $X_{\mathbb{C}}$ is in Family \textnumero 2.12 or \textnumero 2.21.
Hence, $\rho_{\Bbbk}(X)=2$, so both contractions associated to the extremal rays of $\overline{\mathrm{NE}}(X_{\mathbb{C}})$ are defined over ${\Bbbk}$.
This completes the proof.
\end{proof}

We now turn our attention to some certain stability threshold type invariants that allow estimations that prove K-(poly)stability.

Let $X$ be a smooth Fano 3-fold, and let $S$ be an irreducible smooth surface in $X$.
Set
$$
\tau=\mathrm{sup}\Big\{u\in\mathbb{R}_{\geqslant 0}\ \big\vert\ \text{the divisor  $-K_X-uS$ is pseudo-effective}\Big\}.
$$
For~$u\in[0,\tau]$, let $P(u)$ be the positive part of the Zariski decomposition of the divisor $-K_X-uS$,
and let $N(u)$ be its negative part. Set
$$
S_X(S)=\frac{1}{-K_X^3}\int_{0}^{\infty}\mathrm{vol}\big(-K_X-uS\big)du=\frac{1}{-K_X^3}\int_{0}^{\tau}\big(P(u)\big)^3du.
$$
For every prime divisor $F$ over $S$, following \cite{Book}, we set
\begin{multline*}
\quad \quad \quad \quad S\big(W^S_{\bullet,\bullet};F\big)=\frac{3}{(-K_X)^3}\int_0^\tau\big(P(u)\cdot P(u)\cdot S\big)\cdot\mathrm{ord}_{F}\big(N(u)\big\vert_{S}\big)du+\\
+\frac{3}{(-K_X)^3}\int_0^\tau\int_0^\infty \mathrm{vol}\big(P(u)\big\vert_{S}-vF\big)dvdu.\quad \quad \quad \quad
\end{multline*}

\begin{theorem}[{\cite{AbbanZhuang}}]
\label{theorem:Hamid-Ziquan-Kento}
For any point $p\in S$ we have
$$
\delta_p(X)\geqslant\min\Bigg\{\frac{1}{S_X(S)},\inf_{\substack{F/S\\ p\in C_S(F)}}\frac{A_S(F)}{S\big(W^S_{\bullet,\bullet};F\big)}\Bigg\},
$$
where the infimum is taken by all prime divisors over $S$ whose center on $S$ contains $p$.
\end{theorem}

This theorem can be used to show that $\delta_p(X)\geqslant 1$.
However, if $S(W^S_{\bullet,\bullet};F)>A_S(F)$ for at least one prime divisor $F$ over the surface $S$ with $p\in C_F(S)$, then
we cannot use Theorem~\ref{theorem:Hamid-Ziquan-Kento} to prove that $\delta_p(X)\geqslant 1$.
In this case, we use a similar approach to estimate the $\delta$-invariant for prime divisors over $X$
whose centers on $X$ are curves. To do this, let $C$ be an irreducible curve in $S$. Write
$$
N(u)\big\vert_{S}=N^\prime(u)+\mathrm{ord}_{C}\big(N(u)\big\vert_{S}\big)C,
$$
so $N^\prime(u)$ is an effective $\mathbb{R}$-divisor on $S$ whose support does not contain $C$.
For $u\in[0,\tau]$, let
$$
t(u)=\mathrm{sup}\Big\{v\in\mathbb{R}_{\geqslant 0}\ \big\vert\ \text{the divisor  $P(u)\big\vert_{S}-vC$ is pseudo-effective}\Big\}.
$$
For~$v\in[0,t(u)]$, we let $P(u,v)$ be the positive part of the Zariski decomposition of~$P(u)\vert_{S}-vC$,
and we let $N(u,v)$ be the negative part  of the Zariski decomposition of~$P(u)\vert_{S}-vC$. Then
\begin{multline*}
\quad \quad \quad \quad S\big(W^S_{\bullet,\bullet};C\big)=\frac{3}{(-K_X)^3}\int_0^\tau\big(P(u)\cdot P(u)\cdot S\big)\cdot\mathrm{ord}_{C}\big(N(u)\big\vert_{S}\big)du+\\
+\frac{3}{(-K_X)^3}\int_0^\tau\int_0^{t(u)}\big(P(u,v)\big)^2dvdu.\quad \quad \quad \quad
\end{multline*}

\begin{theorem}[{\cite{AbbanZhuang,Book}}]
\label{theorem:Hamid-Ziquan-Kento-1}
Let $E$ be a prime divisor over $X$ such that $C_X(E)=C$. Then
$$
\frac{A_X(E)}{S_X(E)}\geqslant\min\Bigg\{\frac{1}{S_X(S)},\frac{1}{S\big(W^S_{\bullet,\bullet};C\big)}\Bigg\}.
$$
In particular, if $S_X(S)<1$ and $S(W^S_{\bullet,\bullet};C)<1$, then $\beta(E)>0$.
\end{theorem}

Now, we suppose, in addition, that $C$ is smooth and $p\in C$. Then, following \cite{AbbanZhuang,Book}, we let
$$
F_p\big(W_{\bullet,\bullet,\bullet}^{S,C}\big)=\frac{6}{(-K_X)^3} \int_0^\tau\int_0^{t(u)}\big(P(u,v)\cdot C\big)\cdot \mathrm{ord}_p\big(N^\prime(u)\big|_C+N(u,v)\big|_C\big)dvdu
$$
and
$$
S\big(W_{\bullet, \bullet,\bullet}^{S,C};p\big)=\frac{3}{(-K_X)^3}\int_0^\tau\int_0^{t(u)}\big(P(u,v)\cdot C\big)^2dvdu+F_p\big(W_{\bullet,\bullet,\bullet}^{S,C}\big).
$$
We have the following estimate:

\begin{theorem}[{\cite{AbbanZhuang,Book}}]
\label{theorem:Hamid-Ziquan-Kento-2}
One has
$$
\delta_p(X)\geqslant\min\Bigg\{\frac{1}{S_X(S)},\frac{1}{S\big(W^S_{\bullet,\bullet};C\big)},\frac{1}{S\big(W_{\bullet, \bullet,\bullet}^{S,C};p\big)}\Bigg\}.
$$
In particular, if $S_X(S)<1$, $S(W^S_{\bullet,\bullet};C)<1$ and $S(W_{\bullet, \bullet,\bullet}^{S,C};p)<1$, then $\delta_p(X)>1$.
\end{theorem}

Now, let $f\colon\widetilde{S}\to S$ be a blowup of the surface $S$ at the point $p$, let $E$ be the $f$-exceptional curve,
and let $\widetilde{N}^\prime(u)$ be the proper transform on $\widetilde{S}$ of the divisor $N(u)\vert_{S}$.
For $u\in[0,\tau]$, we let
$$
\widetilde{t}(u)=\sup\Big\{v\in \mathbb{R}_{\geqslant 0} \ \big| \ f^*\big(P(u)|_S\big)-vE \text{ is pseudo-effective}\Big\}.
$$
For $v\in [0,\widetilde{t}(u)]$, let $\widetilde{P}(u,v)$ be the positive part of the Zariski decomposition of $f^*(P(u)|_S)-vE$,
and let $\widetilde{N}(u,v)$ be the negative part of this~Zariski decomposition. Set
$$
S\big(W_{\bullet,\bullet}^{S};E\big)=\frac{3}{(-K_X)^3}\int_0^\tau\big(P(u)\cdot P(u)\cdot S\big)\cdot\mathrm{ord}_{E}\big(f^*(N(u)\vert_{S})\big)du+\frac{3}{(-K_X)^3}\int_0^\tau\int_0^{\widetilde{t}(u)}\big(\widetilde{P}(u,v)\big)^2dvdu.\quad \quad \quad \quad
$$
Finally, for every point $q\in E$, we set
$$
F_q\big(W_{\bullet,\bullet,\bullet}^{S,E}\big)=
\frac{6}{(-K_X)^3}\int_0^\tau\int_0^{\widetilde{t}(u)}\big(\widetilde{P}(u,v)\cdot E\big)\times\mathrm{ord}_q\big(\widetilde{N}^\prime(u)\big|_E+\widetilde{N}(u,v)\big|_{E}\big)dvdu
$$
and
$$
S\big(W_{\bullet, \bullet,\bullet}^{S,E};q\big)=
\frac{3}{(-K_X)^3}\int_0^\tau\int_0^{\widetilde{t}(u)}\big(\widetilde{P}(u,v)\cdot E\big)^2dvdu+
F_q\big(W_{\bullet,\bullet,\bullet}^{S,E}\big).
$$
If $p\not\in\mathrm{Supp}(N(u))$ for every $u\in[0,\tau]$, the formulae for $S(W_{\bullet, \bullet}^{S};E)$ and
$F_q(W_{\bullet,\bullet,\bullet}^{S,E})$ simplify as
\begin{align*}
S\big(W_{\bullet,\bullet}^{S};E\big)&=\frac{3}{(-K_X)^3}\int_0^\tau\int_0^{\widetilde{t}(u)}\big(\widetilde{P}(u,v)\big)^2dvdu,\\
F_q\big(W_{\bullet,\bullet,\bullet}^{S,E}\big)&=\frac{6}{(-K_X)^3}\int_0^\tau\int_0^{\widetilde{t}(u)}\big(\widetilde{P}(u,v)\cdot E\big)\times\mathrm{ord}_q\big(\widetilde{N}(u,v)\big|_{E}\big)dvdu.
\end{align*}
Moreover, Theorem~\ref{theorem:Hamid-Ziquan-Kento-2} can be generalized as follows:

\begin{theorem}[{\cite{AbbanZhuang,Book}}]
\label{theorem:Hamid-Ziquan-Kento-3}
One has
$$
\delta_p(X)\geqslant\min\Bigg\{\frac{1}{S_X(S)},\frac{2}{S\big(W^S_{\bullet,\bullet};E\big)},\inf_{q\in E}\frac{1}{S\big(W_{\bullet, \bullet,\bullet}^{S,E};q\big)}\Bigg\}.
$$
\end{theorem}

\section{Singular del Pezzo surfaces}
\label{section:dP-surfaces}

The overall goal of this section is to prove Theorem~\ref{theorem:1}.
We first gather some technical results about birational invariants of del Pezzo surfaces that will be used in the proof.

\subsection{On $\alpha$-invariants of del Pezzo surfaces}
\label{section:delPezzos}

Let $\Bbbk$ be a subfield of $\mathbb{C}$,
and let $S$ be a del Pezzo surface defined over $\Bbbk$ with quotient singularities.
Recall that
$$
\alpha(S)=\mathrm{sup}\left\{\lambda\in\mathbb{R}_{\geqslant 0}\ \left|\ \aligned
&\text{the log pair}\ \left(S, \lambda D\right)\ \text{has log canonical singularities for}\\
&\text{every effective $\mathbb{Q}$-divisor $D$ on $S$ such that}\ D\sim_{\mathbb{Q}} -K_{S}\\
\endaligned\right.\right\}.
$$

\begin{lemma}
\label{lemma:alpha}
Suppose that $S(\Bbbk)=\varnothing$ and $\alpha(S)<1$.
Then $\alpha(S)\in\mathbb{Q}_{>0}$, and the surface $S$ contains a smooth geometrically irreducible and geometrically rational curve $C$ such that
$$
\frac{1}{\alpha(S)}=\mathrm{sup}\big\{u\in\mathbb{R}_{\geqslant 0}\ \big\vert\ \text{the divisor}\ -K_S-uC\ \text{is pseudo-effective}\big\},
$$
and $-K_S\sim_{\mathbb{Q}}\frac{1}{\alpha(S)}C+\Delta$,
where $\Delta$ is an effective $\mathbb{Q}$-divisor on $S$ such that
$(S,C+\alpha(S)\Delta)$ has purely log terminal singularities.
\end{lemma}

\begin{proof}
Arguing as in the proof of \cite[Proposition~3.4]{Birkar2022}, we see that
\begin{equation}
\label{equation:Birkar}
\alpha(S)=\mathrm{lct}(S,D)
\end{equation}
for some effective $\mathbb{Q}$-divisor $D\sim_{\mathbb{Q}} -K_S$ on the surface $S$.
One can~deduce \eqref{equation:Birkar} without using the proof of \cite[Proposition~3.4]{Birkar2022}, since
it follows from Koll\'ar--Shokurov connectedness theorem \cite[Corollary~5.49]{KollarMori1998} that
$$
\alpha(S)=\mathrm{inf}\Big\{\frac{1}{\tau(C)}\ \big\vert\ C\ \text{is an irreducible curve in}\ S\Big\},
$$
where $\tau(C)=\mathrm{sup} \big\{u\in\mathbb{R}_{\geqslant 0}\ \big\vert\  \text{the divisor}\ -K_S-uC\ \text{is pseudo-effective}\big\}$.

The log pair $(S,\alpha(S)D)$ has log canonical singularities, but it is not klt (Kawamata log terminal), hence the locus $\mathrm{Nklt}(S,\alpha(S)D)$ is not empty
and it is connected by Koll\'ar--Shokurov connectedness.
Since $S(\Bbbk)=\varnothing$, we conclude that  $\mathrm{Nklt}(S,\alpha(S)D)$ is a connected union of irreducible curves and, in particular, the surface $S$ contains a ${\Bbbk}$-irreducible curve $C$ with
$
D=\tau C+\Delta
$, where $\tau=\frac{1}{\alpha(S)}$ and $\Delta$ is an effective $\mathbb{Q}$-divisor whose support does not contain $C$. This gives
$$
\tau=\mathrm{sup}\big\{u\in\mathbb{R}_{\geqslant 0}\ \big\vert\  \text{the divisor}\ -K_S-uC\ \text{is pseudo-effective}\big\}.
$$

If $C$ is not a minimal log canonical center of the log pair $(S,\alpha(S)D)=(S,C+\alpha(S)\Delta)$,
then using Kawamata--Shokurov trick \cite[Lemma~2.4.10]{CheltsovShramov2015}, also known as \emph{tie breaking},
we can replace the divisor $D$ by an effective $\mathbb{Q}$-divisor
$$
D^\prime\sim_{\mathbb{Q}}(1+\epsilon)\big(-K_S\big)
$$
for some sufficiently small $\epsilon\in\mathbb{Q}_{>0}$ such that the log pair $(S,\alpha(S)D^\prime)$ has log canonical singularities
but $\mathrm{Nklt}(S,\alpha(S)D^\prime)$ consists of finitely many points.
Now, using Koll\'ar--Shokurov connectedness, we obtain a contradiction with the assumption that $S(\Bbbk)=\varnothing$.
Hence, the curve $C$ is a minimal log canonical center of the log pair $(S,C+\alpha(S)\Delta)$,
which implies that $C$ is smooth \cite{Kawamata1998}.

Now, using properties of log canonical centers \cite{Kawamata1997,Kawamata1998},
we conclude that $\mathrm{Nklt}(S,\alpha(S)D)=C$,
which implies that the curve $C$ is connected
and the log pair  $(S,\alpha(S)D)$ is purely log terminal.
Hence, the curve $C$ is geometrically irreducible.
Finally, using Kawamata's subadjunction theorem, we see that the curve $C$ is geometrically rational.
\end{proof}

\begin{corollary}
\label{corollary:alpha-conic-bundles}
Suppose that $S(\Bbbk)=\varnothing$, the rank of the Picard group of $S$ is~$2$,
and $\mathrm{NE}(S)$ is generated by irreducible curves $C_1$ and $C_2$
such that $C_1^2=C_2^2=0$ and $C_1\cdot C_2>0$.
Then $\alpha(S)\geqslant\frac{1}{2}$.
\end{corollary}

\begin{proof}
Suppose that $\alpha(S)<\frac{1}{2}$.
By Lemma~\ref{lemma:alpha},
there is a geometrically irreducible curve $C\subset S$ such that
$$
-K_S\sim_{\mathbb{Q}}\frac{1}{\alpha(S)}C+\Delta,
$$
where $\Delta$ is an effective $\mathbb{Q}$-divisor on $S$.
Without loss of generality, we may assume that $C\cdot C_1>0$.

Observe that $|nC_1|$ is base point free for $n\gg 0$.
Moreover, replacing $\Bbbk$ by its algebraic closure, we may assume that $|nC_1|$ gives a morphism $\pi\colon S\to\mathbb{P}^1$
such that its general fiber $F\simeq\mathbb{P}^1$. Then
$$
2=-K_S\cdot F=\frac{1}{\alpha(S)} C\cdot F+\Delta\cdot F\geqslant \frac{1}{\alpha(S)} C\cdot F\geqslant \frac{1}{\alpha(S)} >2,
$$
which is absurd.
\end{proof}

Let $C$ be a geometrically irreducible curve in $S$ and set
$$
\tau=\mathrm{sup}\big\{u\in\mathbb{R}_{\geqslant 0}\ \big\vert\  \text{the divisor}\ -K_S-uC\ \text{is pseudo-effective}\big\}.
$$
Then $\tau\in\mathbb{Q}_{>0}$ and
$
-K_{S}\sim_{\mathbb{Q}}\tau C+\Delta
$, where $\Delta$ is an effective $\mathbb{Q}$-divisor on $S$ with $C \not\subseteq \mathrm{Supp} (\Delta )$.
In particular, we see that $\alpha(S)\leqslant\frac{1}{\tau}$.
Set
$$
\beta(C)=1-\frac{1}{(-K_S)^2}\int\limits_{0}^{\tau}\mathrm{vol}\big(-K_S-uC\big)du.
$$

\begin{lemma}
\label{lemma:Kento}
Suppose that $\beta(C)\leqslant 0$. Then the following assertions hold:
\begin{itemize}
\item[$\mathrm{(1)}$] if $C^2>0$, then $\tau>2$;

\item[$\mathrm{(2)}$] if $C^2\geqslant 0$, then $\tau\geqslant 2$;

\item[$\mathrm{(3)}$] if $C^2=0$ and $\tau=2$, then
$
-K_S\sim_{\mathbb{Q}}2C+aZ
$, for $a\in\mathbb{Q}_{>0}$ and an irreducible curve $Z\subset S$ with $Z^2=0$.
\end{itemize}
\end{lemma}

\begin{proof}
All required assertions follow from \cite[Lemma~9.7]{Fujita2016}.
\end{proof}

Let $f\colon\widetilde{S}\to S$ be the minimal resolution of the del Pezzo surface $S$,
and denote by $\widetilde{C}$ the strict transform on $\widetilde{S}$ of the curve $C$.

\begin{lemma}
\label{lemma:negative-curves}
If $\widetilde{C}^2<0$ then $\widetilde{S}(\Bbbk)\ne\varnothing$ and $S(\Bbbk)\ne\varnothing$.
\end{lemma}

\begin{proof}
If $\widetilde{C}^2<0$, then it follows from the adjunction formula that $\widetilde{C}^2=-1$, and $\widetilde{C}$ is a $\Bbbk$-form of $\mathbb{P}^1$.
Since $-\widetilde{C}\vert_{\widetilde{C}}$ is a line bundle of degree $1$, we have $\widetilde{C}\cong \mathbb{P}^1$
which implies that $\widetilde{S}(\Bbbk)\ne\varnothing$.
\end{proof}

\begin{corollary}
\label{corollary:negative-curves}
If $S(\Bbbk)=\varnothing$ and $\beta(C)\leqslant 0$, then either $C^2>0$ or $C^2=0$ and $C\cap\mathrm{Sing}(S)=\varnothing$.
\end{corollary}

\begin{corollary}
\label{corollary:Kento}
Suppose that $S(\Bbbk)=\varnothing$ and $\beta(C)\leqslant 0$.
Then either $\tau>2$ or
$
S\simeq\mathbb{P}^1\times C_2
$
for a pointless conic $C_2\subset\mathbb{P}^2$.
\end{corollary}

\begin{proof}
Suppose that $\tau\leqslant 2$.
Then $C^2=0$ and $C\cap\mathrm{Sing}(S)=\varnothing$ by Lemma~\ref{lemma:Kento} and Corollary~\ref{corollary:negative-curves}.
Moreover, it follows from  Lemma~\ref{lemma:Kento} that $\tau=2$ and
$
-K_S\sim_{\mathbb{Q}}2C+aZ
$,
for some positive rational number $a$ and an irreducible curve $Z\subset S$ with $Z^2=0$.
By Riemann--Roch formula, the linear system $|C|$ is a pencil that gives a conic bundle $S\to\mathbb{P}^1$.
Since $S$ is a Mori Dream Space, the linear system $|nZ|$ is base point free for some positive integer $n$,
and it also gives a conic bundle $S\to C_2$ where $C_2$ is a conic defined over $\Bbbk$.
If $C_2(\Bbbk)\ne\varnothing$, that $C_2(\Bbbk)\cong\mathbb{P}^1$ , and we can replace $Z$ by a general fiber of the conic bundle $S\to C_2$.
Similarly, if $C_2(\Bbbk)=\varnothing$, we may assume that $Z$ is an irreducible geometrically reducible curve
that is a preimage of a general irreducible zero-dimensional subscheme of the conic $C_2$ of length $2$.
In both cases, we have $Z\cap\mathrm{Sing}(S)=\varnothing$.
Using adjunction formula, we get
$$
2C\cdot Z=\big(2C+aZ\big)\cdot Z=-K_S\cdot Z=
\left\{\aligned
&2\ \text{if $C_2(\Bbbk)\ne\varnothing$}, \\
&4\ \text{if $C_2(\Bbbk)=\varnothing$}.
\endaligned
\right.
$$
But $C\cdot Z\ne 1$, because  $S(\Bbbk)=\varnothing$.
Thus, we see that $C_2$ is a pointless conic and $C\cdot Z=2$.
Now, taking the product of the morphisms $S\to\mathbb{P}^1$ and $S\to C_2$,
we obtain the isomorphism $S\to\mathbb{P}^1\times C_2$.
\end{proof}

Let $S\to S^\prime$ be a birational morphism such that $S^\prime$ is normal.
Then $S^\prime$ is a del Pezzo surface with quotient singularities.
Applying Lemma~\ref{lemma:alpha} and Corollary~\ref{corollary:negative-curves}, we get the following result:

\begin{corollary}
\label{corollary:alpha}
The following assertions hold:
\begin{itemize}
\item[$\mathrm{(1)}$] if $S(\Bbbk)=\varnothing$, then $S^\prime(\Bbbk)=\varnothing$;

\item[$\mathrm{(2)}$] if $S(\Bbbk)=\varnothing$ and $\alpha(S)<1$, then $\alpha(S^\prime)\leqslant\alpha(S)$.
\end{itemize}
\end{corollary}

Note that we cannot always deduce that $\alpha(S^\prime)\leqslant\alpha(S)$ without using the condition $S(\Bbbk)=\varnothing$.
Indeed, if $S^\prime=\mathbb{P}^1\times\mathbb{P}^1$ and $S$ is a blow up of a point in $S^\prime$, then
$$
\frac{1}{3}=\alpha(S)<\alpha(S^\prime)=\frac{1}{2}.
$$

\subsection{Real del Pezzo surfaces; a warm up}
\label{subsection:real-dP}
To convey the ideas, we first prove Theorem \ref{theorem:1} for del Pezzo surfaces defined over the real numbers. We then proceed with the proof over other fields.

Let $S$ be a del Pezzo surface with quotient singularities defined over $\mathbb{R}$.

\begin{lemma}
\label{lemma:alpha-real}
Suppose that $\alpha(S)<\frac{1}{2}$. Then $S(\mathbb{R})\ne\varnothing$.
\end{lemma}

\begin{proof}
Set $\tau=\frac{1}{\alpha(S)}$ and suppose that $S(\mathbb{R})=\varnothing$. Then it follows from Lemma~\ref{lemma:alpha} that $S$ contains
a~geometrically irreducible curve $C$ with
$
-K_{S}\sim_{\mathbb{Q}}\tau C+\Delta
$
for some effective $\mathbb{Q}$-divisor $\Delta$ on $S$.

Let $f\colon\widetilde{S}\to S$ be the minimal resolution of the del Pezzo surface $S$ and denote by $\widetilde{C}$ and $\widetilde{\Delta}$, respectively, the strict transforms on $\widetilde{S}$ of the curve $C$
and the divisor $\Delta$. Then $\widetilde{S}(\mathbb{R})=\varnothing$  and
$
-K_{\widetilde{S}}\sim_{\mathbb{Q}}\tau\widetilde{C}+\widetilde{\Delta}+\widetilde{B}
$,
where $\widetilde{B}$ is an effective $\mathbb{Q}$-divisor on the surface $\widetilde{S}$ whose support consists of $f$-exceptional curves.
Now, applying Minimal Model Program to $\tilde{S}$, we obtain a birational morphism $h\colon\widetilde{S}\to\overline{S}$ such that
one of the following two cases holds:
\begin{itemize}
\item[$\mathrm{(1)}$] $\overline{S}$ is a smooth del Pezzo surface of Picard rank $1$;
\item[$\mathrm{(2)}$] $\overline{S}$ is a smooth surface of Picard rank $2$,
and there is a (standard) conic bundle $\pi\colon\overline{S}\to C_2$,
where $C_2$ is a geometrically irreducible conic in $\mathbb{P}^2$.
\end{itemize}
In both cases, we have $\overline{S}(\mathbb{R})=\varnothing$ by Lemma\,\ref{lemma:nonempty}.
Note that $\widetilde{C}$ is not $h$-exceptional, because $\widetilde{C}(\mathbb{R})=\varnothing$.
Set $\overline{C}=h(\widetilde{C})$ and let $\overline{\Delta}$ and $\overline{B}$ be the strict transforms on $\overline{S}$ of the divisors $\widetilde{\Delta}$ and $\widetilde{B}$, respectively.
Then
\begin{equation}
\label{equation:conic-bundle}
-K_{\overline{S}}\sim_{\mathbb{Q}}\tau\overline{C}+\overline{\Delta}+\overline{B}.
\end{equation}
Hence, if $\overline{S}$ is a smooth del Pezzo surface of Picard rank $1$, then it follows from \eqref{equation:conic-bundle} and $\tau>2$ that $\overline{S}$ is a Severi--Brauer surface and $\overline{C}$
is a \emph{twisted line} on it \cite{Kollar2016}, which implies that $\overline{S}\cong\mathbb{P}^2$,
which is a contradiction since $\overline{S}(\mathbb{R})=\varnothing$.

Thus, there is a conic bundle $\pi\colon\overline{S}\to C_2$.
Now, using \eqref{equation:conic-bundle} and intersecting $\tau\overline{C}+\overline{\Delta}+\overline{B}$ with a general fiber of the conic bundle~$\pi$,
we see that $\overline{C}$ is a fiber of $\pi$, because $\tau>2$.
Then $C_2\simeq\mathbb{P}^1$.
Now, using  $\rho(\overline{S})=2$ and $\overline{S}(\mathbb{R})=\varnothing$,
we see that $\overline{S}$ is a form of $\mathbb{F}_n$ for some $n\in\mathbb{Z}_{\geqslant 0}$, see~\cite{Kollar1997,Mangolte2020}.
Then \eqref{equation:conic-bundle} and $\tau>2$ give $n\ne 0$,
so the surface $\overline{S}$ contains the unique geometrically irreducible curve $\overline{Z}$ with $\overline{Z}^2=-n$.
Since $\overline{Z}\cdot\overline{C}=1$, we see that $\overline{Z}\cap\overline{C}$ consists of a single point in $\overline{S}(\mathbb{R})$,
which is a contradiction.
\end{proof}

\begin{corollary}
\label{corollary:alpha-real}
Suppose that $S(\mathbb{R})=\varnothing$. Then $S_{\mathbb{C}}$ is K-polystable.
\end{corollary}

\begin{proof}
Suppose that $S_{\mathbb{C}}$ is not K-polystable. Then it follows from \cite{Fujita2019,Li2017,Zhuang2021} that
$S$ contains a geometrically irreducible curve $C$ with $\beta(C)\leqslant 0$.
Moreover, using Corollary~\ref{corollary:Kento}, we see that $-K_{S}\sim_{\mathbb{Q}}\tau C+\Delta$
for some rational number $\tau>2$ and an effective $\mathbb{Q}$-divisor $\Delta$ on the surface $S$.
This gives $\alpha(S)\leqslant\frac{1}{\tau}<\frac{1}{2}$,
which contradicts Lemma~\ref{lemma:alpha-real}.
\end{proof}

\subsection{Del Pezzo surfaces of Picard rank one}
\label{subsection:rank-1}

Let $\Bbbk$ be a subfield of $\mathbb{C}$,
let $S$ be a del Pezzo surface with quotient singularities defined over $\Bbbk$,
and let $\rho(S)$ be the rank of the Picard group of the surface $S$.

\begin{lemma}
\label{lemma:alpha-rank-1}
Suppose that $\rho(S)=1$ and $\alpha(S)<\frac{1}{2}$. Then $S(\Bbbk)\ne\varnothing$.
\end{lemma}

\begin{proof}
Set $\tau=\frac{1}{\alpha(S)}$ and suppose that $S(\Bbbk)=\varnothing$.
By Lemma~\ref{lemma:alpha}, there exists a geometrically irreducible smooth and geometrically rational curve $C\subset S$ with
$
-K_{S}\sim_{\mathbb{Q}}\tau C,
$
and the log pair $(S,C)$ has purely log terminal singularities.
Let us seek for a contradiction.

Let $f\colon\widetilde{S}\to S$ be the minimal resolution of $S$
and let $\widetilde{C}$ be the strict transform on $\widetilde{S}$ of the curve $C$.
Then $\widetilde{S}(\Bbbk)=\varnothing$ and
$
-K_{\widetilde{S}}\sim_{\mathbb{Q}}\tau\widetilde{C}+\widetilde{B}
$,
where $\widetilde{B}$ is an effective $\mathbb{Q}$-divisor on the surface $\widetilde{S}$ whose support consists of $f$-exceptional curves.
Now, applying Minimal Model Program to $\widetilde{S}$, we obtain a birational morphism $h\colon\widetilde{S}\to\overline{S}$ such that
one of the following two cases holds:
\begin{itemize}
\item $\overline{S}$ is a smooth del Pezzo surface of Picard rank $1$;
\item $\overline{S}$ is a smooth surface of Picard rank $2$
and there is a (standard) conic bundle $\pi\colon\overline{S}\to C_2$,
where $C_2$ is a geometrically irreducible conic in $\mathbb{P}^2$.
\end{itemize}
Moreover, arguing as in the proof of Lemma~\ref{lemma:alpha-real}, we see that the former case is impossible,
which implies that $\overline{S}$ is a smooth surface of Picard rank $2$
and there exists a conic bundle $\pi\colon\overline{S}\to C_2$.
Note that $\overline{S}(\Bbbk)=\varnothing$ by Lemma\,\ref{lemma:nonempty}.

Set $\overline{C}=h(\widetilde{C})$, and let $\overline{B}$ be the strict transform of the divisor $\widetilde{B}$ on the surface $\overline{S}$ .
Then
\begin{equation}
\label{equation:conic-bundle-rank-1}
-K_{\overline{S}}\sim_{\mathbb{Q}}\tau\overline{C}+\overline{B}.
\end{equation}
Arguing as in the proof of Lemma~\ref{lemma:alpha-real},
we see that $C_2\simeq\mathbb{P}^1$ and $\overline{C}$ is a fiber of the conic bundle $\pi$.
Then
$\widetilde{C}^2\leqslant\overline{C}^2=0$,
which implies that $\widetilde{C}^2=0$ by Lemma~\ref{lemma:negative-curves}, because we know that $\widetilde{S}(\Bbbk)=\varnothing$.
Hence, we see that $h$ is an isomorphism in a neighborhood of the curve $\widetilde{C}$,
and the complete linear system $|\widetilde{C}|$ gives the composition morphism $\pi\circ h\colon\widetilde{S}\to C_2$.

Let $S_{\mathbb{C}}$ and $\widetilde{S}_{\mathbb{C}}$ be the models of the surfaces $S$ and $\widetilde{S}$ over
the algebraic closure $\mathbb{C}$ of the field~$\Bbbk$,
let $C_{\mathbb{C}}$ and $\widetilde{C}_{\mathbb{C}}$ be the curves in $S_{\mathbb{C}}$ and $\widetilde{S}_{\mathbb{C}}$
that correspond to  $C$ and $\widetilde{C}$, respectively.
Then $C_{\mathbb{C}}\simeq\widetilde{C}_{\mathbb{C}}\simeq\mathbb{P}^1$.
Note that $\widetilde{C}_{\mathbb{C}}^2=0<C_{\mathbb{C}}^2$, which implies that
$C_{\mathbb{C}}\cap\mathrm{Sing}\big(S_{\mathbb{C}}\big)\ne\varnothing$.
Since $(S_{\mathbb{C}},C_{\mathbb{C}})$ has purely log terminal singularities,
it follows from \cite{Kawamata1998,Prokhorov2001} that $C_{\mathbb{C}}$ contains at most three singular points of the surface $S_{\mathbb{C}}$,
and all these singular points are cyclic quotient singularities.
Thus, since $C(\Bbbk)=\varnothing$ and $C_{\mathbb{C}}\simeq\mathbb{P}^1$,
the curve $C_{\mathbb{C}}$ contains two singular points of the surface $S_{\mathbb{C}}$,
which are swapped by the action of $\mathrm{Gal}(\mathbb{C}/\Bbbk)$.

Let $P_1$ and $P_2$ be the singular points of the surface $S_{\mathbb{C}}$ contained in $C_{\mathbb{C}}$.
Then the exceptional curves of the minimal resolution $\widetilde{S}_{\mathbb{C}}\to S_{\mathbb{C}}$
that are mapped to the singular points $P_1$ and $P_2$ form two disjoint Hirzebruch--Jung strings,
which are swapped by the action of the group $\mathrm{Gal}(\mathbb{C}/\Bbbk)$.
Since $(S_{\mathbb{C}},C_{\mathbb{C}})$ is purely log terminal,
the curve $\widetilde{C}_{\mathbb{C}}$ intersects the first (or the last) curves of these strings,
which we denote by $E_1$ and $E_2$, respectively.
Set $E=E_1+E_2$.
Then $E$ is defined over $\Bbbk$, so we consider it as a curve in $\widetilde{S}$.
Then $E$ is the only $f$-exceptional curve that intersect the curve $\widetilde{C}$.
Hence, there exists the following commutative diagram:
$$
\xymatrix@R=1em{
&\widetilde{S}\ar@{->}[dl]_{g}\ar@{->}[dr]^{f}&\\%
\widehat{S}\ar@{->}[rr]^{q} && S}
$$
where $g$ is the contraction of all $f$-exceptional curves except for the curve $E$,
and $q$ is a partial resolution of singularities of the surface $S$ that contracts the strict transform of the curve $E$.

Let $\widehat{E}=g(E)$ and $\widehat{C}=g(\widetilde{C})$,
Then $-K_{\widehat{S}}\sim_{\mathbb{Q}}\tau\widehat{C}+a\widehat{E}$ for some  $a\in\mathbb{Q}_{>0}$.
On the other hand, we have $\widehat{C}^2=0$ and $\widehat{C}\cdot\widehat{E}=2$,
because $g$ is an isomorphism in a neighborhood of the curve $\widetilde{C}$.
Furthermore, we have $-K_{\widehat{S}}\cdot \widehat{C}=2$ by the adjunction formula.
This gives $a=1$, because
$$
2=-K_{\widehat{S}}\cdot \widehat{C}=\big(\tau\widehat{C}+a\widehat{E}\big)\cdot \widehat{C}=a\widehat{E}\cdot \widehat{C}=2a.
$$
Hence, since $\widehat{E}$ is smooth, the subadjunction formula applied to $\widehat{E}$ gives
$$
-4=\mathrm{deg}\big(K_{\widehat{E}}\big)\leqslant\big(K_{\widehat{S}}+\widehat{E}\big)\cdot \widehat{E}=-\tau\widehat{C}\cdot\widehat{E}=-2\tau<-4,
$$
which is a contradiction.
\end{proof}

\begin{corollary}
\label{corollary:alpha-rank-1}
Suppose that $\rho(S)=1$ and $S(\Bbbk)=\varnothing$.
Then $S_{\mathbb{C}}$ is K-polystable.
\end{corollary}

\begin{proof}
Suppose that $S_{\mathbb{C}}$ is not K-polystable. Then it follows from \cite{Li2017,Fujita2019,Zhuang2021} that
$S$ contains a geometrically irreducible curve $C$ such that $\beta(C)\leqslant 0$.
Then $-K_{S}\sim_{\mathbb{Q}}\tau C$
for some rational number $\tau\geqslant 3$, which implies that $\alpha(S)\leqslant\frac{1}{\tau}\leqslant\frac{1}{3}<\frac{1}{2}$,
but this contradicts Lemma~\ref{lemma:alpha-rank-1}.
\end{proof}

\subsection{The proof of Theorem~\ref{theorem:1}}
\label{subsection:proof}

Let $\Bbbk$ be a subfield of the field $\mathbb{C}$,
let $S$ be a del Pezzo surface with quotient singularities defined over $\Bbbk$,
and let $\rho(S)$ be the rank of the Picard group of the surface $S$.

\begin{lemma}
\label{lemma:alpha-big-rank}
Suppose that $\alpha(S)<\frac{1}{2}$. Then $S(\Bbbk)\ne\varnothing$.
\end{lemma}

\begin{proof}
Let us prove the assertion by induction on $\rho(S)$.
The case $\rho(S)=1$ is done by Lemma~\ref{lemma:alpha-rank-1}.
Suppose that $\rho(S)\geqslant 2$,
and the assertion holds for del Pezzo surfaces with smaller Picard rank.
We have to show that $S(\Bbbk)\ne\varnothing$.
Suppose that $S(\Bbbk)=\varnothing$.
Let us seek for a contradiction.

If there exists a non-biregular birational morphism $S\to S^\prime$ such that $S^\prime$ is a normal surface,
then $S^\prime$ is a del Pezzo surface with quotient singularities,
and it follows from Corollary~\ref{corollary:alpha} that $\alpha(S^\prime)\leqslant\alpha(S)<\frac{1}{2}$
and $S^\prime(\Bbbk)=\varnothing$, which contradicts the induction hypotheses.
Hence, we see that $S$ does not admit any non-biregular birational morphism to a normal surface.
This is only possible when $\rho(S)=2$, and  $\mathrm{NE}(S)$ is generated by irreducible curves $C_1$ and $C_2$
such that $C_1^2=C_2^2=0$ and $C_1\cdot C_2>0$.
But in this case we have $\alpha(S)\geqslant\frac{1}{2}$ by Corollary~\ref{corollary:alpha-conic-bundles}.
\end{proof}

Now, using Lemma~\ref{lemma:alpha-big-rank} and arguing as in the proof Corollary~\ref{corollary:alpha-rank-1}, we obtain Main Theorem.
Indeed, if $S(\Bbbk)=\varnothing$ and the surface $S_{\mathbb{C}}$ is not K-polystable, then it follows from \cite{Li2017,Fujita2019,Zhuang2021} that
$S$ contains a geometrically irreducible curve $C$ such that $\beta(C)\leqslant 0$,
so it follows from Corollary~\ref{corollary:Kento} that
$-K_{S}\sim_{\mathbb{Q}}\tau C+\Delta$ for some rational number $\tau>2$ and an effective $\mathbb{Q}$-divisor $\Delta$ on the surface $S$,
so $\alpha(S)\leqslant\frac{1}{\tau}<\frac{1}{2}$,
which contradicts Lemma~\ref{lemma:alpha-big-rank}.

\section{Smooth Fano 3-folds with $\Bbbk$-points}
\label{Section:k-point}

In this section, we prove existence of points for a number of Fano 3-folds defined over any subfield $\Bbbk$ of $\mathbb{C}$. 
As stated in the introduction, there are 26 families of smooth Fano 3-folds containing K-polystable members 
such that either they contain a non-K-polystable member or the K-stability picture for all smooth elements is lacking. Among them, we single out 8 families and prove in Subsection\,\ref{subsection:K-poly-with-points} that any members in those families, when defined over $\Bbbk$, contains $\Bbbk$-rational points. There are also 27 families of smooth Fano 3-folds where every smooth member is known to be non-K\"ahler-Einstein. In Subsection\,\ref{subsection:K-unstable-with-points} we show that every smooth member in 19 families (out of 27) always contain $\Bbbk$-rational points. This leaves 8 families that contain exceptional cases of Theorem\,\ref{theorem:2}.

\subsection{Families containing K-polystable members}
\label{subsection:K-poly-with-points}

\begin{lemma}
\label{lemma:2-9}
Suppose that $X$ is contained in Family \textnumero 2.9.
Then $X$ has a $\Bbbk$-point.
\end{lemma}

\begin{proof}
By Mori--Mukai \cite{MoMu81}, we have the following commutative diagram
$$
\xymatrix@R=1em{
&X_{\mathbb{C}}\ar@{->}[ld]_{f}\ar@{->}[rd]^{\pi}&\\%
\mathbb{P}^3\ar@{-->}[rr]&&\mathbb{P}^2}
$$
where $f$ is the blowup of a smooth curve $C$ of degree $7$ and genus $5$,
$\pi$ is a standard conic bundle with discriminant curve $\Delta\subset\mathbb{P}^2$ of degree $5$,
and the dashed arrow is given by the two-dimensional linear system of all cubic surfaces that contain the curve $C$.
Moreover, it follows from Lemma~\ref{lemma:Prokhorov} that this diagram can be defined over $\Bbbk$ with
$X_{\mathbb{C}}$ replaced by $X$,
$\mathbb{P}^3$ replaced by a $\Bbbk$-form $U$
of $\mathbb{P}^3$ and $\mathbb{P}^2$ replaced by a $\Bbbk$-form $V$ of $\mathbb{P}^2$.
Applying Lemma~\ref{lemma:SB} to $\Delta$ and $V$, we conclude that $V\simeq\mathbb{P}^2$.
Let $L$ be a line in $V$ and let $S=f_\ast \big{(} \pi^\ast (L) \big{)}$.
Then applying Lemma~\ref{lemma:SB} to $S$ and $U$, we conclude that $U\simeq\mathbb{P}^3$.
In particular, we see that $X({\Bbbk})\ne\varnothing$.
\end{proof}

\begin{lemma}
\label{lemma:2-11}
Suppose that $X$ is contained in Family \textnumero 2.11.
Then $X$ has a $\Bbbk$-point.
\end{lemma}

\begin{proof}
Arguing as in the proof of Lemma~\ref{lemma:2-9}, we see that there exists the following commutative diagram
$$
\xymatrix@R=1em{
&X\ar@{->}[ld]_{f}\ar@{->}[rd]^{\pi}&\\%
Y&&\mathbb{P}^2}
$$
where $Y$ is a form of a smooth cubic hypersurface $Y_{\mathbb{C}}\subset \mathbb{P}^4$,
$f$ is the blowup of a curve $C\subset Y$ such that $C_{\mathbb{C}}$ is a line in the cubic hypersurface $Y_{\mathbb{C}}$,
$\pi$ is a conic bundle. 
Then $Y$ is a $Y$ is a cubic hypersurface in $\mathbb{P}^4$ by Corollary~\ref{corollary:ICMS}, so $C$ is a line in it, which gives $C(\Bbbk)\ne\varnothing$.
In particular, $Y(\Bbbk)\ne\varnothing$, and Lemma \ref{lemma:nonempty} says that $X({\Bbbk})\ne \varnothing$ as well.
\end{proof}

\begin{lemma}
\label{lemma:2-14}
Suppose that $X$ is contained in Family \textnumero 2.14.
Then $X$ has a $\Bbbk$-point.
\end{lemma}

\begin{proof}
By Mori--Mukai \cite{MoMu81}, $X_{\mathbb{C}}$ can be obtained by blowing up of the smooth quintic del Pezzo 3-fold $V_5$ along an elliptic curve.
Thus, it follows from Lemma~\ref{lemma:Prokhorov} that $X$ can be obtained by blowing up a $\Bbbk$-form of $V_5$.
By \cite[Theorem 1.1]{KP2023}, any $\Bbbk$-form of $V_5$ is $\Bbbk$-rational, hence so is $X$,
in particular $X(\Bbbk)\ne\varnothing$.
\end{proof}

\begin{lemma}
\label{lemma:2-16}
Suppose that $X$ is contained in Family \textnumero 2.17.
Then $X$ has a $\Bbbk$-point.
\end{lemma}

\begin{proof}
It follows from Mori--Mukai \cite{MoMu81} and Lemma~\ref{lemma:Prokhorov} that there exists a birational morphism $\pi\colon X\to Q$ 
such that $Q$ is a $\Bbbk$-form of a smooth quadric 3-fold in $\mathbb{P}^4$, and $\pi$ is the blowup of a smooth elliptic curve $C$ such that $-K_Q\cdot C=15$. 
Then $Q$ is a quadric in $\mathbb{P}^4$ by Corollary~\ref{corollary:ICMS},
so $C$ is a curve of degree $5$ in it.
Now, taking hyperplane section of $C$, we obtain a zero-cycle in $Q$ of degree $5$ defined over $\Bbbk$,
which implies that $Q$ has a $\Bbbk$-point, so $X$ also has a $\Bbbk$-point by Lemma~\ref{lemma:nonempty}.
\end{proof}

\begin{lemma}
\label{lemma:2-20}
Suppose that $X$ is contained in Family \textnumero 2.20.
Then $X$ has a $\Bbbk$-point.
\end{lemma}

\begin{proof}
By Mori--Mukai \cite{MoMu81}, $X_{\mathbb{C}}$ can be obtained by blowing up a smooth quintic del Pezzo 3-fold $V_5$ along a twisted cubic curve.
Now, arguing as in the proof of Lemma~\ref{lemma:2-14}, we conclude that $X$ is rational over $\Bbbk$ and, in particular, it has a $\Bbbk$-point.
\end{proof}

\begin{lemma}
\label{lemma:2-22}
Suppose that $X$ is contained in Family \textnumero 2.22.
Then $X$ has a $\Bbbk$-point.
\end{lemma}

\begin{proof}
By Mori--Mukai \cite{MoMu81}, $X_{\mathbb{C}}$ can be obtained by blowing up a smooth quintic del Pezzo 3-fold $V_5$ along a conic.
Now, arguing as in the proof of Lemma~\ref{lemma:2-14}, we conclude that $X$ is rational over $\Bbbk$ and, in particular, it has a $\Bbbk$-point.
\end{proof}

\begin{lemma}
\label{lemma:3-8}
Suppose that $X$ is contained in Family \textnumero 3.8.
Then $X$ has a $\Bbbk$-point.
\end{lemma}

\begin{proof}
Note that $X_{\mathbb{C}}\subset \mathbb{F}_1\times\mathbb{P}^2$ is a divisor in the linear system \mbox{$|(\varsigma\circ\mathrm{pr}_1)^*(\mathcal{O}_{\mathbb{P}^2}(1))\otimes\mathrm{pr}_2^*(\mathcal{O}_{\mathbb{P}^2}(2))|$},
where $\mathrm{pr}_1\colon\mathbb{F}_1\times\mathbb{P}^2\to\mathbb{F}_1$ and $\mathrm{pr}_2\colon\mathbb{F}_1\times\mathbb{P}^2\to\mathbb{P}^2$ are projections to the first and the second factors,
respectively, and $\varsigma\colon\mathbb{F}_1\to\mathbb{P}^2$ is the blowup of a point.
Combining $\varsigma\circ\mathrm{pr}_1$ and $\mathrm{pr}_2$, we obtain a morphism $\sigma\colon X_{\mathbb{C}}\to Y$ such that $Y$ is a smooth divisor of degree $(1,2)$ in $\mathbb{P}^2\times\mathbb{P}^2$.
Let $\pi_1\colon Y\to\mathbb{P}^2$ and $\pi_2\colon Y\to\mathbb{P}^2$ be projections to the first and the second factors, respectively.
Then $\sigma$ is a blowup of a smooth curve $\mathcal{C}$ that is a fiber of the morphism~$\pi_1$.
Let $p=\pi_1(\mathcal{C})$. Then $\varsigma$ is a blowup of the point $p$ with commutative diagram
$$
\xymatrix{
&&X_{\mathbb{C}}\ar@/_1pc/@{->}[dll]_{\theta}\ar[d]_{\mathrm{pr}_1\vert_{X_{\mathbb{C}}}} \ar[rr]^{\sigma}\ar@/^3pc/@{->}[rrrr]^{\mathrm{pr}_2\vert_{X_{\mathbb{C}}}} && Y\ar[d]^{\pi_1} \ar[rr]^{\pi_2} && \mathbb{P}^2\\
\mathbb{P}^1&&\mathbb{F}_1 \ar[ll]_{\vartheta} \ar[rr]^{\varsigma} && \mathbb{P}^2}
$$
where $\vartheta$ is a~natural projection, $\theta$ is a fibration into del Pezzo surfaces of degree $5$.
Moreover, combining morphisms $\theta$ and $\mathrm{pr}_2\vert_{X_{\mathbb{C}}}$, we obtain a birational morphism
$\upsilon\colon X_{\mathbb{C}}\to\mathbb{P}^1\times \mathbb{P}^2$ that is a blowup of a smooth curve of degree~$(4,2)$.
This shows that the Mori cone $\mathrm{NE}(X_{\mathbb{C}})$ is simplicial and is generated by the following extremal rays:
\begin{enumerate}
\item the ray generated by the curves contracted by $\sigma\colon X_{\mathbb{C}}\to Y$,
\item the ray generated by the curves contracted by $\upsilon\colon X_{\mathbb{C}}\to \mathbb{P}^1\times \mathbb{P}^2$,
\item the ray generated by the curves contracted by $\mathrm{pr}_1\vert_{X_{\mathbb{C}}}\colon X_{\mathbb{C}}\to\mathbb{F}_1$.
\end{enumerate}
Now, arguing as in the proof of Lemma~\ref{lemma:Prokhorov},
we see that the conic bundle $\mathrm{pr}_1\vert_{X_{\mathbb{C}}}\colon X_{\mathbb{C}}\to\mathbb{F}_1$
descends to a conic bundle $X\to\mathbb{F}_1$ defined over $\Bbbk$,
because $\mathbb{F}_1$ does not have non-trivial forms over $\Bbbk$.
Now, composing this conic bundle with the projection $\vartheta\colon \mathbb{F}_1\to \mathbb{P}^1$,
we see that the del Pezzo fibration $\vartheta\colon X_{\mathbb{C}}\to\mathbb{P}^1$ is also defined over $\Bbbk$.
Since del Pezzo surfaces of degree $5$ are rational over any field, we see that $X$ is $\Bbbk$-birational to $\mathbb{P}^2\times\mathbb{P}^1$,
so that it is $\Bbbk$-rational and, in particular, it has a $\Bbbk$-point.
\end{proof}

\begin{lemma}
\label{lemma:3-11}
Suppose that $X$ is contained in Family \textnumero 3.11.
Then $X$ has a $\Bbbk$-point.
\end{lemma}

\begin{proof}
Over $\mathbb{C}$ we have a commutative diagram
$$
\xymatrix@C=1em{
&&&&\mathbb{P}^1\times\mathbb{P}^2\ar@/^2pc/@{->}[rrrrdd]^{\mathrm{pr}_2}\ar@/_2pc/@{->}[lllldd]_{\mathrm{pr}_1}&&&&\\
&&&&X_{\mathbb{C}}\ar@/^1pc/@{->}[rrrrd]^{\sigma}\ar@/_1pc/@{->}[lllld]_{\phi}\ar@{->}[dll]_{\theta}\ar@{->}[drr]^{\pi}\ar@{->}[dd]_{\xi}\ar@{->}[u]_{\zeta}&&&&\\
\mathbb{P}^1&&Y\ar@{->}[ll]_{\nu}\ar@{->}[drr]^{\varpi}&&&&V_7\ar@{->}[dll]_{\vartheta}\ar@{->}[rr]^{\eta}&&\mathbb{P}^2\\%
&&&&\mathbb{P}^3\ar@/_1pc/@{-->}[rrrru]\ar@/^1pc/@{-->}[llllu]&&&&}
$$
where $\vartheta$ is the blowup of a point $p\in\mathbb{P}^3$,
$\pi$ is the blow up of the strict transform of a smooth quartic elliptic curve $C$ that passes through the point $p$,
$\zeta$ is a birational contraction of the strict transform of the cubic cone in $\mathbb{P}^3$ with vertex at $p$ that contains the elliptic curve $C$
to a~smooth curve in $\mathbb{P}^1\times\mathbb{P}^2$ of degree $(2,3)$,
$\varpi$ is the blowup of the curve $C$, $\theta$ is the blowup of the fiber of $\varpi$ over the point $p$,
$\eta$ is a~$\mathbb{P}^1$-bundle, $\nu$ is a~fibration into quadric surfaces, $\sigma$ is a~conic bundle,
the left dashed arrow is given by the pencil of quadric surfaces that contain $C$,
the right dashed arrow is the linear projection from the point $p$,
and $\mathrm{pr}_1$ and $\mathrm{pr}_2$ are projections to the first and the second factors, respectively.
This shows that the Mori cone $\mathrm{NE}(X_{\mathbb{C}})$ is simplicial and is generated by the following extremal rays:
\begin{enumerate}
\item the ray generated by the curves contracted by $\theta\colon X_{\mathbb{C}}\to Y$,
\item the ray generated by the curves contracted by $\pi\colon X_{\mathbb{C}}\to V_7$,
\item the ray generated by the curves contracted by $\zeta\colon X_{\mathbb{C}}\to \mathbb{P}^1\times \mathbb{P}^2$.
\end{enumerate}
Now, arguing as in the proof of Lemma~\ref{lemma:Prokhorov}, we see that these extremal rays are defined over $\Bbbk$,
so their contractions can also be defined over $\Bbbk$.
Since $V_7$ does not have non-trivial forms over $\Bbbk$, see Lemma~\ref{lemma:2-35} below,
we see that $X$ is rational over $\Bbbk$. In particular, we have $X(\Bbbk)\ne\varnothing$.
\end{proof}

\subsection{K-unstable families}
\label{subsection:K-unstable-with-points}

\begin{lemma}
\label{lemma:2-26}
Suppose that $X$ is contained in Family \textnumero 2.26.
Then $X$ has a $\Bbbk$-point.
\end{lemma}

\begin{proof}
By Mori--Mukai \cite{MoMu83}, $X_{\mathbb{C}}$ can be obtained by blowing up a smooth quintic del Pezzo 3-fold $V_5$ along a line.
Arguing as in the proof of Lemma~\ref{lemma:2-14}, we conclude that $X$ is rational over $\Bbbk$ and $X(\Bbbk)\ne\varnothing$.
\end{proof}

\begin{lemma}
\label{lemma:2-28-2-30}
Suppose that $X$ is contained in Family \textnumero 2.28 or in Family \textnumero 2.30.
Then $X$ has a $\Bbbk$-point.
\end{lemma}

\begin{proof}
Over $\mathbb{C}$, the 3-fold $X_{\mathbb{C}}$ can be obtained by blowing up $\mathbb{P}^3$ along a smooth plane curve of degree $3$ or $2$.
Applying Lemma~\ref{lemma:Prokhorov} and Lemma~\ref{lemma:plane}, we see that $X$ is also obtained by blowing up $\mathbb{P}^3$ along a smooth plane curve defined over $\Bbbk$.
This implies that $X$ is rational over $\Bbbk$, in particular $X(\Bbbk)\ne\varnothing$.
\end{proof}

\begin{lemma}
\label{lemma:2-31}
Suppose that $X$ is contained in Family \textnumero 2.31.
Then $X$ has a $\Bbbk$-point.
\end{lemma}

\begin{proof}
By Mori--Mukai \cite{MoMu83}, the base extension $X_{\mathbb{C}}$ can be obtained by blowing up a smooth quadric 3-fold $Q\subseteq \mathbb{P}^4$ along a line.
Let $E$ be the exceptional divisor of this blowup. Then, by Lemma~\ref{lemma:Prokhorov}, the surface $E$ is defined over $\Bbbk$.
On the other hand, it is well known that $E_{\mathbb{C}}$ is isomorphic to $\mathbb{F}_1$.
Since $\mathbb{F}_1$ does not have non-trivial forms over $\Bbbk$, we conclude that $E\simeq\mathbb{F}_1$, so $E(\Bbbk)\ne\varnothing$.
Hence, $X(\Bbbk)\ne\varnothing$ as well. This also implies that $X$ is $\Bbbk$-rational,
since forms of smooth quadrics containing $\Bbbk$-points are $\Bbbk$-rational.
\end{proof}

\begin{lemma}
\label{lemma:2-35}
Suppose that $X$ is contained in Family \textnumero 2.35.
Then $X$ is isomorphic to the blowup of $\mathbb{P}^3$ at a point.
In particular, $X$ has a $\Bbbk$-point.
\end{lemma}

\begin{proof}
A variety in this family is often called $V_7$. By Mori--Mukai \cite{MoMu83}, $X_{\mathbb{C}}$ can be obtained by blowing up $\mathbb{P}^3$ at a point $p$.
By Lemma~\ref{lemma:Prokhorov}, $X$ can be obtained by blowing up a $\Bbbk$-form of $\mathbb{P}^3$,
say $X\to U$ such that the image $p$ of the exceptional divisor yields a ${\Bbbk}$-rational point of a Severi--Brauer 3-fold $U$,
in particular, $U$ is isomorphic to $\mathbb{P}^3$.
\end{proof}

\begin{lemma}
\label{lemma:2-36}
Suppose that $X$ is contained in Family \textnumero 2.36.
Then $X$ has a $\Bbbk$-point.
\end{lemma}

\begin{proof}
By Mori--Mukai \cite{MoMu83}, the base extension $X_{\mathbb{C}}$,
which is isomorphic to $\mathbb{P}( {\mathcal O}_{\mathbb{P}^2} \oplus {\mathcal O}_{\mathbb{P}^2}(-2))$,
possesses two extremal contractions: a divisorial contraction $f\colon X_{\mathbb{C}} \to\mathbb{P}(1,1,1,1,2)$
and a $\mathbb{P}^1$-bundle $\pi\colon X_{\mathbb{C}} \to \mathbb{P}^2$.
Since the action of $\mathrm{Gal}(\mathbb{C}/{\Bbbk})$ on the Mori cone $\mathrm{NE}(X_{\mathbb{C}})$ leaves the two rays invariant,
both $f$ and $\pi$ are defined over $\Bbbk$.
With a slight abuse of notation, denote by $\pi$ the descent $\pi\colon X \to U$ over $\Bbbk$, where $U$ is a Severi--Brauer surface.
The exceptional divisor $E$ of $f$ is defined over $\Bbbk$, and $\pi$ induces an isomorphism $E\simeq U$.
Moreover, applying Lemma~\ref{lemma:SB} to the divisor $E|_E$ and the Severi--Brauer surface $E$, we conclude that $E\cong\mathbb{P}^2$.
In particular, we have $E(\Bbbk)\ne\varnothing$, so $X$ has a $\Bbbk$-point.
Indeed, one can show that $X$ is rational over $\Bbbk$.
\end{proof}

\begin{lemma}
\label{lemma:3-14}
Suppose that $X$ is contained in Family \textnumero 3.14.
Then $X$ has a $\Bbbk$-point.
\end{lemma}

\begin{proof}
Let $\Pi$ be a plane in $\mathbb{P}^3$,
and let $p$ be a point in $\mathbb{P}^3$ with $p\not\in \Pi$,
let $\phi\colon V_7\to\mathbb{P}^3$ be the blowup of this point,
and let $\widetilde{\Pi}$ be the proper transform on $V_7$ of the plane $\Pi$.
Then there exists a birational morphism $\pi\colon X_{\mathbb{C}}\to V_7$ that is a blowup of a smooth elliptic curve $C\subset \widetilde{\Pi}$.
Set $\mathscr{C}=\phi(C)$. Then $\mathscr{C}$ is smooth plane cubic curve in $\mathbb{P}^3$.
Let $E_C$ be the $\pi$-exceptional surface, and let $E_P$, $H_C$, $F$ be the proper transforms on the 3-fold $X_{\mathbb{C}}$ of the $\phi$-exceptional surface,
the plane $\Pi$, and the cubic cone in $\mathbb{P}^3$ over the curve $\mathscr{C}$ with vertex $p$, respectively.
Then we have the following commutative diagram:
$$
\xymatrix{
\mathbb{P}^2&&\mathbb{P}\big(\mathcal{O}_{\mathbb{P}^2}\oplus\mathcal{O}_{\mathbb{P}^2}(2)\big)\ar@{->}[ll]\ar@{->}[rr]&&\mathbb{P}(1,1,1,2) \\
V_7\ar@{->}[d]_{\phi}\ar@{->}[u]&&X_{\mathbb{C}}\ar@{->}[ll]_{\pi}\ar@{->}[u]_{\psi}\ar@{->}[d]^{\varphi}\ar@{->}[rr]^{\sigma} &&\widehat{Y}\ar@{->}[u]\ar@{->}[d]\\
\mathbb{P}^3 &&\widetilde{\mathbb{P}}^3\ar@{->}[ll]^{\varpi}\ar@{->}[rr]_{\varsigma}&&Y}
$$
where $\varpi$ is the blowup of the curve $\mathscr{C}$,
 $\varphi$ is the contraction of the surface $E_P$,
  $\sigma$ and $\psi$ are the contractions of the surfaces $H_C$ and $F$, respectively,
 $\varsigma$ is the contraction of the surface $\varphi(H_C)$,
 $Y$ is a Fano 3-fold that has a singular point of type $\frac{1}{2}(1,1,1)$,
the morphism $\widehat{Y}\to Y$ is the blowup of a smooth point of the 3-fold $Y$,
both $V_7\to\mathbb{P}^2$ and  $\mathbb{P}(\mathcal{O}_{\mathbb{P}^2}\oplus\mathcal{O}_{\mathbb{P}^2}(2))\to\mathbb{P}^2$ are $\mathbb{P}^1$-bundles,
the morphism $\mathbb{P}(\mathcal{O}_{\mathbb{P}^2}\oplus\mathcal{O}_{\mathbb{P}^2}(2))\to\mathbb{P}(1,1,1,2)$ is the contraction of the surface $\psi(H_C)$,
and  $\widehat{Y}\to\mathbb{P}(1,1,1,2)$ is the contraction of  $\sigma(F)$.
This shows that the Mori cone $\mathrm{NE}(X_{\mathbb{C}})$ is generated by the extremal rays
that are spanned by the curves contracted by $\psi$, $\varphi$, $\pi$, $\sigma$.
Since the Galois group $\mathrm{Gal}({\mathbb{C}}/{\Bbbk})$ cannot permute any of these rays,
we see that the commutative diagram above descents to $\Bbbk$.
Since $V_7$ does not have non-trivial forms over $\Bbbk$ by Lemma~\ref{lemma:2-35},
we see that $X$ is $\Bbbk$-rational and, in particular, has a $\Bbbk$-point.
\end{proof}

\begin{lemma}
\label{lemma:3-16}
Suppose that $X$ is contained in Family \textnumero 3.16.
Then $X$ has a $\Bbbk$-point.
\end{lemma}

\begin{proof}
Let $\mathscr{C}$ be a twisted cubic curve in the space $\mathbb{P}^3$, let $p$ be a~point in the curve $\mathscr{C}$,
let $\phi\colon V_7\to\mathbb{P}^3$ be the blowup of this point, and let $C$ be the proper transform of the cubic curve $\mathscr{C}$ on the 3-fold $V_7$.
Then $X_{\mathbb{C}}$ can be obtained as the blowup $\pi\colon X_{\mathbb{C}}\to V_7$ along the curve $C$.
One can see that $X$ fits into the commutative diagram
$$
\xymatrix@R=1em{
& &W\ar@{->}[dll]_{p_1}\ar@{->}[drr]^{p_2} & & \\
\mathbb{P}^2& & & &\mathbb{P}^2 \\
& & X_{\mathbb{C}}\ar@{->}[drr]^{\varphi}\ar@{->}[dll]_{\pi}\ar@{->}[uu]_{\psi} & & \\
V_7\ar@{->}[drr]_{\phi}\ar@{->}[uu]& & & &\widetilde{\mathbb{P}}^3\ar@{->}[dll]^{\varpi}\ar@{->}[uu] \\
& &\mathbb{P}^3 & &}
$$
where $W$ is a smooth divisor of degree $(1,1)$ in $\mathbb{P}^2\times\mathbb{P}^2$,
both $p_1$ and $p_2$ are $\mathbb{P}^1$-bundles,
the morphism $\varpi$ is the blowup of $\mathbb{P}^3$ along $\mathscr{C}$, the morphism~\mbox{$\widetilde{\mathbb{P}}^3\to\mathbb{P}^2$} is a $\mathbb{P}^1$-bundle
whose fibers are proper transforms of the secant lines in $\mathbb{P}^3$ of the twisted cubic curve $\mathscr{C}$,
the morphism $V_7\to\mathbb{P}^2$ is the $\mathbb{P}^1$-bundle whose fibers are proper transforms of the lines in the space $\mathbb{P}^3$ that pass through $p$,
and $\varphi$ is the blowup of the fiber of $\varpi$ over $p$.
Observe that the Mori cone $\mathrm{NE}(X_{\mathbb{C}})$ is simplicial and is generated by the extremal rays
spanned by the curves contracted by $\psi$, $\varphi$, $\pi$.
Since the Galois group $\mathrm{Gal}({\mathbb{C}}/{\Bbbk})$ cannot permute any of these rays,
we see that the commutative diagram above descents to $\Bbbk$.
Now, as in the proof of Lemma~\ref{lemma:3-14}, we see that $X$ is rational over $\Bbbk$.
In particular, we have $X(\Bbbk)\ne\varnothing$.
\end{proof}

\begin{lemma}
\label{lemma:3-18}
Suppose that $X$ is contained in Family \textnumero 3.18.
Then $X$ has a $\Bbbk$-point.
\end{lemma}

\begin{proof}
The Fano 3-fold $X_{\mathbb{C}}$ can be obtained as a blowup
$\pi\colon X_{\mathbb{C}}\to \mathbb{P}^3$ along a disjoint union of a smooth conic $C$ and a line $L$.
There is a commutative diagram
$$
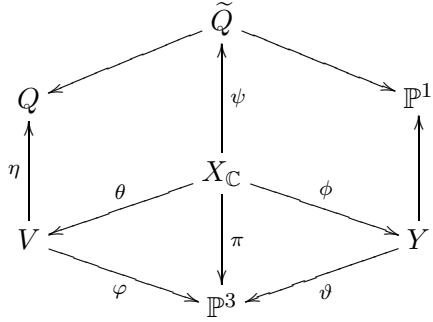
\xymatrix@R=1em{
& &\widetilde{Q}\ar@{->}[dll]\ar@{->}[drr] & & \\
Q& & & &\mathbb{P}^1 \\
& & X_{\mathbb{C}}\ar@{->}[dd]^{\pi}\ar@{->}[drr]^{\phi}\ar@{->}[dll]_{\theta}\ar@{->}[uu]_{\psi} & & \\
V\ar@{->}[drr]_{\varphi}\ar@{->}[uu]^{\eta}& & & &Y\ar@{->}[dll]^{\vartheta}\ar@{->}[uu]\\
& &\mathbb{P}^3 & &}
$$
where  $\vartheta$ is the blowup of the line $L$,
the morphism $\varphi$ is the blowup of the conic $C$,
the morphisms $\theta$ and $\phi$ are blowups of the proper transforms of the curves $L$ and $C$, respectively,
 $Q$ is a smooth quadric in $\mathbb{P}^4$,
the morphism $\eta$ is the blowup of a point in $Q$,
the morphism $\widetilde{Q}\to Q$ is the blowup of a conic (the proper transform of the line $L$),
the morphism \mbox{$Y\to\mathbb{P}^1$} is a $\mathbb{P}^2$-bundle,
the morphism $\widetilde{Q}\to\mathbb{P}^1$ is a fibration into quadric surfaces,
and $\psi$ is the contraction of the proper transform of the plane in $\mathbb{P}^3$ containing $C$.

This shows that the Mori cone $\mathrm{NE}(X_{\mathbb{C}})$ is simplicial and is generated by the extremal rays
spanned by the curves contracted by $\theta$, $\phi$, $\psi$.
Since the Galois group $\mathrm{Gal}({\mathbb{C}}/{\Bbbk})$ cannot permute any of these rays,
we see that the commutative diagram above descends to $\Bbbk$.
In particular, we can obtain $X$ as the blowup of a Severi--Brauer 3-fold $U$ along a disjoint union of a
twisted line and a twisted conic.
Now, applying Lemma~\ref{lemma:Prokhorov} to $U$ and the twisted conic, we see that $U\simeq\mathbb{P}^3$.
Hence, we see that $X$ is rational over $\Bbbk$.
\end{proof}

\begin{lemma}
\label{lemma:3-21}
Suppose that $X$ is contained in Family \textnumero 3.21.
Then $X$ has a $\Bbbk$-point.
\end{lemma}

\begin{proof}
Over $\mathbb{C}$, there exists a blowup $\pi\colon X_{\mathbb{C}}\to\mathbb{P}^1\times\mathbb{P}^2$ of a smooth curve $C$ of degree $(2,1)$.
Let $S$ be the proper transform on $X$ of the surface in $\mathbb{P}^1\times\mathbb{P}^2$ of degree $(0,1)$ that passes through the curve $C$,
let $\ell_1$ and $\ell_2$ be the rulings of the surface $S\cong\mathbb{P}^1\times\mathbb{P}^1$
such that the curves $\pi(\ell_1)$ and $\pi(\ell_2)$ are of degree $(1,0)$ and $(0,1)$ in  $\mathbb{P}^1\times\mathbb{P}^2$, respectively.
Let $E$ be the $\pi$-exceptional surface, and let $\ell_3$ be a fiber of the natural projection \mbox{$E\to C$}.
Then the curves $\ell_1$, $\ell_2$, $\ell_3$ generate the Mori cone $\overline{\mathrm{NE}}(X)$,
and the extremal rays $\mathbb{R}_{\geqslant 0}[\ell_1]$ and $\mathbb{R}_{\geqslant 0}[\ell_2]$ give birational contractions $X\to U_1$ and $X\to U_2$, respectively.
Moreover, it follows from  the proof of \cite[Lemma~8.22]{CheltsovShramovUMN} that there is a commutative diagram
$$
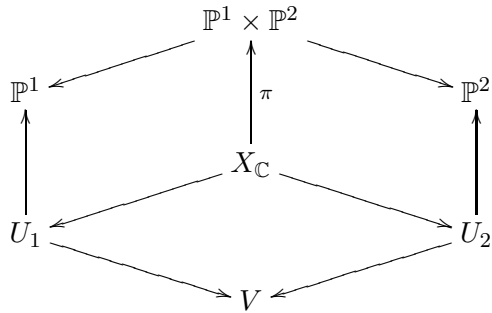
\xymatrix@R=1em{
& &\mathbb{P}^1\times\mathbb{P}^2\ar@{->}[dll]\ar@{->}[drr] & & \\
\mathbb{P}^1& & & &\mathbb{P}^2 \\
& & X_{\mathbb{C}}\ar@{->}[drr]\ar@{->}[dll]\ar@{->}[uu]_{\pi} & & \\
U_1\ar@{->}[drr]\ar@{->}[uu]& & & &U_2\ar@{->}[dll]\ar@{->}[uu] \\
& &V & &}
$$
where the morphism $U_1\to \mathbb{P}^1$ is a quadric bundle, the morphism $U_2\to \mathbb{P}^2$ is a $\mathbb{P}^1$-bundle,
the map $U_1\dashrightarrow U_2$ is a flop,
and $V$ is a Fano 3-fold in Family \textnumero 1.15 with one isolated ordinary double point singularity.
For details, we refer the reader to the case~(2.3.2) in~\cite[Theorem~2.3]{TakeuchiNew}.

Since the Galois group $\mathrm{Gal}({\mathbb{C}}/{\Bbbk})$ cannot non-trivially permute extremal rays $\mathbb{R}_{\geqslant 0}[\ell_1]$,
$\mathbb{R}_{\geqslant 0}[\ell_2]$, $\mathbb{R}_{\geqslant 0}[\ell_3]$, we see that the diagram above is defined over $\Bbbk$
with $X_{\mathbb{C}}$ replaced by $X$, $\mathbb{P}^1$ replaced by a (possibly pointless) conic $C_2$,
and $\mathbb{P}^2$ is replaced by its $\Bbbk$-form $U$.
Then we may assume that $C$ is a curve in $C_2\times U$ defined over $\Bbbk$,
and $\pi$ is the blowup of the product $C_2\times U$ along this curve.
Then the image of the curve $C$ in $U$ via the natural projection $\mathrm{pr}_2\colon C_2\times U\to U$ is a twisted line in the Severi--Brauer surface $U$,
so it follows from Lemma~\ref{lemma:SB} that $U\simeq\mathbb{P}^2$ and $\mathrm{pr}_2(C)$ is a line.
Moreover, since $\mathrm{pr}_2|_C: C \to \mathrm{pr}_2(C)$ is an isomorphism, we see that $C$ is isomorphic to $\mathbb{P}^1$,
which implies in turn that $C_2\simeq\mathbb{P}^1$ via the projection $\mathrm{pr}_1\colon \colon C_2\times U\to C_2$ in consideration of L\"uroth's Theorem.
Therefore $X$ is birational to $\mathbb{P}^1\times\mathbb{P}^2$ over $\Bbbk$, so $X$ is rational over $\Bbbk$.
In particular, we have $X ({\Bbbk})\ne\varnothing$.
\end{proof}

\begin{lemma}
\label{lemma:3-22}
Suppose that $X$ is contained in Family \textnumero 3.22.
Then $X$ has a $\Bbbk$-point.
\end{lemma}

\begin{proof}
Let $\mathrm{pr}_1\colon \mathbb{P}^1\times\mathbb{P}^2\to\mathbb{P}^1$ and
$\mathrm{pr}_2\colon \mathbb{P}^1\times\mathbb{P}^2\to\mathbb{P}^2$ be the projections to the first and the second factors, respectively,
let $H_1$ be a~fiber of the map $\mathrm{pr}_1$, let \mbox{$H_2=\mathrm{pr}_2^*(\mathcal{O}_{\mathbb{P}^2}(1))$},
and let $C$ be a~conic in $H_1\cong\mathbb{P}^2$.
Then there is a~blow up $\psi\colon X_{\mathbb{C}}\to \mathbb{P}^1\times\mathbb{P}^2$ along the curve $C$.

Let $E_C$ be the $\psi$-exceptional surface, let $\widetilde{H}_1$ be the proper transform  of the surface $H_1$~on the 3-fold $X$,
let $F$ be~the surface in $|H_2|$ that contains $C$,
and let $\widetilde{F}$ be  the proper transform of this surface on $X$.
We have the following commutative diagram:
$$
\xymatrix@R=1em{
& &\mathbb{P}^1\times\mathbb{P}^2\ar@{->}[dll]_{\mathrm{pr}_1}\ar@{->}[drr]^{\mathrm{pr}_2} & & \\
\mathbb{P}^1& & & &\mathbb{P}^2 \\
& & X_{\mathbb{C}}\ar@{->}[drr]^{\phi}\ar@{->}[dll]_{\pi}\ar@{->}[uu]_{\psi} & & \\
Y\ar@{->}[drr]_{\varphi}\ar@{->}[uu]^{\eta}& & & &\mathbb{P}\big(\mathcal{O}_{\mathbb{P}^2}\oplus\mathcal{O}_{\mathbb{P}^2}(2)\big)\ar@{->}[dll]^{\varpi}\ar@{->}[uu]_{\sigma} \\
& &\mathbb{P}(1,1,1,2) & &}
$$
where $\pi$ and $\phi$ are the contraction of the surfaces $\widetilde{H}_1\cong\mathbb{P}^2$ and $\widetilde{F}\cong\mathbb{P}^1\times\mathbb{P}^1$, respectively,
the morphisms $\varpi$ and $\varphi$ are the contractions of the surfaces $\phi(\widetilde{H}_1)$ and $\pi(\widetilde{F})$, respectively,
the morphism $\sigma$ is a~$\mathbb{P}^1$-bundles, and $\eta$ is a~fibration into del Pezzo surfaces such that all its fibers except $\pi(\widetilde{F})$ are isomorphic to $\mathbb{P}^2$,
while $\pi(\widetilde{F})\cong\mathbb{P}(1,1,4)$.
Note that the Mori cone $\overline{\mathrm{NE}}(X)$ is simplicial and it is generated by the extremal rays contracted by $\pi$, $\phi$ and $\psi$.
As in the proof of Lemma~\ref{lemma:3-21}, we see that the above diagram can be defined over $\Bbbk$
with $X_{\mathbb{C}}$ replaced by $X$, $\mathbb{P}^1$ replaced by a (possibly pointless) conic $C_2$,
and $\mathbb{P}^2$ is replaced by its $\Bbbk$-form $U$.
Then $\pi$ is a blow up of the product $C_2\times U$ along a curve $C$ defined over $\Bbbk$
such that $\mathrm{pr}_1(C)$ is a point in $C_2$, and $\mathrm{pr}_2(C)$ is a twisted conic in the Severi--Brauer surface $U$.
Now, applying Lemma~\ref{lemma:SB}, we see that $C_2\cong\mathbb{P}^1$ and $U\cong\mathbb{P}^2$,
so $X$ is $X$ is rational over $\Bbbk$, which gives $X ({\Bbbk})\ne\varnothing$.
\end{proof}

\begin{lemma}
\label{lemma:3-23}
Suppose that $X$ is contained in Family \textnumero 3.23.
Then $X$ has a $\Bbbk$-point.
\end{lemma}

\begin{proof}
Let $\mathscr{C}$ be a~smooth conic in $\mathbb{P}^3$, let $p$ be an arbitrary point in the conic $\mathscr{C}$,
let $\phi\colon V_7\to\mathbb{P}^3$ be the blowup of the point $p$, and let $C$ be the proper transform on the 3-fold $V_7$ of the conic $\mathscr{C}$ .
Then, over $\mathbb{C}$, there exists a birational morphism $\pi\colon X_C\to V_7$ that is the blowup of the curve $C$.
One can see that $X$ fits into the commutative diagram
$$
\xymatrix@R=1em{
& &\widehat{Q}\ar@{->}[dll]\ar@{->}[drr] & & \\
\mathbb{P}^2& & & &Q \\
& & X\ar@{->}[drr]^{\varphi}\ar@{->}[dll]_{\pi}\ar@{->}[uu]_{\psi} & & \\
V_7\ar@{->}[drr]_{\phi}\ar@{->}[uu]& & & &\widetilde{\mathbb{P}}^3\ar@{->}[dll]^{\varpi}\ar@{->}[dll]\ar@{->}[uu] \\
& &\mathbb{P}^3 & &}
$$
where $Q$ is a smooth quadric 3-fold in $\mathbb{P}^4$,
the morphism $\varpi$ is the blowup of the conic $\mathscr{C}$,
the morphism \mbox{$\widetilde{\mathbb{P}}^3\to Q$} is the contraction to a point of the proper transform of the plane in $\mathbb{P}^3$ that contains $\mathscr{C}$,
$\varphi$ is the blowup of the fiber of the morphism $\varpi$ over the point $p$,
the morphism $\widehat{Q}\to Q$ is the blowup of a line in $Q$ that passes through the latter point,
and $\widehat{Q}\to\mathbb{P}^2$ is a $\mathbb{P}^1$-bundle.
Hence, the Mori cone $\mathrm{NE}(X_{\mathbb{C}})$ is simplicial and is generated by the extremal rays
spanned by the curves contracted by $\psi$, $\varphi$, $\pi$.
Since the Galois group $\mathrm{Gal}({\mathbb{C}}/{\Bbbk})$ cannot permute any of these rays, the commutative diagram above descents to $\Bbbk$.
Since $V_7$ does not have non-trivial forms over $\Bbbk$ by Lemma~\ref{lemma:2-35},
we see that $X$ is $\Bbbk$-rational.
\end{proof}

\begin{lemma}
\label{lemma:3-24}
Suppose that $X$ is contained in Family \textnumero 3.24.
Then $X$ has a $\Bbbk$-point.
\end{lemma}

\begin{proof}
Over $\mathbb{C}$, there is a blowup $\phi\colon X_C\to\mathbb{P}^1\times\mathbb{P}^2$ of a smooth curve $C$ of degree~$(1,1)$,
and we have the following commutative diagram
$$
\xymatrix{
&&\mathbb{P}^1\times\mathbb{P}^2\ar@/_2pc/@{->}[ddll]_{\mathrm{pr}_1}\ar@/^2.5pc/@{->}[ddrr]^{\mathrm{pr}_2}&&\\%
&&X_{\mathbb{C}}\ar@{->}[u]_{\phi}\ar@{->}[dll]_{\zeta}\ar@{->}[d]^{\pi}\ar@{->}[r]^{\alpha}&W\ar@{->}[dr]^{\omega_{1}}&\\%
\mathbb{P}^{1}&&\mathbb{F}_{1}\ar@{->}^{\xi}[ll]\ar@{->}^{\gamma}[rr]&&\mathbb{P}^{2}}
$$
where $W$ is a divisor of degree $(1,1)$ on $\mathbb{P}^{2}\times\mathbb{P}^{2}$,
$\omega_{1}$ is a natural $\mathbb{P}^{1}$-bundle,
$\alpha$ contracts a smooth surface $E\cong\mathbb{P}^{1}\times\mathbb{P}^{1}$ to a fiber $L$ of $\omega_{1}$,
$\gamma$ is the blowup of the point $\omega_1(L)$, the morphism $\xi$ is a $\mathbb{P}^{1}$-bundle,
$\zeta$ is a $\mathbb{F}_{1}$-bundle, $\mathrm{pr}_1$ and $\mathrm{pr}_2$ are projections to the first and the second factors, respectively.
This commutative diagram shows that the Mori cone $\mathrm{NE}(X_{\mathbb{C}})$ is generated by the extremal rays
spanned by the curves contracted by $\pi$, $\alpha$, $\phi$.
Since the Galois group $\mathrm{Gal}({\mathbb{C}}/{\Bbbk})$ cannot non-trivially permute these rays,
we see that the commutative diagram above descents to $\Bbbk$
with $X_{\mathbb{C}}$ replaced by $X$, $\mathbb{P}^1$ replaced by a (possibly pointless) conic $C_2$,
and $\mathbb{P}^2$ is replaced by its $\Bbbk$-form $U$.
Then we may assume that $C$ is a curve in $C_2\times U$ defined over $\Bbbk$,
and $\pi$ is the blowup of the product $C_2\times U$ along this curve.
But $\mathrm{pr}_2(C)$ is a twisted line in $U$, which gives $U\simeq\mathbb{P}^2$ by Lemma~\ref{lemma:SB}.
Moreover, since $\mathrm{pr}_2|_C\colon C \to \mathrm{pr}_2(C)$ is an isomorphism and $\mathrm{pr}_2(C)$ is a line, we see that $C\cong\mathbb{P}^1$.
Then $C_2\simeq\mathbb{P}^1$ as well, since $\mathrm{pr}_1|_C\colon C \to C_2$ is an isomorphism.
Therefore, we see that $X$ is birational to $\mathbb{P}^1\times\mathbb{P}^2$ over $\Bbbk$.
In particular, $X$ has $\Bbbk$-point.
\end{proof}

\begin{lemma}
\label{lemma:3-26}
Suppose that $X$ is contained in Family \textnumero 3.26.
Then $X$ has a $\Bbbk$-point.
\end{lemma}

\begin{proof}
Let $V_7$ be the blowup of $\mathbb{P}^3$ at a point $p$,
let $L$ be a line in  $\mathbb{P}^3$ not containing $p$,
and let $C$ be its strict transform on $V_7$.
Then $X_{\mathbb{C}}$ can be obtained by blowing up $V_7$ along the curve $C$,
and it follows from \cite{MoMu81} and \cite{MoMu83} that $X_{\mathbb{C}}$ has exactly one divisorial contraction,
the inverse of the blowing up $X_{\mathbb{C}} \to V_7$ of the curve $C$.
Thus, the blowup $X_{\mathbb{C}} \to V_7$ descents to $\Bbbk$, but $V_7$ does not have non-trivial $\Bbbk$-forms by Lemma~\ref{lemma:2-35},
which implies that $X$ can be obtained by blowing up $V_7$ over $\Bbbk$. In particular, $X({\Bbbk})\ne\varnothing$.
\end{proof}

\begin{lemma}
\label{lemma:3-29}
Suppose that $X$ is contained in Family \textnumero 3.29.
Then $X$ has a $\Bbbk$-point.
\end{lemma}

\begin{proof}
As in the proof of Lemma~\ref{lemma:3-26}, let $f\colon V_7\to \mathbb{P}^3$ be the blowup of a point $p$,
let $E$ be the $f$-exceptional surface, and let $C$ be a line in $E\simeq\mathbb{P}^2$.
Then $X_{\mathbb{C}}$ can be obtained by blowing up $V_7$ along the curve $C$.
Moreover, it follows from \cite{MoMu81} and \cite{MoMu83} that $X_{\mathbb{C}}$ has two extremal contractions, one of which is to $V_7$ and the other is to $\mathbb{P}(\mathcal{O}_{\mathbb{P}^2} \oplus \mathcal{O}_{\mathbb{P}^2}(2))$.
In particular, they are both defined over $\Bbbk$ and we obtain the required assertion as in the proof of Lemma~\ref{lemma:3-26}.
\end{proof}

\begin{lemma}
\label{lemma:3-30}
Suppose that $X$ is contained in Family \textnumero 3.30.
Then $X$ has a $\Bbbk$-point.
\end{lemma}

\begin{proof}
As in the proof of Lemmas~\ref{lemma:3-26} and \ref{lemma:3-29}, let $f\colon V_7\to \mathbb{P}^3$ be the blowup of a point $p$.
Then $V_7\cong\mathbb{P}( \mathcal{O}_{\mathbb{P}^2} \oplus \mathcal{O}_{\mathbb{P}^2}(1))$.
Let $L$ be a fiber of the natural projection $V_7\to\mathbb{P}^2$.
Then $X_{\mathbb{C}}$ can be obtained by blowing up $V_7$ along $L$.
Moreover, it follows from  \cite{MoMu81} and \cite{MoMu83} that $X_{\mathbb{C}}$ has two extremal contractions: one of them is the birational morphism $X_{\mathbb{C}}\to V_7$,
and the other one is a birational contraction of $V_7\to\mathbb{P}(\mathcal{O}_{\mathbb{P}^1}^{\oplus 2}\oplus\mathcal{O}_{\mathbb{P}^1}(1))$.
In particular, both morphisms must be defined over $\Bbbk$.
Now, arguing as in the proof of Lemma~\ref{lemma:3-26}, we see that  $X$ is $\Bbbk$-rational and, in particular, it has a $\Bbbk$-point.
\end{proof}

\begin{lemma}
\label{lemma:4-5}
Suppose that $X$ is contained in Family \textnumero 4.5.
Then $X$ has a $\Bbbk$-point.
\end{lemma}

\begin{proof}
Let $C$ be a curve of degree $(2,1)$ in $\mathbb{P}^1\times\mathbb{P}^2$,
let $L$ be a curve of degree $(1,0)$ in $\mathbb{P}^1\times\mathbb{P}^2$ that is disjoint from $C$.
let $f\colon Y\to \mathbb{P}^1\times\mathbb{P}^2$ be the blowup of the curve $C$,
and let $\widetilde{L}$ be the strict transform of the curve $L$ on the 3-fold $Y$.
Then it follows from \cite{MoMu81,MoMu83} that the base extension $X_{\mathbb{C}}$ is isomorphic over $\mathbb{C}$
to the the blowup of $Y$ along the curve $\widetilde{L}$.
Moreover, analyzing the Mori cone of the 3-fold $X_{\mathbb{C}}$, we see that the birational morphism $X_C\to Y$
descends to the $\Bbbk$-birational morphism from $X$ to a $\Bbbk$-form of $Y$.
From the proof of Lemma~\ref{lemma:3-21}, we know that all $\Bbbk$-forms of $Y$ are $\Bbbk$-rational,
so $X$ is also $\Bbbk$-rational. In particular, $X$ has a $\Bbbk$-point.
\end{proof}

\begin{lemma}
\label{lemma:4-9}
Suppose that $X$ is contained in Family \textnumero 4.9.
Then $X$ has a $\Bbbk$-point.
\end{lemma}

\begin{proof}
Let $L_1$ and $L_2$ be two disjoint lines in $\mathbb{P}^3$,
let $f\colon Y\to\mathbb{P}^3$ be the blowup of the curves $L_1$ an $L_2$,
let $E_1$ and $E_2$ be the $f$-exceptional surfaces such that $f(E_1)=L_1$ and $f(E_2)=L_2$,
and let $C$ be a fiber of the natural projection $E_1\to L_1$.
Then there exists a birational morphism $g\colon X_{\mathbb{C}}\to Y$ that blows up the curve $C$.

Let $E_C$ be the $g$-exceptional surface,
let $\widetilde{E}_1$ and $\widetilde{E}_2$ be the strict transforms on $X_{\mathbb{C}}$ of the surfaces $E_1$ and $E_2$, respectively.
Then there exists a birational contraction $h\colon X_{\mathbb{C}}\to V_7$ of the surfaces $\widetilde{E}_1$ and $\widetilde{E}_2$
such that $V_7$ is the blowup of $\mathbb{P}^3$ at the point $f(C)$, and $h(E_C)$ is the exceptional divisor of the morphism $V_7\to \mathbb{P}^3$.
To be precise, we have the following commutative diagram:
$$
\xymatrix@R=1em{
&&X_{\mathbb{C}}\ar@{->}[dll]_{g}\ar@{->}[drr]^{h}&&\\%
Y\ar@{->}[drr]_{f}&&&&V_7\ar@{->}[dll]\\%
&&\mathbb{P}^3&&}
$$
Moreover, it follows from \cite{MoMu81} and \cite{MoMu83} that this commutative diagram is  defined over $\Bbbk$,
so that $X$ is $\Bbbk$-rational, since $V_7$ does not have non-trivial $\Bbbk$-forms by Lemma~\ref{lemma:2-35}.
In particular, we see that $X$ has a $\Bbbk$-point.
\end{proof}

\begin{lemma}
\label{lemma:4-11}
Suppose that $X$ is contained in Family \textnumero 4.11.
Then $X$ has a $\Bbbk$-point.
\end{lemma}

\begin{proof}
Let $V=\mathbb{P}^1 \times\mathbb{F}_1$, let $S$ be a fiber of the natural projection $V\to\mathbb{P}^1$,
and let $C$ be the $(-1)$-curve in $S\cong\mathbb{F}_1$.
Then it follows from \cite{MoMu81,MoMu83} that
there exists a birational morphism $f\colon X_{\mathbb{C}} \to\mathbb{P}^1 \times\mathbb{F}_1$ that is the blowup of the curve $C$.
Let $E$ be the $f$-exceptional divisor. Then it follows from \cite{Matsuki} that $E$ is defined over $\Bbbk$.
Since $\mathbb{F}_1$ does not have non-trivial forms over $\Bbbk$,
the 3-fold $X$ is the blowup of $\mathscr{C}\times\mathbb{F}_1$, where $\mathscr{C}$ is a conic in $\mathbb{P}^2$.
However, the image of $E$ in $\mathscr{C}\times\mathbb{F}_1$ is a curve that is contained in a fiber of the natural projection
$\mathscr{C}\times\mathbb{F}_1\to\mathscr{C}$, which implies that $\mathscr{C}$ has a $\Bbbk$-point.
Therefore, $X$ is $\Bbbk$-birational to $\mathbb{P}^1 \times\mathbb{F}_1$ and, in particular, it has a $\Bbbk$-point.
\end{proof}

\begin{remark}
As seen in the above argument, all $\Bbbk$-forms of strictly K-semistable Fano 3-folds except
for those belonging to Family \textnumero 2.11 are rational over $\Bbbk$.
\end{remark}

\section{Pointless K-polystable Fano 3-folds}
\label{section:pointless-3-folds}

In this section, we work through the 18 families of Fano 3-folds that contain K-polystable elements but K-polystability is not known for all elements. They also exhibit the phenomenon that their smooth elements do not always admit $\Bbbk$-points. The next section deals with examples in each case without $\Bbbk$-points.

\textbf{Strategy of the proof in this section:} In each case, we denote by $X$ the smooth Fano 3-fold defined over $\Bbbk\subset\mathbb{C}$, which we assume has no $\Bbbk$-rational points, and by $X_\mathbb{C}$ its geometric model, which we aim to prove is K-polystable. The argument of the proof starts by assuming $X_\mathbb{C}$ is not K-polystable. Suppose $X_{\mathbb{C}}$ is not K-polystable. Then it follows from the valuative criterion for K-stability (\cite{Fujita2019,Li2017}) that $\delta(X_{\mathbb{C}})\leqslant 1$. Since $\delta(X_{\mathbb{C}})<\frac{4}{3}$, it follows from \cite[Theorem 1.2]{LXZ} that there exists a prime divisor $\mathbf{F}$ over $X_{\mathbb{C}}$ that computes $\delta$:
$$
\delta(X_{\mathbb{C}})=\frac{A_X(\mathbf{F})}{S_X(\mathbf{F})}.
$$
Moreover, if $\delta(X_{\mathbb{C}})<1$ it follows from \cite[Theorem 4.4]{Zhuang2021} that $\mathbf{F}$ is defined over $\Bbbk$. If $\delta(X_{\mathbb{C}})=1$ we can also assume $\mathbf{F}$ is defined over $\Bbbk$ by \cite[Corollary 4.14]{Zhuang2021}, because $X_{\mathbb{C}}$ is not K-polystable.
Let $Z\subset X$ be the center of the divisor $\mathbf{F}$.
Then  $Z$ is not a surface by \mbox{\cite[Theorem~3.17]{Book}}.
On the other hand, since $X(\Bbbk)=\varnothing$, we conclude that $Z$ is a geometrically irreducible curve defined over $\Bbbk$. In each case, we get a contradiction, often by estimating lower bounds of greater than $1$ for $\delta_C(X_\mathbb{C})$ for all curves $C\subset X_\mathbb{C}$, and noting that
\[\delta_C(X_\mathbb{C})=\inf\frac{A_X(E)}{S_X(E)},\]
where infimum runs over all prime divisors over $X_\mathbb{C}$ whose centers contain $C$.

\begin{lemma}
\label{lemma:pointless-1-9}
Suppose that $X$ is contained in Family \textnumero 1.9 and $X(\Bbbk)=\varnothing$.
Then $X_{\mathbb{C}}$ is K-polystable.
\end{lemma}

\begin{proof}

Let $S$ be a very general surface in $|-K_{X_{\mathbb{C}}}|$.
Then $S$ is a smooth K3 surface with $\mathrm{Pic}(S)=\mathbb{Z}[-K_{X_{\mathbb{C}}}]$, so that it follows from \cite{Knutsen} and \cite[Theorem~A]{AbbanZhuangSeshadri} that
$$
\delta\big(S,-K_{X_{\mathbb{C}}}\vert_{S}\big)\geqslant\frac{4}{5},
$$
so that
$$
\delta(X_{\mathbb{C}})\leqslant 1 \leqslant\frac{4}{3}\delta(S,- K_{X_{\mathbb{C}}}\vert_{S}).
$$
We are now in a position to apply \cite[Corollary 5.6]{AbbanZhuangSeshadri} to $\mathbf{F}$ over $X$. Note that \cite[Corollary 5.6]{AbbanZhuangSeshadri} applies over $\Bbbk$ by tautology.
Hence, it follows that at least one of the following two cases holds:
\begin{enumerate}
\item either there exists an effective $\mathbb{Q}$-divisor $D$ on the 3-fold $X$ such that $D\sim_{\mathbb{Q}} -K_X$
and $Z$ is a center of non-log canonical singularities of the log pair $(X,\frac{1}{2}D)$, so, in particular,
the log pair $(X,\frac{1}{2}D)$ is not log canonical along the curve $Z$;

\item or there exists a mobile linear system $\mathcal{M}\subset |-nK_{X}|$ such that
$Z$ is a center of non-klt singularities of the log pair $(X,\frac{1}{2n}\mathcal{M})$.
\end{enumerate}
In the second case, if $M_1$ and $M_2$ are general surfaces in $\mathcal{M}$, then it follows from
Corti's inequality \cite[Theorem 3.1]{Corti2000}, see also \mbox{\cite[Theorem A.22]{Book}}, that
$
M_1\cdot M_2=mZ+\Delta
$
for some positive integer $m\geqslant 16n^2$ and some effective one-cycle $\Delta$ on the 3-fold $X$,
which implies that
$$
18n^2=-K_X\cdot M_1\cdot M_2=m(-K_X)\cdot Z+(-K_X)\cdot\Delta\geqslant m(-K_X)\cdot Z\geqslant 16n^2(-K_X)\cdot Z,
$$
so that $-K_X\cdot Z=1$, which is impossible, since $-K_X$ is very ample \cite{IsPr99} and $X(\Bbbk)=\varnothing$.
So, we are left to analyze the first case.

Now, arguing as in the proof of \cite[Theorem~1.52]{Book}, we can replace the effective $\mathbb{Q}$-divisor $D$
with another $\mathbb{Q}$-divisor $D^\prime$ on the 3-fold $X$ such that
$$
D^\prime\sim_{\mathbb{Q}} D\sim_{\mathbb{Q}} -K_X,
$$
the log pair $(X,\lambda D^\prime)$ has log canonical singularities for some positive rational number $\lambda<\frac{1}{2}$
such that the singularities of the log pair $(X,\lambda D^\prime)$ are not klt (not Kawamata log terminal),
and the locus $\mathrm{Nklt}(X,\lambda D^\prime)$ is geometrically irreducible,
and consists of a minimal center of log canonical singularities of the pair $(X_{\mathbb{C}},\lambda D_{\mathbb{C}}^\prime)$.
Here, we implicitly used Nadel's vanishing theorem and Koll\'ar--Shokurov connectedness theorem, see \cite[Appendix~A.1]{Book}.

Set $C=\mathrm{Nklt}(X,\lambda D^\prime)$. Then $C$ is not a surface, since $\mathrm{Pic}(X)$ is generated by $-K_X$.
Similarly, as above, we see that $C$ is not a point, because  $X(\Bbbk)=\varnothing$.
Thus, we see that $C$ is a geometrically irreducible curve.
Then it follows from the proof of  \cite[Theorem~1.52]{Book} that $C$ is a smooth geometrically rational curve with
$$
-K_X\cdot C\leqslant \frac{2}{1-\lambda}<4.
$$
Here, we have implicitly used basic properties of minimal centers of log canonical singularities and Kawamata's subadjunction theorem \cite{Kawamata1997,Kawamata1998}.

As above, we see that $-K_X\cdot C\ne 1$, because $-K_X$ is very ample.
Similarly, we see that $-K_X\cdot C\ne 3$ by Lemma~\ref{lemma:SB}, because $C$ is geometrically rational and $C(\Bbbk)=\varnothing$.
Hence, we conclude that $-K_X\cdot C=2$.

Starting from now, we work with the geometric model $X_{\mathbb{C}}$ and with abuse of notation we write $X$ and $C$ for their geometric models over $\mathbb{C}$.
Moreover, we identify $X$ with its anticanonical embedding in $\mathbb{P}^{11}$, so $C$ is a conic in $X$.
Let $\phi\colon\widetilde{X}\to X$ be the blowup of the conic $C$, and let $E$ be the $\phi$-exceptional surface.
Then it follows from \cite[(2.13.2)]{Takeuchi} or from \cite[Theorem~4.4.11]{IsPr99} and \cite[Corollary~4.4.3]{IsPr99}
that the linear system $|-K_{\widetilde{X}}|$ is base point free, and the linear system $|-K_{\widetilde{X}}-E|$ gives a birational map
$\chi\colon X\dasharrow \mathbb{P}^2$ such that we have the following commutative diagram.
$$
\xymatrix@R=1em{
&\widetilde{X}\ar@{->}[rd]_{\alpha}\ar@{->}[ldd]_{\phi}\ar@{-->}[rr]^{\zeta}&&V\ar@{->}[rdd]^{\pi}\ar@{->}[ld]^{\beta} &\\%
&&Y&&\\
X\ar@{-->}[rrrr]^{\chi}&&&&\mathbb{P}^2}
$$
where $\alpha$ is a small birational morphism given by the linear system $|-K_{\widetilde{X}}|$,
$Y$ is a Fano 3-fold with Gorenstein non-$\mathbb{Q}$-factorial terminal singularities such that $-K_{Y}^3=12$,
the map $\zeta$ is a pseudo-isomorphism that flops the curves contracted by $\alpha$,
$\beta$ is a small birational morphism given by the linear systems $|-K_{V}|$, and $\pi$ is a conic bundle.
This shows that the cone of effective divisors of the 3-fold $\widetilde{X}$ is generated by the divisors $E$ and $-K_{\widetilde{X}}-E\sim \phi^*(-K_X)-2E$.
On the other hand, we have
$$
\mathrm{mult}_{C}\big(D^\prime\big)\geqslant\frac{1}{\lambda}>2.
$$
This is a well-known fact, see for example \cite[Proposition~9.5.13]{Lazarsfeld}.
Thus, if $\widetilde{D}^\prime$ is the strict transform of the divisor $D^\prime$ on the 3-fold $\widetilde{X}$, then
$$
\widetilde{D}^\prime\sim_{\mathbb{Q}}\phi^*\big(-K_X\big)-\mathrm{mult}_{C}\big(D^\prime\big)E
\sim_{\mathbb{Q}}\big(-K_{\widetilde{X}}-E\big)-\big(\mathrm{mult}_{C}\big(D^\prime\big)-2\big)E,
$$
which is a contradiction, since $\mathrm{mult}_{C}(D^\prime)>2$.
\end{proof}

\begin{lemma}
\label{lemma:pointless-1-10}
Suppose that $X$ is contained in Family \textnumero 1.10 and $X(\Bbbk)=\varnothing$.
Then $X_{\mathbb{C}}$ is K-polystable.
\end{lemma}

\begin{proof}
Arguing as in the proof of Lemma~\ref{lemma:pointless-1-9}, we see that
there exists an effective $\mathbb{Q}$-divisor $D$ on $X$ and a positive rational number $\lambda<\frac{1}{2}$ such that $D\sim_{\mathbb{Q}} -K_X$,
the pair $(X,\lambda D)$ has log canonical singularities,
and the locus $\mathrm{Nklt}(X,\lambda D)$ consists of a geometrically irreducible smooth curve such that $-K_X\cdot C=1$.
Then, as in the proof of Lemma~\ref{lemma:pointless-1-9}, we see that
\begin{equation}
\label{equation:bigger-than-2}
\mathrm{mult}_C\big(D\big)\geqslant\frac{1}{\lambda}>2.
\end{equation}
However, unlike the proof of Lemma~\ref{lemma:pointless-1-9}, we cannot immediately use \eqref{equation:bigger-than-2} to derive a contradiction.
Nevertheless, we are still able to obtain a contradiction via a more delicate analysis of the geometry of the 3-fold $X$ and the properties of the log pair $(X,\lambda D)$.
As in the proof of Lemma~\ref{lemma:pointless-1-9}, let us identify $X$ with its anticanonical embedding in $\mathbb{P}^{13}$,
so that $C$ is a conic in $X$.

Let $\phi\colon\widetilde{X}\to X$ be the blowup along $C$, and let $E$ be the $\phi$-exceptional surface.
Then it follows from \cite[(2.13.2)]{Takeuchi} or from \cite[Theorem~4.4.11]{IsPr99} and \cite[Corollary~4.4.3]{IsPr99}
that the linear system $|-K_{\widetilde{X}}|$ is base point free, and the linear system $|-K_{\widetilde{X}}-E|$ gives a birational map
$\chi\colon X\dasharrow Q$, where $Q$ is a smooth (pointless over $\Bbbk$) quadric 3-fold in $\mathbb{P}^{4}$.
Moreover, we have the following commutative diagram
$$
\xymatrix@R=1em{
&\widehat{Q}\ar@{->}[ldd]_{\pi}\ar@{->}[rd]_{\beta}&& \widetilde{X}\ar@{->}[ld]^{\alpha}\ar@{->}[rdd]^{\phi}\ar@{-->}[ll]_{\zeta}&\\%
&&Y&&\\
Q&&&&X\ar@{-->}[llll]_{\chi}}
$$
where $\alpha$ is a small birational morphism given by $|-K_{\widetilde{X}}|$,
$Y$ is a Fano 3-fold with Gorenstein non-$\mathbb{Q}$-factorial terminal singularities with $-K_{Y}^3=16$,
the map $\zeta$ is a pseudo-isomorphism that flops the curves contracted by $\alpha$,
$\pi$ is the blowup of a smooth rational sextic curve $\Gamma\subset Q$,
and $\beta$ is a small birational morphism given by the linear systems $|-K_{\widetilde{Q}}|$.
Let $F$ be the $\pi$-exceptional surface, let $\widetilde{F}$ be its strict transform on $\widetilde{X}$, and let $\overline{F}=\phi(\widetilde{F})$.
Then $\widetilde{F}\sim \phi^*(-2K_X)-5E$,
and the cone of effective divisors of the 3-fold $\widetilde{X}$ is generated by the surfaces $F$ and $E$.
Moreover, the divisors $-K_{\widetilde{X}}-E$ and $\phi^*(-K_X)$ generate the movable cone of divisors on $\widetilde{X}$,
so it follows from \eqref{equation:bigger-than-2} that $F\subset\mathrm{Supp}(D)$.
Moreover, arguing as in the proof of \cite[Lemma~A.34]{Book} and using $\mathrm{Pic}(X)=\mathbb{Z}[-K_X]$,
we see that the log pair $(X,\frac{\lambda}{2}F)$ is also non-klt along $C$.
In particular, we see that the log pair $(X,\frac{1}{4}F)$ is not log canonical along the curve $C$.

Let us show that the latter is impossible. Observe that
$$
K_{\widetilde{X}}+\frac{1}{4}\widetilde{F}+\frac{1}{4}E\sim_{\mathbb{Q}}\phi^*\big(K_X+\frac{1}{4}F\big),
$$
which implies that $E$ contains an irreducible curve $Z$ with $\phi(Z)=C$,
and the log pair $(\widetilde{X}, \frac{1}{4}\widetilde{F}+\frac{1}{4}E)$ is not log canonical along $Z$.
Then the log pair $(\widetilde{X}, \frac{1}{4}\widetilde{F}+E)$ is also not log canonical along $Z$,
so that it follows from Inversion of Adjunction \cite[Theorem 5.50]{KollarMori1998} that
the log pair $(E,\frac{1}{4}\widetilde{F}\vert_{E})$ is also not log canonical along $Z$.
But this simply means that $(\widetilde{F}\cdot E)_{Z}>4$.
Now, intersecting the restriction $\widetilde{F}\vert_{E}$ with a general (geometric) fiber of the projection $E\to C$,
we immediately see that $(\widetilde{F}\cdot E)_{Z}=5$ and $Z$ is a section of this projection.
In particular, we see that $Z$ is a geometrically irreducible curve.

Now, we recall from \cite[Theorem~1.1.1]{KuznetsovProkhorovShramov} and \cite[Corollary~2.1.6]{KuznetsovProkhorovShramov}
that the normal bundle of the smooth conic $C_{\mathbb{C}}\simeq\mathbb{P}^1$ in $X_{\mathbb{C}}$ is isomorphic either to $\mathcal{O}_{\mathbb{P}^1}\oplus\mathcal{O}_{\mathbb{P}^1}$
or to $\mathcal{O}_{\mathbb{P}^1}(1)\oplus\mathcal{O}_{\mathbb{P}^1}(-1)$, so that either $E_{\mathbb{C}}\simeq\mathbb{P}^1\times\mathbb{P}^1$ or $E_{\mathbb{C}}\simeq\mathbb{F}_2$.
Moreover, it follows from elementary computations or from \cite[Lemma~4.4.4]{IsPr99} that the restriction $-K_{\widetilde{X}}\vert_{E}$ is ample.
In particular, the birational map $\zeta$ is an isomorphism in a neighborhood of a general point of the curve $Z$.
This gives $\big(F\cdot \widehat{E}\big)_{\widehat{Z}}=\big(\widetilde{F}\cdot E\big)_{Z}=5$,
where $\widehat{E}$ and $\widehat{Z}$ are strict transforms on $\widehat{Q}$ of the surface $E$ and the curve $Z$, respectively.
On the other hand, we have $\widehat{E}\sim\pi^{*}(2H)-F$,
where $H$ is the class of a hyperplane section of the quadric $Q$.
Moreover, we have $F_{\mathbb{C}}\simeq\mathbb{F}_n$ for some $n\in\mathbb{Z}_{>0}$.
Let $\mathbf{s}$ be a section of the natural projection $F_{\mathbb{C}}\to\Gamma_{\mathbb{C}}$ such that $\mathbf{s}^2=-n$,
and let $\mathbf{f}$ be a geometric fiber of this projection.
Then $5\widehat{Z}_{\mathbb{C}}+\Delta=\widehat{E}_{\mathbb{C}}\big\vert_{F_{\mathbb{C}}}\sim \mathbf{s}+a\mathbf{f}$
for some effective divisor $\Delta$ on the surface $F_{\mathbb{C}}$ and some non-negative integer $a$. This gives $\widehat{Z}_{\mathbb{C}}\cdot \mathbf{f}=0$.
Since the curve $Z_{\mathbb{C}}$ is irreducible, we see that $\widehat{Z}_{\mathbb{C}}\sim \mathbf{f}$, which implies that $\pi(\widehat{Z})$ is a $\Bbbk$-point in $Q$.
Then $X(\Bbbk)\ne\varnothing$ by Lemma~\ref{lemma:nonempty}, which is a contradiction.
\end{proof}

\begin{lemma}
\label{lemma:pointless-2-5}
Suppose that $X$ is contained in Family \textnumero 2.5 and $X(\Bbbk)=\varnothing$.
Then $X_{\mathbb{C}}$ is K-polystable.
\end{lemma}

\begin{proof}
It follows from \cite{MoMu81,MoMu83} and Lemma~\ref{lemma:Prokhorov} that there exists the following diagram
$$
\xymatrix@R=1em{
&X\ar@{->}[ld]_{\pi}\ar@{->}[rd]^{\phi}&&\\%
V&&\mathbb{P}^1}
$$
where $V$ is a~smooth cubic 3-fold in $\mathbb{P}^4$,
the morphism $\pi$ is the blowup of a smooth plane cubic curve (defined over $\Bbbk$),
and $\phi$ is a morphism whose fibers are normal cubic surfaces.

Let $p$ be any point in $Z$. Then  $\delta_p(X)\leqslant 1$, which we aim to contradict.
Let $E$ be the $\pi$-exceptional surface, and let $S$ be the fiber of $\phi$ that contains $p$.
Then $S$ is a possibly singular irreducible cubic surface with at worst isolated singularities,
so either $S$ has  Du Val singularities or it is a cone over a smooth cubic curve.
If $S$ is Du Val, then it follows from \cite[Lemma 2.1]{CheltsovDenisovaFujita} that
\begin{equation}
\label{equation:2-5}
1\geqslant\delta_p(X)\geqslant\left\{\aligned
&\mathrm{min}\Big\{\frac{16}{11},\frac{16}{15}\delta_p(S)\Big\}\ \text{if $p\not\in E$}, \\
&\mathrm{min}\Big\{\frac{16}{11},\frac{16\delta_{p}(S)}{\delta_{p}(S)+15}\Big\}\ \text{if $p\in E$}.
\endaligned
\right.
\end{equation}
If $\phi(Z)=\mathbb{P}^1$, we may assume that $p$ is a general point in $Z$ so that $S$ is smooth.
In this case, we know from \cite{AbbanZhuangSeshadri} or \cite[Lemma 2.13]{Book} that $\delta_p(S)\geqslant\delta(S)\geqslant \frac{3}{2}$,
which gives the desired contradiction.
Hence, we conclude that $\phi(Z)$ is a point in $\mathbb{P}^1$, so $Z$ is contained in $S$. In that case, the surface $S$ is defined over $\Bbbk$, since $Z$ is defined over $\Bbbk$.
Note that $S(\Bbbk)=\varnothing$ as $X(\Bbbk)=\varnothing$.
In particular, we see that the surface $S$ is not a cone, since otherwise its vertex would be defined over $\Bbbk$. Then, it follows from \cite[Lemma 2.2]{CheltsovDenisovaFujita} that $Z\not\subset E$,
so we may assume that $p\not\in E$ either.
Thus, it follows from \eqref{equation:2-5} that
$$
\delta(S)\leqslant\delta_p(S)\leqslant \frac{15}{16}.
$$
On the other hand, all possible values of $\delta(S)$ have been computed in \cite{Denisova-cubic}.
In particular, since $\delta(S)\leqslant \frac{15}{16}$, it follows from \cite[Main~Theorem]{Denisova-cubic} that
the cubic surface $S$ is singular, and at least one of its $\mathbb{C}$-singular points is not a singular points of type $\mathbb{A}_1$ or $\mathbb{A}_2$.
Now, using the classification of Du Val cubic surfaces \cite{BruceWall1979}, we see that such singular point is unique,
hence defined over $\Bbbk$, which is impossible since  $S(\Bbbk)=\varnothing$.
\end{proof}

\begin{lemma}
\label{lemma:pointless-2-10}
Suppose that $X$ is contained in Family \textnumero 2.10 and $X(\Bbbk)=\varnothing$.
Then $X_{\mathbb{C}}$ is K-polystable.
\end{lemma}

\begin{proof}
It follows from \cite{MoMu81,MoMu83} and Lemma~\ref{lemma:Prokhorov} that there exists the following diagram
$$
\xymatrix@R=1em{
&X\ar@{->}[ld]_{\pi}\ar@{->}[rd]^{\phi}&&\\%
V&&\mathbb{P}^1}
$$
where $V$ is a~smooth complete intersection of two quadrics in $\mathbb{P}^5$,
the morphism $\pi$ is the blowup of a smooth quartic elliptic curve (defined over $\Bbbk$),
and $\phi$ is a morphism whose fibers are normal complete intersection of two quadrics in $\mathbb{P}^4$.

Let $p$ be a general point in $Z_{\mathbb{C}}$, and let $S$ be the fiber of $\phi$ that contains $p$.
If $S$ has at worst Du Val singularities, then it follows from the proof of \cite[Lemma 2.1]{CheltsovDenisovaFujita} that
\begin{equation}
\label{equation:2-10}
1\geqslant\delta_p(X_{\mathbb{C}})\geqslant\left\{\aligned
&\mathrm{min}\Big\{\frac{16}{11},\frac{16}{15}\delta_p(S)\Big\}\ \text{if $p\not\in E$}, \\
&\mathrm{min}\Big\{\frac{16}{11},\frac{16\delta_{p}(S)}{\delta_{p}(S)+15}\Big\}\ \text{if $p\in E$},
\endaligned
\right.
\end{equation}
where $E$ is the exceptional surface of the complexification of the blowup $\pi$.
On the other hand, if $\phi(Z)=\mathbb{P}^1$, then $S$ is smooth and
we know from \cite[Lemma 2.12]{Book} that $\delta_p(S)\geqslant\delta(S)\geqslant \frac{3}{4}$,
which contradicts \eqref{equation:2-10}.
Thus, we conclude that $\phi(Z)$ is a point in $\mathbb{P}^1$, so $Z$ is contained in $S$ and, in particular,
the surface $S$ is defined over $\Bbbk$, since $Z$ is defined over $\Bbbk$.
then the surface $S$ is not a cone, since otherwise its vertex would be defined over $\Bbbk$.
If $\delta(S)\geqslant 1$, it follows from \eqref{equation:2-10} that $p\in E$ and
$$
1\geqslant\frac{A_X(\mathbf{F})}{S_X(\mathbf{F})}\geqslant\delta_p(X_{\mathbb{C}})\geqslant\frac{16\delta_{p}(S)}{\delta_{p}(S)+15}\geqslant 1
$$
which implies that $\delta_p(S)=\delta(S)=1$.
In this case, \cite{AbbanZhuang} and the proof of \cite[Lemma 2.1]{CheltsovDenisovaFujita} give a contradiction, since $Z\subset S$.
Hence, we see that $\delta(S)<1$.

Recall that $S$ has Du Val singularities.
The list of all possible singularities that can occur on $S$ are listed in \cite{CorayTsfasman1988}.
Moreover, all possible values of $\delta(S)$ for singular surfaces are computed in \cite{Denisova-DP}.
Now, using the classification of singularities in \cite{CorayTsfasman1988} and the computations in \cite{Denisova-DP},
we see that the inequality $\delta(S)<1$ implies that $S$ has at least one singular point whose type is different from the other singular points of $S$ (if any).
Hence, this singular point must be defined over $\Bbbk$, which contradicts $X(\Bbbk)=\varnothing$.
\end{proof}

\begin{lemma}
\label{lemma:pointless-2-12}
Suppose that $X$ is contained in Family \textnumero 2.12 and $X(\Bbbk)=\varnothing$.
Then $X_{\mathbb{C}}$ is K-polystable.
\end{lemma}

\begin{proof}
The required assertion is \cite[Corollary~9]{CheltsovLiMauPinardin}.
\end{proof}

\begin{lemma}
\label{lemma:pointless-2-13}
Suppose that $X$ is contained in Family \textnumero 2.13 and $X(\Bbbk)=\varnothing$.
Then $X_{\mathbb{C}}$ is K-polystable.
\end{lemma}

\begin{proof}
Using \cite{MoMu81,MoMu83} and Lemma~\ref{lemma:Prokhorov}, we see that there is a morphism
$f\colon X\to Q$ such that $Q$ is a form of a smooth quadric 3-fold in $\mathbb{P}^4$,
and $f$ is the blowup of a smooth geometrically irreducible curve $C\subset Q$ such that $C$ has genus $2$ and $-K_Q\cdot C=18$.
Over $\mathbb{C}$, we have the following diagram
$$
\xymatrix@R=1em{
&X_{\mathbb{C}}\ar@{->}[ld]_{f_{\mathbb{C}}}\ar@{->}[rd]^{\pi}&&\\%
Q_{\mathbb{C}}&&\mathbb{P}^2}
$$
where $\pi$ is a conic bundle with discriminant curve a quartic in $\mathbb{P}^2$.
By Lemma~\ref{lemma:SB}, we may assume that the conic bundle $\pi$ is defined over $\Bbbk$.
Let $\ell$ be a line in $\mathbb{P}^2$ and $S=\pi^*(\ell)$. Then $f_{\mathbb{C}}(S_{\mathbb{C}})$ is cut out on the quadric $Q_{\mathbb{C}}$ by another quadric hypersurface in $\mathbb{P}^4$.
In particular, the linear system $|-K_Q-f(S)|$ gives an embedding $Q\hookrightarrow\mathbb{P}^4$, which implies that $Q$ is a smooth quadric 3-fold.
Note that $Q(\Bbbk)=\varnothing$ by Lemma~\ref{lemma:nonempty}.

We may have the following two cases for the curve $Z$:
\begin{enumerate}
\item $\pi(Z)$ is a point in $\mathbb{P}^2$;
\item $\pi(Z)$ is a curve in $\mathbb{P}^2$.
\end{enumerate}
In the first case, $Z$ is a fiber of the conic bundle $\pi$, since otherwise $f(Z)$ would be a line in $Q$,
which would contradict $Q(\Bbbk)=\varnothing$.
In this case, we let $\ell$ be a general line in $\mathbb{P}^2$ that passes through the point $\pi(Z)$.
In the second case, let $\ell$ be a general line in $\mathbb{P}^2$.
As before, let $S=\pi^*(\ell)$, which is smooth.
It follows from the adjunction formula that $S$ is a del Pezzo surface of degree $4$.

Starting from now, we work with geometrical models of $X$, $S$ and $Z$.
For simplicity, we denote them by $X$, $S$ and $Z$, respectively.
Let $p$ be a point in $Z\cap S$,
and let $A$ be the fiber of the conic bundle $\pi$ that passes through $p$.
Then $A$ is smooth. Note that $A=Z$ in the case when $\pi(Z)$ is a point.
By assumption, we have $\delta_P(X)\leqslant 1$.
We apply Abban--Zhuang method \cite{AbbanZhuang} to the flag $p\in A\subset S$ to show that $\delta_P(X)\geqslant\frac{80}{77}$,
which would imply the desired contradiction.

Let $E$ be the $f$-exceptional surface, and let $H=f^*(\mathcal{O}_{Q}(1))$.
Then $-K_X\sim 3H-E$ and $S\sim 2H-E$.
Let $u$ be a non-negative real number. Then
$$
-K_X-uS\sim_{\mathbb{R}} (3-2u)H+(u-1)E,
$$
which implies that the divisor $-K_X-uS$ is pseudoeffective if and only if $u\leqslant \frac{3}{2}$.
For $u\leqslant\frac{3}{2}$, let us denote by $P(u)$ the positive part of Zariski decomposition of the divisor $-K_X-uS$,
and let us denote by $N(u)$ its negative part. Then
$$
P(u)=\left\{\aligned
&(3-2u)H+(u-1)E\ \text{if $0\leqslant u\leqslant 1$}, \\
&(3-2u)H\ \text{if $1\leqslant u\leqslant \frac{3}{2}$},
\endaligned
\right.
$$
and
$$
N(u)=\left\{\aligned
&0\ \text{if $0\leqslant u\leqslant 1$}, \\
&(u-1)E\ \text{if $1\leqslant u\leqslant \frac{3}{2}$}.
\endaligned
\right.
$$
This gives
$$
P(u)\big\vert_{S}=\left\{\aligned
&-K_S+(1-u)A\ \text{if $0\leqslant u\leqslant 1$}, \\
&(3-2u)(-K_S)\ \text{if $1\leqslant u\leqslant \frac{3}{2}$}.
\endaligned
\right.
$$
Now, integrating $(P(u))^3$, we get $S_X(S)=\frac{41}{80}$. Then, using \cite[Theorem~3.3]{AbbanZhuang} and \cite[Corollary 1.102]{Book}, we get
$$
\delta_p(X)\geqslant \min\left\{\frac{1}{S_X(S)},\inf_{\substack{F/S\\ p\in C_S(F)}}\frac{A_S(F)}{S(W_{\bullet,\bullet}^S;F)}\right\}=
\min\left\{\frac{80}{41},\inf_{\substack{F/S\\ p\in C_S(F)}}\frac{A_S(F)}{S(W_{\bullet,\bullet}^S;F)}\right\}
$$
where the infimum is taken over all prime divisors $F$ over $S$ for which $p$ contained in the center $C_S(F)$ of the divisor $F$ on $S$.
The value $S(W_{\bullet,\bullet}^S;F)$ can be computed using \cite[Corollary~1.108]{Book} as follows:
$$
S\big(W_{\bullet,\bullet}^S; F\big)=
\frac{3}{(-K_X)^3}\int_1^{\frac{3}{2}}\big(P(u)\big\vert_{S}\big)^2(u-1)\mathrm{ord}_{F}\big(E\big\vert_{S}\big)du+
\frac{3}{(-K_X)^3}\int_{0}^{\frac{3}{2}}\int_0^\infty \mathrm{vol}\big(P(u)\big\vert_{S}-vF\big)dvdu.
$$
Recall that $(-K_X)^3=20$.
Hence, if $\delta_p(X)<\frac{80}{77}$, then there exists a prime divisor $F$ over the surface $S$ such that
\begin{equation}
\label{equation:2-13}
\frac{3}{20}\int_1^{\frac{3}{2}}\big(P(u)\big\vert_{S}\big)^2(u-1)\mathrm{ord}_{F}\big(E\big\vert_{S}\big)du+
\frac{3}{20}\int_{0}^{\frac{3}{2}}\int_0^\infty \mathrm{vol}\big(P(u)\big\vert_{S}-vF\big)dvdu>\frac{77}{80}A_S(F).
\end{equation}
Let us show that this is impossible.
First, we observe that $E\vert_{S}$ is a smooth curve.
Hence, the pair $(S,E\vert_{S})$ is log canonical. This gives
$$
\mathrm{ord}_{F}\big(E\big\vert_{S}\big)\leqslant A_{S}(F).
$$
Thus, we can estimate the first term in the left hand side of \eqref{equation:2-13} as follows:
\begin{multline*}
\frac{3}{20}\int_1^{\frac{3}{2}}\big(P(u)\big\vert_{S}\big)^2(u-1)\mathrm{ord}_{F}\big(E\big\vert_{S}\big)du
\leqslant \frac{3}{20}\int_1^{\frac{3}{2}}\big(P(u)\big\vert_{S}\big)^2(u-1)A_{S}(F)du=\\
=\frac{3A_{S}(F)}{20}\int_1^{\frac{3}{2}}(3-2u)^2(-K_S)^2(u-1)du=\frac{3A_{S}(F)}{5}\int_1^{\frac{3}{2}}(3-2u)^2(u-1)du=\frac{A_{S}(F)}{80}.
\end{multline*}
To estimate the second term in the left hand side of \eqref{equation:2-13}, let $L=-K_S+(1-u)A$.
Then $L$ is an ample divisor on $S$ when $u\in[0,1]$.
In this case, it follows from \cite[Lemma 23]{CheltsovFujitaKishimotoOkada} that
$$
\delta_{p}(S,L)\geqslant\frac{24}{u^2-10u+28},
$$
where $\delta_{p}(S,L)$ is the (local) $\delta$-invariant of the polarized pair $(S,L)$ defined in  \cite[Appendix A]{CheltsovFujitaKishimotoOkada}.
Note also that it follows from \cite[Lemma 2.12]{Book} that
$$
\delta_{p}(S,-K_S)=\delta_{p}(S)\geqslant\delta(S)=\frac{4}{3},
$$
since $S$ is a smooth del Pezzo surface of degree $4$.
Now, using these estimates, we have
\begin{multline*}
\frac{3}{20}\int_{0}^{\frac{3}{2}}\int_0^\infty\mathrm{vol}\big(P(u)\big\vert_{S}-vF\big)dvdu=\\
=\frac{3}{20}\int_{0}^{1}\int_0^\infty\mathrm{vol}\big(-K_S+(1-u)A-vF\big)dvdu+
\frac{3}{20}\int_{1}^{\frac{3}{2}}\int_0^\infty\mathrm{vol}\big((3-2u)(-K_S)-vF\big)dvdu=\\
=\frac{3}{20}\int_{0}^{1}\int_0^\infty\mathrm{vol}\big(-K_S+(1-u)A-vF\big)dvdu+
\frac{3}{20}\int_{1}^{\frac{3}{2}}(3-2u)^3\int_0^\infty\mathrm{vol}\big((-K_S)-vF\big)dvdu\leqslant\\
\leqslant\frac{3}{20}\int_{0}^{1}L^2\frac{A_S(F)}{\frac{24}{u^2-10u+28}}du+
\frac{3}{20}\int_{1}^{\frac{3}{2}}(3-2u)^3(-K_S)^2\frac{A_S(F)}{\delta(S)}du=\\
=A_S(F)\Bigg(\frac{3}{20}\int_{0}^{1}(8-4u)\frac{u^2-10u+28}{24}du+
\frac{9}{20}\int_{1}^{\frac{3}{2}}\int_0^\infty(3-2u)^3du\Bigg)=\frac{19}{20}A_S(F).
\end{multline*}
This implies that the left hand side of \eqref{equation:2-13} does not exceed $\frac{77}{80}A_S(F)$,
which is a contradiction.
\end{proof}

\begin{lemma}
\label{lemma:pointless-2-16}
Suppose that $X$ is contained in Family \textnumero 2.16 and $X(\Bbbk)=\varnothing$.
Then $X_{\mathbb{C}}$ is K-polystable.
\end{lemma}

\begin{proof}
Using \cite{MoMu81,MoMu83} and Lemma~\ref{lemma:Prokhorov},
we see that there is a morphism $f\colon X\to V$ such that $V$ is a form of a smooth complete intersection of two quadrics in $\mathbb{P}^5$,
and $f$ is the blowup of a smooth geometrically rational curve $C\subset V$ with $-K_V\cdot C=4$.
Over $\mathbb{C}$, the curve $C_{\mathbb{C}}$ is a smooth conic in $V_{\mathbb{C}}\subset \mathbb{P}^5$,
and we have the following commutative diagram:
$$
\xymatrix@R=1em{
&X_{\mathbb{C}}\ar@{->}[ld]_{f_{\mathbb{C}}}\ar@{->}[rd]^{\pi}&&\\%
V_{\mathbb{C}}\ar@{-->}[rr]&&\mathbb{P}^2}
$$
where $\pi$ is a conic bundle whose discriminant curve is a (possibly singular) reduced quartic curve in $\mathbb{P}^2$,
and the dashed arrow is induced by the linear projection from the plane in $\mathbb{P}^5$ that contains $C_{\mathbb{C}}$.
Now, arguing exactly as in the proof of Lemma~\ref{lemma:pointless-2-13}, we see that the conic bundle $\pi$ is defined over the field~$\Bbbk$,
and $V$ is a (pointless) complete intersection of two quadrics in $\mathbb{P}^5$.

Arguing as in the proof of Lemma~\ref{lemma:pointless-2-13},
we see that either $Z$ is a smooth fiber of the conic bundle $\pi$, or $\pi(Z)$ is a curve in $\mathbb{P}^2$.
Moreover, $f(Z)$ is a curve in $V$, since otherwise $f(Z)$ would be a $\Bbbk$-point,
but $V(\Bbbk)=\varnothing$ by Lemma~\ref{lemma:nonempty}.

Starting from now, we will work exclusively with geometrical models of $X$ and $Z$,
which (for simplicity) we will denote by $X$ and $Z$, respectively.
Let $E$ be the $f$-exceptional surface, and let $H=f^*(\mathcal{O}_{V}(1))$.
Then $-K_X\sim 2H-E$, the conic bundle $\pi$ is given by the linear system $|H-E|$,
the divisors $E$ and $H-E$ generate the cone of effective divisors of the 3-fold $X$,
the nef cone of the 3-fold $X$ is generated by the divisors $H$ and $H-E$,
and the Mori cone $\overline{\mathrm{NE}}(X)$ is generated by fibers of the natural projection $E\to C$
and the fibers of the conic bundle $\pi$.
We claim that $Z\not\subset E$.
Indeed, suppose that this is not the case and $Z\subset E$. Then $f(Z)=C$, since $f(Z)$ is not a point.
Let us seek for a contradiction using Abban--Zhuang method \cite{AbbanZhuang}.
Let $u$ be a non-negative real number. Then  $-K_X-uE$ is pseudoeffective if and only if the divisor $-K_X-uE$ is nef if and only if $u\leqslant 1$, because
$$
-K_X-uE\sim_{\mathbb{R}} 2H-(1+u)E.
$$
Note that
$$
\big(-K_X-uE\big)^3=-E^3u^3+(6H\cdot E^2-3E^3)u^2+(12H\cdot E^2-12E\cdot H^2-3E^3)u+8H^3-12E\cdot H^2-6H\cdot E^2-E^3.
$$
This gives $(-K_X-uE)^3=2u^3-6u^2-18u+22$, because $H^3=4$, $E\cdot H^2=0$, $H\cdot E^2=-2$ and $E^3=-\mathrm{c}_1(\mathcal{N}_{C/V})=-2$.
Now, integrating $(-K_X-uE)^3$, we get $S_X(E)=\frac{23}{44}$.
Then, it follows from \cite[Corollary 1.110]{Book} that
$$
1\geqslant\frac{A_X(\mathbf{F})}{S_X(\mathbf{F})}\geqslant \min\left\{\frac{1}{S_X(E)},\frac{1}{S(W_{\bullet,\bullet}^E;Z)}\right\}=\min\left\{\frac{44}{23},\frac{1}{S(W_{\bullet,\bullet}^E;Z)}\right\}
$$
where
$$
S\big(W_{\bullet,\bullet}^E; Z\big)=\frac{3}{22}\int_{0}^{1}\int_0^\infty \mathrm{vol}\big((-K_X-uE)\big\vert_{E}-vZ\big)dvdu.
$$
Thus, we conclude that $S(W_{\bullet,\bullet}^E;Z)\geqslant 1$. Let us show that $S(W_{\bullet,\bullet}^E;Z)<1$.
First, we observe that either $E\simeq\mathbb{P}^1\times\mathbb{P}^1$ or $E\simeq\mathbb{F}_2$.
If $E\cong\mathbb{P}^1\times\mathbb{P}^1$, we let $\mathbf{s}$ be a section of the natural projection $E\to C$ with $\mathbf{s}^2=0$.
Similarly, if $E\cong\mathbb{F}_2$, we let $\mathbf{s}$ be the section of the projection $E\to C$ with $\mathbf{s}^2=-2$.
In both cases, we let $\mathbf{l}$ be a fiber of the natural projection $E\to C_2$. Then
\begin{align*}
H\big\vert_{E}&\sim 2\mathbf{l},\\
-E\big\vert_{E}&\sim\mathbf{s}+a\mathbf{l}.
\end{align*}
Note that $-2=E^3=(\mathbf{s}+a\mathbf{l})^2=\mathbf{s}^2+2a$.
Thus, if $E\cong\mathbb{P}^1\times\mathbb{P}^1$, then $a=-1$, which gives
$$
(-K_X-uE)\big\vert_{E}\sim_{\mathbb{R}}(1+u)\mathbf{s}+(3-u)\mathbf{l}.
$$
Likewise, if $E\cong\mathbb{F}_2$, then $a=0$, which gives
$$
(-K_X-uE)\big\vert_{E}\sim_{\mathbb{R}}(1+u)\mathbf{s}+4\mathbf{l}.
$$
Observe also that $|Z-\mathbf{s}|\ne\varnothing$, because $f(Z)=C$. This implies that
$$
S\big(W^E_{\bullet,\bullet};Z\big)\leqslant S\big(W^E_{\bullet,\bullet};\mathbf{s}\big)=\frac{3}{22}\int_0^1\int_0^\infty \mathrm{vol}\big((-K_X-uE)\big\vert_{E}-v\mathbf{s}\big)dvdu.
$$
On the other hand, if $E\cong\mathbb{P}^1\times\mathbb{P}^1$, then
\begin{multline*}
\frac{3}{22}\int_0^1\int_0^\infty \mathrm{vol}\big((-K_X-uE)\big\vert_{E}-v\mathbf{s}\big)dvdu=\frac{3}{22}\int_0^1\int_0^{1+u}\big((1+u-v)\mathbf{s}+(3-u)\mathbf{l}\big)^2dvdu=\\
=\frac{3}{22}\int_0^1\int_0^{1+u}2(3-u)(1+u-v)dvdu=\frac{67}{88}.
\end{multline*}
Similarly, if $E\cong\mathbb{F}_2$, then
\begin{multline*}
\frac{3}{22}\int_0^1\int_0^\infty \mathrm{vol}\big((-K_X-uE)\big\vert_{E}-v\mathbf{s}\big)dvdu=\frac{3}{22}\int_0^1\int_0^{1+u}\big((1+u-v)\mathbf{s}+4\mathbf{l}\big)^2dvdu=\\
\frac{3}{22}\int_0^1\int_0^{1+u}2(3-u+v)(1+u-v)dvdu=\frac{41}{44}.
\end{multline*}
This shows that $S(W_{\bullet,\bullet}^E;Z)<1$, which contradicts the inequality $S(W_{\bullet,\bullet}^E;Z)\geqslant 1$ obtained earlier.
Hence, we conclude that the curve $Z$ is not contained in the $f$-exceptional divisor $E$.

Now, we let $S$ be a general surface in the linear system $|\pi^*(\mathcal{O}_{\mathbb{P}^2}(1))|$ such that $S\cap Z$ is not empty.
Fix a point $p\in Z\cap S$, and let $A$ be the fiber of the conic bundle $\pi$ that passes through $p$.
Then $\delta_p(X)\leqslant 1$, and the curve $A$ is smooth.
One the other hand, arguing exactly as in the proof of Lemma~\ref{lemma:pointless-2-13},
one can show that $\delta_p(X)\geqslant\frac{176}{169}$, which gives us the desired contradiction.
For convenience of the reader, let us present the details here.
As above, $E$ stands for the $f$-exceptional surface, $H=f^*(\mathcal{O}_{V}(1))$, and $u$ is a non-negative real number.
Then $-K_X-uS\sim_{\mathbb{R}} (2-u)H+(u-1)E$,
so the divisor $-K_X-uS$ is pseudoeffective $\iff$ $u\leqslant 2$.
For $u\in[0,2]$, let $P(u)$ be the positive part of the Zariski decomposition of the divisor $-K_X-uS$,
and let $N(u)$ be its negative part. Then
$$
P(u)=\left\{\aligned
&(2-u)H+(u-1)E\ \text{if $0\leqslant u\leqslant 1$}, \\
&(2-u)H\ \text{if $1\leqslant u\leqslant 2$},
\endaligned
\right.
$$
and
$$
N(u)=\left\{\aligned
&0\ \text{if $0\leqslant u\leqslant 1$}, \\
&(u-1)E\ \text{if $1\leqslant u\leqslant 2$}.
\endaligned
\right.
$$
Now, integrating $(P(u))^3$, we get $S_X(S)=\frac{13}{22}$,
so it follows from \cite[Theorem~3.3]{AbbanZhuang} and \cite[Corollary 1.102]{Book} that
$$
\delta_p(X)\geqslant \min\left\{\frac{1}{S_X(S)},\inf_{\substack{F/S\\ p\in C_S(F)}}\frac{A_S(F)}{S(W_{\bullet,\bullet}^S;F)}\right\}=
\min\left\{\frac{22}{13},\inf_{\substack{F/S\\ p\in C_S(F)}}\frac{A_S(F)}{S(W_{\bullet,\bullet}^S;F)}\right\}
$$
where the infimum is taken by all prime divisors $F$ over the surface $S$ such that $p\in C_S(F)$, and
$$
S\big(W_{\bullet,\bullet}^S; F\big)=
\frac{3}{22}\int_1^{2}\big(P(u)\big\vert_{S}\big)^2(u-1)\mathrm{ord}_{F}\big(E\big\vert_{S}\big)du+
\frac{3}{22}\int_{0}^{2}\int_0^\infty \mathrm{vol}\big(P(u)\big\vert_{S}-vF\big)dvdu.
$$
So, since $\delta_p(X)\leqslant 1<\frac{176}{169}$, we see  that there exists a prime divisor $F$ over $S$ such that
\begin{equation}
\label{equation:2-16}
S\big(W_{\bullet,\bullet}^S; F\big)>\frac{169}{176}A_S(F).
\end{equation}
Let us show that this inequality is impossible.
Observe that the surface $S$ is smooth by construction.
Moreover, it follows from the adjunction formula that  $-K_S\sim H\vert_{S}$.
In particular, the divisor $S$ is nef and big, $S$ is a smooth weak del Pezzo surface of degree $(-K_S)^2=4$.
However, the divisor $-K_S$ may not be ample.
Indeed, if $Z=A$ and $Z$ is contained in the $f$-exceptional divisor $E$,
then $E\simeq\mathbb{F}_2$, $Z$ is the $(-2)$-curve in $E$, and $S\vert_{E}=Z+\mathbf{l}_1+\mathbf{l}_2$
for two distinct fibers $\mathbf{l}_1$ and $\mathbf{l}_2$ of the natural projection $E\to C$.
In this case, the divisor $-K_S$ intersects both curves $\mathbf{l}_1$ and $\mathbf{l}_2$ trivially,
and these are the only (irreducible) curves in $S$ that have trivial intersection with the anticanonical divisor $-K_S$.
However, this is the only case when the divisor $-K_S$ is not ample.
Thus, since we already proved that $Z\not\subset E$, we see that $S$ is a smooth del Pezzo surface of degree $4$.

We prefer to think of $S$ as of a complete intersection of two quadrics in $\mathbb{P}^4$.
Then $E\vert_{S}$ is a smooth conic, and $A$ is also a smooth conic in $S$.
These two conics are different, since $E\vert_S$ is not contracted by $\pi$.
Moreover, we have
$$
P(u)\big\vert_{S}=\left\{\aligned
&-K_S+(1-u)A\ \text{if $0\leqslant u\leqslant 1$}, \\
&(2-u)(-K_S)\ \text{if $1\leqslant u\leqslant 2$}.
\endaligned
\right.
$$
This gives
\begin{multline*}
S\big(W_{\bullet,\bullet}^S; F\big)=\frac{3}{22}\int_1^{2}\big(P(u)\big\vert_{S}\big)^2(u-1)\mathrm{ord}_{F}\big(E\big\vert_{S}\big)du+\frac{3}{22}\int_{0}^{1}\int_0^\infty\mathrm{vol}\big(-K_S+(1-u)A-vF\big)dvdu+\\
+\frac{3}{22}\int_{1}^{2}\int_0^\infty\mathrm{vol}\big((2-u)(-K_S)-vF\big)dvdu\leqslant \frac{3}{22}\int_1^{2}\big(P(u)\big\vert_{S}\big)^2(u-1)A_{S}(F)du+\\
+\frac{3}{22}\int_{0}^{1}\int_0^\infty\mathrm{vol}\big(-K_S+(1-u)A-vF\big)dvdu+\frac{3}{22}\int_{1}^{2}\int_0^\infty\mathrm{vol}\big((2-u)(-K_S)-vF\big)dvdu=\\
=\frac{3A_{S}(F)}{22}\int_1^{2}(2-u)^2(-K_S)^2(u-1)du+\frac{3}{22}\int_{0}^{1}\int_0^\infty\mathrm{vol}\big(-K_S+(1-u)A-vF\big)dvdu+\\
+\frac{3}{22}\int_{1}^{2}(2-u)^3\int_0^\infty\mathrm{vol}\big(-K_S-vF\big)dvdu=\frac{6A_{S}(F)}{11}\int_1^{2}(2-u)^2(u-1)du+\\
+\frac{3}{22}\int_{0}^{1}\int_0^\infty\mathrm{vol}\big(-K_S+(1-u)A-vF\big)dvdu+\frac{3}{22}\int_{1}^{2}(2-u)^3\int_0^\infty\mathrm{vol}\big(-K_S-vF\big)dvdu\leqslant\\
\leqslant\frac{A_{S}(F)}{22}+\frac{3}{22}\int_{0}^{1}\int_0^\infty\mathrm{vol}\big(-K_S+(1-u)A-vF\big)dvdu+\frac{3}{22}\int_{1}^{2}(2-u)^3(-K_S)^2\frac{A_{S}(F)}{\delta(S)}du=\\
=\frac{A_{S}(F)}{22}+\frac{3}{22}\int_{0}^{1}\int_0^\infty\mathrm{vol}\big(-K_S+(1-u)A-vF\big)dvdu+\frac{9A_{S}(F)}{22}\int_{1}^{2}(2-u)^3du=\\
=\frac{13A_{S}(F)}{88}+\frac{3}{22}\int_{0}^{1}\int_0^\infty\mathrm{vol}\big(-K_S+(1-u)A-vF\big)dvdu,
\end{multline*}
because $\delta(S)=\frac{4}{3}$ by \cite[Lemma 2.12]{Book},
and $\mathrm{ord}_{F}(E\vert_{S})\leqslant A_{S}(F)$,
since $(S,E\vert_{S})$ has log canonical singularities.
Moreover, if $u\in[0,1]$, then it follows from \cite[Lemma 23]{CheltsovFujitaKishimotoOkada} that
$$
\frac{1}{(-K_S+(1-u)A)^2}\int_0^\infty\mathrm{vol}\big(-K_S+(1-u)A-vF\big)\leqslant A_{S}(F)\frac{u^2-10u+28}{24},
$$
where $(-K_S+(1-u)A)^2=8-4u$. Thus, we have
$$
S\big(W_{\bullet,\bullet}^S; F\big)\leqslant\frac{13A_{S}(F)}{88}+\frac{3A_{S}(F)}{22}\int_{0}^{1}\frac{(u^2-10u+28)(8-4u)}{24}du=\frac{169}{176}A_{S}(F),
$$
which contradicts \eqref{equation:2-16}. This completes the proof of the lemma.
\end{proof}

\begin{lemma}
\label{lemma:pointless-2-19}
Suppose that $X$ is contained in Family \textnumero 2.19 and $X(\Bbbk)=\varnothing$.
Then $X_{\mathbb{C}}$ is K-polystable.
\end{lemma}

\begin{proof}
Using \cite{MoMu81,MoMu83} and Lemmas~\ref{lemma:Prokhorov} and \ref{lemma:nonempty}, we see that there exists the following diagram:
$$
\xymatrix@R=1em{
&X\ar@{->}[ld]_{\pi}\ar@{->}[rd]^{\phi}&\\%
U&&V}
$$
where $U$ is a pointless form of $\mathbb{P}^3$, $V$ is a pointless form of a complete intersection of two quadrics in $\mathbb{P}^5$,
$\pi$ is the blowup of a smooth geometrically irreducible curve $\mathscr{C}$ of genus $2$ with $-K_U\cdot \mathscr{C}=40$,
and $\phi$ is the blowup of a smooth geometrically rational curve $\mathscr{L}$ such that $-K_V\cdot\mathscr{L}=2$.
Moreover, the curve $\mathscr{C}$ is contained in a unique smooth surface $\mathscr{S}\subset U$ with $-K_U\sim 2\mathscr{S}$.
In the following, we will denote by $E$ the $\pi$-exceptional surfaces,
and we will denote by $Q$ the proper transform of the surface $\mathscr{S}$ on the 3-fold $X$.
Over $\mathbb{C}$, we have $U_{\mathbb{C}}\simeq\mathbb{P}^3$, the surface $\mathscr{S}_{\mathbb{C}}$ is a smooth quadric surface,
the curve $\mathscr{C}_{\mathbb{C}}$ is a divisor of degree $(3,2)$ in $\mathscr{S}_{\mathbb{C}}\simeq\mathbb{P}^1\times\mathbb{P}^1$,
and the curve $\mathscr{L}_{\mathbb{C}}$ is a line in $V_{\mathbb{C}}$.

Using an arithmetic analogues of \cite[Lemma 1.42]{Book} and \cite[Lemma 1.45]{Book},
we see that there is an~effective $\mathbb{Q}$-divisor $D$ on the 3-fold $X$ defined over $\Bbbk$
such that $D\sim_{\mathbb{Q}}-K_X$ and $Z_{\mathbb{C}}$ is contained in the locus $\mathrm{Nklt}(X_{\mathbb{C}},\lambda D_{\mathbb{C}})$
for some positive rational number $\lambda<\frac{3}{4}$.
Moreover, if $R$ is an irreducible (but possibly geometrically reducible) surface in $X$ such that
$R_{\mathbb{C}}$ is contained in the locus $\mathrm{Nklt}(X_{\mathbb{C}},\lambda D_{\mathbb{C}})$,
then it follows from the proof of \cite[Lemma 4.43]{Book} that either $R_{\mathbb{C}}=Q_{\mathbb{C}}$ or $\pi_{\mathbb{C}}(R_{\mathbb{C}})$ is a plane in $U_{\mathbb{C}}\simeq\mathbb{P}^3$.
Since $U\not\simeq\mathbb{P}^3$, the latter possibility is excluded by Lemma~\ref{lemma:plane}.
Thus, the surface $Q_{\mathbb{C}}$ is the only surface that can (a priori) be contained in the locus $\mathrm{Nklt}(X_{\mathbb{C}},\lambda D_{\mathbb{C}})$.

We claim that $Z_{\mathbb{C}}\not\subset Q_{\mathbb{C}}$.
Indeed, it follows from the proof of \cite[Lemma~4.41]{Book} that $Z_{\mathbb{C}}\ne E_{\mathbb{C}}\cap Q_{\mathbb{C}}$.
Moreover, after a very minor modification, the proof of \cite[Lemma~4.41]{Book} also gives $Z\not\subset Q$.
To see this, suppose that $Z_{\mathbb{C}}\subset Q_{\mathbb{C}}$.
Let $H$ be a~hyperplane in $\mathbb{P}^3_{\mathbb{C}}$, and let $u$ be a~non-negative real number.
Then the divisor $-K_{X_{\mathbb{C}}}-uQ_{\mathbb{C}}$ is nef for $u\in[0,1]$, and it is not pseudo-effective for $u>2$.
Moreover, if $u\in[1,2]$, then the positive part of the Zariski decosmposition of the divisor $-K_{X_{\mathbb{C}}}-uQ_{\mathbb{C}}$
is $(4-2u)\pi_{\mathbb{C}}^*(H)$, and its negative part is $(u-1)E_{\mathbb{C}}$.
Furthermore, we have $S_{X_{\mathbb{C}}}(Q_{\mathbb{C}})<1$ by \cite[Theorem~3.17]{Book}.
Since we know that $Z_{\mathbb{C}}\ne E_{\mathbb{C}}\cap Q_{\mathbb{C}}$,
it follows from \cite[Corollary 1.110]{Book} that
$$
\int_{0}^{1}\int_0^\infty \mathrm{vol}\Big(\big((4-2u)\pi_{\mathbb{C}}^*(H)+(u-1)E_{\mathbb{C}}\big)\big\vert_{Q_{\mathbb{C}}}-vZ\Big)dvdu+
\int_{1}^{2}\int_0^\infty \mathrm{vol}\big((4-2u)\pi_{\mathbb{C}}^*(H)\big\vert_{Q_{\mathbb{C}}}-vZ\big)dvdu\geqslant \frac{26}{3}.
$$
Let $\mathbf{l}_1$ be a curve in  $Q_{\mathbb{C}}\simeq\mathscr{S}_{\mathbb{C}}\simeq\mathbb{P}^1\times\mathbb{P}^1$ of degree $(0,1)$,
and let $\mathbf{l}_2$ be a curve in $Q_{\mathbb{C}}$ of degree $(1,0)$.
Then either $|Z_{\mathbb{C}}-\mathbf{l}_1|\ne\varnothing$ or $|Z_{\mathbb{C}}-\mathbf{l}_2|\ne\varnothing$ (or both).
In the former case, we get a contradiction:
\begin{multline*}
\int_{0}^{1}\int_0^\infty \mathrm{vol}\Big(\big((4-2u)\pi_{\mathbb{C}}^*(H)+(u-1)E_{\mathbb{C}}\big)\big\vert_{Q_{\mathbb{C}}}-vZ\Big)dvdu+
\int_{1}^{2}\int_0^\infty \mathrm{vol}\big((4-2u)\pi_{\mathbb{C}}^*(H)\big\vert_{Q_{\mathbb{C}}}-vZ\big)dvdu\leqslant\\
\leqslant\int_{0}^{1}\int_0^\infty \mathrm{vol}\Big(\big((4-2u)\pi_{\mathbb{C}}^*(H)+(u-1)E_{\mathbb{C}}\big)\big\vert_{Q_{\mathbb{C}}}-v\mathbf{l}_1\Big)dvdu+
\int_{1}^{2}\int_0^\infty \mathrm{vol}\big((4-2u)\pi_{\mathbb{C}}^*(H)\big\vert_{Q_{\mathbb{C}}}-v\mathbf{l}_1\big)dvdu=\\
=\int_{0}^{1}\int_{0}^{u+1}4(u+1-v)dvdu+\int_{1}^{2}\int_{0}^{4-2u}2(4-2u-v)(4-2u)dvdu=\frac{20}{3},
\end{multline*}
because $\big((4-2u)\pi_{\mathbb{C}}^*(H)+(u-1)E_{\mathbb{C}}\big)\vert_{Q_{\mathbb{C}}}$ is an $\mathbb{R}$-divisor of degree $(u+1,2)$ on $Q_{\mathbb{C}}$,
and $(4-2u)\pi_{\mathbb{C}}^*(H)\vert_{Q_{\mathbb{C}}}$ is an $\mathbb{R}$-divisor of degree $(4-2u,4-2u)$.
Similarly, if $|Z_{\mathbb{C}}-\mathbf{l}_2|\ne\varnothing$, we also a contradiction:
\begin{multline*}
\int_{0}^{1}\int_0^\infty \mathrm{vol}\Big(\big((4-2u)\pi_{\mathbb{C}}^*(H)+(u-1)E_{\mathbb{C}}\big)\big\vert_{Q_{\mathbb{C}}}-vZ\Big)dvdu+
\int_{1}^{2}\int_0^\infty \mathrm{vol}\big((4-2u)\pi_{\mathbb{C}}^*(H)\big\vert_{Q_{\mathbb{C}}}-vZ\big)dvdu\leqslant\\
\leqslant\int_{0}^{1}\int_0^\infty \mathrm{vol}\Big(\big((4-2u)\pi_{\mathbb{C}}^*(H)+(u-1)E_{\mathbb{C}}\big)\big\vert_{Q_{\mathbb{C}}}-v\mathbf{l}_2\Big)dvdu+
\int_{1}^{2}\int_0^\infty \mathrm{vol}\big((4-2u)\pi_{\mathbb{C}}^*(H)\big\vert_{Q_{\mathbb{C}}}-v\mathbf{l}_2\big)dvdu=\\
=\int_{0}^{1}\int_{0}^{2}2(u+1)(2-v)dvdu+\int_{1}^{2}\int_{0}^{4-2u}2(4-2u)(4-2u-v)dvdu=\frac{24}{3}.
\end{multline*}
The obtained contradictions show that $Z_{\mathbb{C}}\not\subset Q_{\mathbb{C}}$.

Since $Z_{\mathbb{C}}\not\subset Q_{\mathbb{C}}$, $Z_{\mathbb{C}}\subset\mathrm{Nklt}(X_{\mathbb{C}},\lambda D_{\mathbb{C}})$,
and the locus $\mathrm{Nklt}(X_{\mathbb{C}},\lambda D_{\mathbb{C}})$ does not contain surfaces except possibly for $Q_{\mathbb{C}}$,
it follows from the proof of \cite[Lemma 4.45]{Book} that the curve $Z$ is geometrically rational.
In particular, we see that $Z_{\mathbb{C}}\not\subset E_{\mathbb{C}}$, because the only rational curves in $E_{\mathbb{C}}$ are fibers of the natural projection
$E_{\mathbb{C}}\to\mathscr{C}_{\mathbb{C}}$, and $Z_{\mathbb{C}}$ could not be one of them, since $\mathscr{C}(\Bbbk)=\varnothing$.
Thus, we conclude that $\pi(Z)$ is a~geometrically rational rational curve in $U$ that is not contained in the surface $\mathscr{S}$.
Set $\overline{Z}=\pi(Z)$. Then it follows from the proof of \cite[Lemma 4.48]{Book} that $\overline{Z}_{\mathbb{C}}$ is a line $U_{\mathbb{C}}\simeq\mathbb{P}^3$.

Since $\overline{Z}_{\mathbb{C}}\not\subset \mathscr{S}_{\mathbb{C}}$,
the line $\overline{Z}_{\mathbb{C}}$ transversally intersects the quadric $\mathscr{S}_{\mathbb{C}}$ in two distinct points,
since otherwise the intersection $\overline{Z}_{\mathbb{C}}\cap \mathscr{S}_{\mathbb{C}}$ would consist of a single point defined over $\Bbbk$.
This implies that either the curves  $\overline{Z}_{\mathbb{C}}$ and $\mathscr{C}_{\mathbb{C}}$ are disjoint,
or they meet transversally in exactly two points.
This and Bertini theorem imply that a general plane in $U_{\mathbb{C}}\simeq\mathbb{P}^3$ that contains the line $\overline{Z}_{\mathbb{C}}$
intersects the curve $\mathscr{C}_{\mathbb{C}}$ by five distinct points in linearly general position,
because every trisecant of the curve $\mathscr{C}_{\mathbb{C}}\subset\mathbb{P}^3$ is contained in the quadric surface $\mathscr{S}_{\mathbb{C}}\simeq\mathbb{P}^1\times\mathbb{P}^1$.

Let $H$ be a~sufficiently general plane in $U_{\mathbb{C}}\simeq\mathbb{P}^3$ that contains the line  $\overline{Z}_{\mathbb{C}}$,
and let $S$ be its proper transform on $X_{\mathbb{C}}$.
Then $S$ is a smooth del Pezzo surface of degree $4$.
Let us apply Abban--Zhuang method to the flag $\overline{Z}_{\mathbb{C}}\subset S$.
As above, let $u$ be a~non-negative real number. Then
$$
-K_{X_{\mathbb{C}}}-uS\sim_{\mathbb{R}} (2-u)\pi_{\mathbb{C}}^*(H)+Q_{\mathbb{C}}.
$$
This implies that the divisor $-K_{X_{\mathbb{C}}}-uS$ is nef for every $u\in[0,1]$, and it is not pseudo-effective for $u>2$.
For~$u\in[0,2]$, let $P(u)$ be the positive part of the Zariski decomposition of the divisor $-K_{X_{\mathbb{C}}}-uS$,
and let $N(u)$ be the negative part of the Zariski decomposition of the divisor $-K_{X_{\mathbb{C}}}-uS$. Then
$$
P(u)=
\left\{\aligned
&(4-u)\pi_{\mathbb{C}}^*(H)-E_{\mathbb{C}}\ \text{if $0\leqslant u\leqslant 1$}, \\
&(2-u)\big(3\pi^*(H)-E_{\mathbb{C}}\big)\ \text{if $1\leqslant u\leqslant 2$},
\endaligned
\right.
$$
and
$$
N(u)=
\left\{\aligned
&0\ \text{if $0\leqslant u\leqslant 1$}, \\
&(1-u)Q_{\mathbb{C}}\ \text{if $1\leqslant u\leqslant 2$}.
\endaligned
\right.
$$
In particular, we see that the curve $Z_{\mathbb{C}}$ is not contained in the support of the divisor $N(u)$ for $u\in[0,2]$.
Moreover, we have $S_X(S)<1$ by \cite[Theorem~3.17]{Book},
so \cite[Corollary 1.110]{Book} gives $S(W_{\bullet,\bullet}^S; Z_{\mathbb{C}})\geqslant 1$, where
$$
S\big(W_{\bullet,\bullet}^S; Z_{\mathbb{C}}\big)=\frac{3}{26}\int_{0}^{2}\int_0^\infty \mathrm{vol}\big(P(u)\big\vert_{S}-vZ_{\mathbb{C}}\big)dvdu.
$$
Let us compute $S(W_{\bullet,\bullet}^S;Z_{\mathbb{C}})$.
In the case when $\overline{Z}_{\mathbb{C}}\cap\mathscr{C}_{\mathbb{C}}=\varnothing$, this is done in the proof of \cite[Lemma~4.49]{Book},
but we present the computations here for consistency.
Let $\varpi\colon S\to H$ be the birational morphism induced by $\pi_{\mathbb{C}}$.
Then $\varpi$ contracts $5$ disjoint smooth curves $\mathbf{e}_1,\mathbf{e}_2,\mathbf{e}_3,\mathbf{e}_4,\mathbf{e}_5$ such that $E_{\mathbb{C}}\vert_{S}=\mathbf{e}_1+\mathbf{e}_2+\mathbf{e}_3+\mathbf{e}_4+\mathbf{e}_5$.
Let $\ell$ be the proper transform of a~general line in the plane $H$ on the surface $S$, and let  $\mathcal{C}=Q_{\mathbb{C}}\vert_{S}$.
Then $\mathcal{C}^2=-1$ and $\mathcal{C}\sim 2\ell-\mathbf{e}_1-\mathbf{e}_2-\mathbf{e}_3-\mathbf{e}_4-\mathbf{e}_5$.
Note that $\varpi(\mathcal{C})$ is the smooth conic $H\cap\mathscr{S}_{\mathbb{C}}$.

Suppose that $\overline{Z}_{\mathbb{C}}\cap\mathscr{C}_{\mathbb{C}}=\varnothing$.
Then $Z_{\mathbb{C}}\sim\ell$. Thus, if $u\in[0,1]$ and $v\in\mathbb{R}_{\geqslant 0}$, then
$$
P(u)\vert_{S}-vZ_{\mathbb{C}}\sim_{\mathbb{R}} (4-u-v)\ell-\sum_{i=1}^{5}\mathbf{e}_i\sim_{\mathbb{R}} (2-u-v)\ell+\mathcal{C},
$$
which implies that $P(u)\big\vert_{S}-vZ_{\mathbb{C}}$ is not pseudo-effective for $v>2-u$,
and  $P(u)\big\vert_{S}-vZ_{\mathbb{C}}$ is nef for $v\leqslant\frac{3-2u}{2}$.
Moreover, if $\frac{3-2u}{2}\leqslant v\leqslant 2-u$, then the Zariski decomposition of the divisor $P(u)\big\vert_{S}-vZ_{\mathbb{C}}$ is
$$
\underbrace{(2-u-v)\big(5\ell-2\mathbf{e}_1-2\mathbf{e}_2-2\mathbf{e}_3-2\mathbf{e}_4-2\mathbf{e}_5\big)}_{\text{positive part}}+\underbrace{(2u+2v-3)\mathcal{C}}_{\text{negative part}}.
$$
Thus, if $u\in[0,1]$, then
$$
\mathrm{vol}\big(P(u)\big\vert_{S}-vZ_{\mathbb{C}}\big)=
\left\{\aligned
&(4-u-v)^2-5\ \text{if $0\leqslant v\leqslant\frac{3-2u}{2}$}, \\
&5(2-u-v)^2\ \text{if $\frac{3-2u}{2}\leqslant v\leqslant 2-u$}.
\endaligned
\right.
$$
Similarly, if $u\in[1,2]$ and $v\in\mathbb{R}_{\geqslant 0}$, then
$$
P(u)\vert_{S}-vZ_{\mathbb{C}}\sim_{\mathbb{R}} (6-3u-v)\ell-(2-u)\sum_{i=1}^{5}\mathbf{e}_i\sim_{\mathbb{R}} (2-u-v)\ell+(2-u)\mathcal{C},
$$
which implies that  $P(u)\big\vert_{S}-vZ_{\mathbb{C}}$ is not pseudo-effective for $v>2-u$,
and $P(u)\big\vert_{S}-vZ_{\mathbb{C}}$ is nef for $v\leqslant\frac{2-u}{2}$.
Moreover, if $\frac{2-u}{2}\leqslant v\leqslant 2-u$, then the Zariski decomposition of the divisor $P(u)\big\vert_{S}-vZ_{\mathbb{C}}$~is
$$
\underbrace{(2-u-v)\big(5\ell-2\mathbf{e}_1-2\mathbf{e}_2-2\mathbf{e}_3-2\mathbf{e}_4-2\mathbf{e}_5\big)}_{\text{positive part}}+\underbrace{(2u+v-3)\mathcal{C}}_{\text{negative part}}.
$$
Therefore, if $u\in[1,2]$, then
$$
\mathrm{vol}\big(P(u)\big\vert_{S}-vZ_{\mathbb{C}}\big)=
\left\{\aligned
&(6-3u-v)^2-5(2-u)^2\ \text{if $0\leqslant v\leqslant\frac{2-u}{2}$}, \\
&5(2-u-v)^2\ \text{if $\frac{2-u}{2}\leqslant v\leqslant 2-u$}.
\endaligned
\right.
$$
Thus, we have
\begin{multline*}
S\big(W_{\bullet,\bullet}^{S};Z_{\mathbb{C}}\big)=\frac{3}{26}\int_{0}^{1}\int_{0}^{\frac{3-2u}{2}}((4-u-v)^2-5)dudv+\frac{1}{10}\int_{0}^{1}\int_{\frac{3-2u}{2}}^{2-u}5(2-u-v)^2dudv+\\
+\frac{3}{26}\int_{1}^{2}\int_{0}^{\frac{2-u}{2}}\Big((6-3u-v)^2-5(2-u)^2\Big)dudv+\frac{1}{10}\int_{1}^{2}\int_{\frac{2-u}{2}}^{2-u}5(2-u-v)^2dudv=\frac{119}{208}<1.
\end{multline*}
This shows that $\overline{Z}_{\mathbb{C}}\cap\mathscr{C}_{\mathbb{C}}\ne\varnothing$.
Hence, the intersection $\overline{Z}_{\mathbb{C}}\cap\mathscr{C}_{\mathbb{C}}$ consists of two points
among $\varpi(\mathbf{e}_1)$, $\varpi(\mathbf{e}_2)$, $\varpi(\mathbf{e}_3)$, $\varpi(\mathbf{e}_4)$, $\varpi(\mathbf{e}_5)$.
Without loss of generality, we may assume that $\overline{Z}_{\mathbb{C}}\cap\mathscr{C}_{\mathbb{C}}=\varpi(\mathbf{e}_1)\cup \varpi(\mathbf{e}_2)$.
Then we have $Z_{\mathbb{C}}\sim\ell-\mathbf{e}_1-\mathbf{e}_2$. Thus, if $u\in[0,1]$ and $v\in\mathbb{R}_{\geqslant 0}$, then
$$
P(u)\vert_{S}-vZ_{\mathbb{C}}\sim_{\mathbb{R}} (4-u-v)\ell-(1-v)\mathbf{e}_1-(1-v)\mathbf{e}_2-\mathbf{e}_3-\mathbf{e}_4-\mathbf{e}_5.
$$
This implies that $P(u)\big\vert_{S}-vZ_{\mathbb{C}}$ is pseudo-effective if and only if $v\leqslant\frac{5-2u}{2}$.
Moreover, this divisor is nef for $v\leqslant 1$.
Furthermore, if $1\leqslant v\leqslant 2-u$, then the Zariski decomposition of the divisor $P(u)\big\vert_{S}-vZ_{\mathbb{C}}$ is
$$
\underbrace{(4-u-v)\ell-\mathbf{e}_3-\mathbf{e}_4-\mathbf{e}_5\big)}_{\text{positive part}}+\underbrace{(v-1)\big(\mathbf{e}_1+\mathbf{e}_2\big)}_{\text{negative part}}.
$$
Finally, if $2-u\leqslant v\leqslant \frac{5-2u}{2}$, then the Zariski decomposition of the divisor $P(u)\big\vert_{S}-vZ_{\mathbb{C}}$ is
$$
\underbrace{(5-2u-2v)\big(2\ell-\mathbf{e}_3-\mathbf{e}_4-\mathbf{e}_5\big)}_{\text{positive part}}+\underbrace{(v-1)\big(\mathbf{e}_1+\mathbf{e}_2\big)+(v+u-2)\big(L_{34}+L_{35}+L_{45}\big)}_{\text{negative part}},
$$
where $L_{34}$, $L_{35}$, $L_{45}$ are $(-1)$-curves such that $L_{34}\sim\ell-\mathbf{e}_3-\mathbf{e}_4$, $L_{35}\sim\ell-\mathbf{e}_3-\mathbf{e}_5$, $L_{45}\sim\ell-\mathbf{e}_4-\mathbf{e}_5$.
Thus, if $u\in[0,1]$, then
$$
\mathrm{vol}\big(P(u)\big\vert_{S}-vZ_{\mathbb{C}}\big)=
\left\{\aligned
&u^2+2uv-v^2-8u-4v+11\ \text{if $0\leqslant v\leqslant 1$}, \\
&u^2+2uv+v^2-8u-8v+13\ \text{if $1\leqslant v\leqslant 2-u$}, \\
&(5-2u-2v)^2\ \text{if $2-u\leqslant v\leqslant \frac{5-2u}{2}$}.
\endaligned
\right.
$$
Similarly, if $u\in[1,2]$ and $v\in\mathbb{R}_{\geqslant 0}$, then
$$
P(u)\vert_{S}-vZ_{\mathbb{C}}\sim_{\mathbb{R}} (6-3u-v)\ell-(2-u-v)\big(\mathbf{e}_1+\mathbf{e}_2\big)-(2-u)\big(\mathbf{e}_3+\mathbf{e}_4+\mathbf{e}_5\big).
$$
This implies that  $P(u)\big\vert_{S}-vZ_{\mathbb{C}}$ is pseudo-effective if and only if $v\leqslant\frac{6-3u}{2}$,
and this divisor is nef if and only if $v\leqslant 2-u$.
Furthermore, if $2-u\leqslant v\leqslant \frac{6-3u}{2}$, then the Zariski decomposition of the divisor $P(u)\big\vert_{S}-vZ_{\mathbb{C}}$~is
$$
\underbrace{(6-3u-2v)\big(2\ell-\mathbf{e}_3-\mathbf{e}_4-\mathbf{e}_5\big)}_{\text{positive part}}+\underbrace{(v+u-2)\big(\mathbf{e}_1+\mathbf{e}_2+L_{34}+L_{35}+L_{45}\big)}_{\text{negative part}}.
$$
Therefore, if $u\in[1,2]$, then
$$
\mathrm{vol}\big(P(u)\big\vert_{S}-vZ_{\mathbb{C}}\big)=
\left\{\aligned
&4u^2+2uv-v^2-16u-4v+16\ \text{if $0\leqslant v\leqslant 2-u$}, \\
&(6-3u-2v)^2\ \text{if $2-u\leqslant v\leqslant \frac{6-3u}{2}$}.
\endaligned
\right.
$$
Now, integrating, we obtain $S(W_{\bullet,\bullet}^{S};Z_{\mathbb{C}})=\frac{183}{208}$, which contradicts the inequality $S(W_{\bullet,\bullet}^{S};Z_{\mathbb{C}})\geqslant 1$ obtained earlier.
This completes the proof.
\end{proof}

\begin{lemma}
\label{lemma:pointless-2-21}
Suppose that $X$ is contained in Family \textnumero 2.21 and $X(\Bbbk)=\varnothing$.
Then $X_{\mathbb{C}}$ is K-polystable.
\end{lemma}

\begin{proof}
The required assertion follows from \cite{Malbon}.
Indeed, over $\mathbb{C}$, we have the following diagram:
$$
\xymatrix@R=1em{
&X_{\mathbb{C}}\ar@{->}[dl]_{\pi}\ar@{->}[dr]^{\pi^\prime}&\\%
Q&&Q&}
$$
where $Q$ is a smooth quadric 3-fold in $\mathbb{P}^4_{\mathbb{C}}$,
and the morphisms $\pi$ and $\pi^\prime$ are blowups of smooth twisted quartic curves.
Denote by $E$ and $E^\prime$ the exceptional divisors of the blowups $\pi$ and $\pi^\prime$, respectively.
Then it follows from \cite[Technical~Theorem 1]{Malbon} that $\delta_p(X_{\mathbb{C}}\big)>1$
for every point $p\in X_{\mathbb{C}}$ with $p\not\in E\cup E^\prime$.
It follows from \cite[Technical~Theorem~1]{Malbon} that $Z_{\mathbb{C}}\subset E\cup E^\prime$,
and it follows from \cite[Technical~Theorem~2]{Malbon} that $\pi(Z_{\mathbb{C}})$ or $\pi^\prime(Z_{\mathbb{C}})$ is a point.
Hence, without loss of generality, we may assume that $\pi(Z_{\mathbb{C}})$ is a point, so that $Z_{\mathbb{C}}$ is a fiber of the natural projection $E\to C_4$,
where $C_4$ is a twisted quartic curve in $Q$ blown up by $\pi$.
Then $\pi^\prime(Z_{\mathbb{C}})$ is a line, which implies that~$Z_{\mathbb{C}}\not\subset E^\prime$.
Since $Z$ is defined over $\Bbbk$, we see that the divisors $E$ and $E^\prime$ are also defined over $\Bbbk$.
Then $\pi$ induces a birational morphism $X\to\mathscr{Q}$ defined over $\Bbbk$,
where $\mathscr{Q}$ is a form of a smooth quadric 3-fold.
This morphism contracts $Z$ to a point in $\mathscr{Q}$,
so $\mathscr{Q}(\Bbbk)\ne\varnothing$, which gives $X(\Bbbk)\ne\varnothing$ by Lemma~\ref{lemma:nonempty}, which is a contradiction.
\end{proof}

Note that Family \textnumero 2.21 contains smooth Fano 3-folds that are not K-polystable.

\begin{lemma}
\label{lemma:pointless-2-24}
Suppose that $X$ is contained in Family \textnumero 2.24 and $X(\Bbbk)=\varnothing$.
Then $X_{\mathbb{C}}$ is K-polystable.
\end{lemma}

\begin{proof}
Recall that the geometric model of $X$ is a divisor of degree $(1,2)$ in $\mathbb{P}^2\times\mathbb{P}^2$.
Let $\mathrm{pr}_1\colon X_{\mathbb{C}}\to \mathbb{P}^2$ and $\mathrm{pr}_2\colon X_{\mathbb{C}}\to \mathbb{P}^2$ be the projections to the first and the second factors, respectively.
Then the morphism $\mathrm{pr}_1$ is a~conic bundle, and $\mathrm{pr}_2$ is a~$\mathbb{P}^1$-bundle.
Let $\mathscr{C}$ be the discriminant curve of the conic bundle $\mathrm{pr}_1$. Then  $\mathscr{C}$ is a~reduced cubic curve.
Moreover, since $X_{\mathbb{C}}$ is smooth, the curve  $\mathscr{C}$ is either smooth or nodal.
Furthermore, it has been been shown in \cite[\S~4.7]{Book} that
we can choose coordinates $([x:y:z],[u:v:w])$ on $\mathbb{P}^2\times\mathbb{P}^2$ such that
either $X_{\mathbb{C}}$ is given
\begin{equation}
\label{equation:2-24-stable}
\big(\mu vw+u^2\big)x+\big(\mu uw+v^2\big)y+\big(\mu uv+w^2\big)z=0
\end{equation}
for some $\mu\in \mathbb{C}$ such that $\mu^3\ne -1$, or $X_{\mathbb{C}}$ is given by
\begin{equation}
\label{equation:2-24-unstable-finite-Aut}
\big(vw+u^2\big)x+\big(uw+v^2\big)y+w^2z=0,
\end{equation}
or $X_{\mathbb{C}}$ is given by
\begin{equation}
\label{equation:2-24-unstable-infinite-Aut}
\big(vw+u^2\big)x+v^2y+w^2z=0.
\end{equation}
Moreover, it follows from \cite[Lemma~4.70]{Book} that $X_{\mathbb{C}}$ is K-polystable
if and only if $X_{\mathbb{C}}$ can be given by \eqref{equation:2-24-stable} for some $\mu\in \mathbb{C}$ such that $\mu^3\ne -1$.
In this case, the curve $\mathcal{C}$ is either smooth or a union of three lines that do not share a common point.
On the other hand, if $X_{\mathbb{C}}$ is given by \eqref{equation:2-24-unstable-finite-Aut}, then $\mathcal{C}$ is an irreducible cubic curve with $1$ singular point.
Similarly, if $X_{\mathbb{C}}$ is given by \eqref{equation:2-24-unstable-infinite-Aut}, then $\mathcal{C}$ is a union of a line and a smooth conic that meet in two points.

Over $\Bbbk$, the 3-fold $X$ is a divisor in $V\times U$ where $V$ and $U$ are $\Bbbk$-forms of $\mathbb{P}^2$.
We may assume that the natural projection $X\to V$ is a conic bundle with discriminant curve $\mathscr{C}$ such that its geometric model is isomorphic to $\mathcal{C}$.
We claim that $X(\Bbbk)\ne\varnothing$ if $V\simeq\mathbb{P}^2$.
Indeed, if $V\simeq\mathbb{P}^2$, let $P$ be a general $\Bbbk$-point in $V$, let $F$ be the fiber of the conic bundle $X\to V$ over $P$,
and let $C$ be its image in $U$ via the natural projection $X\to U$.
Then $C_{\mathbb{C}}$ is a conic in $U_{\mathbb{C}}\simeq\mathbb{P}^2$, so that $U\simeq\mathbb{P}^2$ by Lemma~\ref{lemma:SB}.
Now, intersecting a general fiber of the projection $X\to U$ with a pull back of a general line in $V\simeq\mathbb{P}^2$,
we obtain a $\Bbbk$-point in $X$. Thus, if $V\simeq\mathbb{P}^2$, then $X(\Bbbk)\ne\varnothing$.

Now, we are ready to prove the requires assertion.
Suppose that $X$ is not K-polystable.
Then either $X_{\mathbb{C}}$ is given by \eqref{equation:2-24-unstable-finite-Aut}, or $X_{\mathbb{C}}$ is given by \eqref{equation:2-24-unstable-infinite-Aut}.
In the former case, $\mathrm{Sing}(\mathscr{C}_{\mathbb{C}})$ consists of one point, which should be defined over $\Bbbk$, so that $V\simeq\mathbb{P}^2$.
In the latter case, we have $V\simeq\mathbb{P}^2$ by Lemma~\ref{lemma:SB}.
Thus, we see that $V\simeq\mathbb{P}^2$ in both cases, which implies $X(\Bbbk)\ne\varnothing$ as we proved earlier.
\end{proof}

\begin{lemma}
\label{lemma:pointless-3-2}
Suppose that $X$ is contained in Family \textnumero 3.2 and $X(\Bbbk)=\varnothing$.
Then $X_{\mathbb{C}}$ is K-polystable.
\end{lemma}

\begin{proof}
The required assertion follows from \cite[Lemma~6.1]{BelousovLoginov2024} and \cite[Lemma~6.2]{BelousovLoginov2024}.
To show this, let us first describe the geometry of the geometric model of the 3-fold $X$.
There are several ways to do this. For~instance, following \cite{MoMu81,MoMu83}, we let
$$
U=\mathbb{P}\Big(\mathcal{O}_{\mathbb{P}^1\times\mathbb{P}^1}\oplus\mathcal{O}_{\mathbb{P}^1\times\mathbb{P}^1}\big(-1,-1\big)\oplus\mathcal{O}_{\mathbb{P}^1\times\mathbb{P}^1}\big(-1,-1\big)\Big),
$$
let $\pi\colon U\to\mathbb{P}^1\times\mathbb{P}^1$ be the natural projection, and let $L$ be a tautological line bundle on the scroll $U$.
Then $X_{\mathbb{C}}$ can be described as a divisor in $|2L+\pi^*(\mathcal{O}_{\mathbb{P}^1\times\mathbb{P}^1}(2,3))|$.
Let $\omega\colon X\to\mathbb{P}^1\times\mathbb{P}^1$ be the restriction of the projection $\pi$ to the 3-fold $X$,
let $\pi_1\colon \mathbb{P}^1\times\mathbb{P}^1\to\mathbb{P}^1$ and $\pi_2\colon \mathbb{P}^1\times\mathbb{P}^1\to\mathbb{P}^1$
be projections to the first and the second factors, respectively.
Set $\phi_1=\pi_1\circ\omega$ and $\phi_2=\pi_2\circ\omega$.
Then a general fiber of the morphism $\phi_1$ is a smooth cubic surface,
a general fiber of the morphism $\phi_2$ is a smooth del Pezzo surface of degree $6$,
and $\omega$ is a standard conic bundle.
On the other hand, following \cite[\S~11]{CheltsovPrzyjalkowskiShramov}, we let
$$
R=\mathbb{P}\Big(\mathcal{O}_{\mathbb{P}^1}(2)\oplus\mathcal{O}_{\mathbb{P}^1}(2)\oplus\mathcal{O}_{\mathbb{P}^1}(1)\oplus\mathcal{O}_{\mathbb{P}^1}(1)\Big),
$$
let $M$ be the tautological line bundle on the scroll $R$,
and let $F$ be a fiber of the natural projection~$R\to\mathbb{P}^1$.
Then $X_{\mathbb{C}}$ can also be described as a divisor in the linear system~$|3M-4F|$.
In the notation of \cite[\S2]{Reid}, we have $R=\mathbb{F}(2,2,1,1)$, and $X$ is given by the following equation:
\begin{multline*}
\alpha^1_2(t_1,t_2)x_1^3+\alpha^2_2(t_1,t_2)x_1^2x_2+\alpha^1_1(t_1,t_2)x_1^2x_3+\alpha^2_1(t_1,t_2)x_1^2x_4+\alpha^3_2(t_1,t_2)x_1x_2^2+\alpha^3_1(t_1,t_2)x_1x_2x_3+\\
+\alpha^4_1(t_1,t_2)x_1x_2x_4+\alpha^1_0(t_1,t_2)x_1x_3^2+\alpha^2_0(t_1,t_2)x_1x_3x_4+\alpha^3_0(t_1,t_2)x_1x_4^2+\alpha^4_2(t_1,t_2)x_2^3+\\
+\alpha^5_1(t_1,t_2)x_2^2x_3+\alpha^6_1(t_1,t_2)x_2^2x_4+\alpha^4_0(t_1,t_2)x_2x_3^2+\alpha^5_0(t_1,t_2)x_2x_3x_4+\alpha^6_0(t_1,t_2)x_2x_4^2=0,
\end{multline*}
where each $\alpha^i_d(t_1,t_2)$ is a polynomial of degree $d$.
Let $S$ be the subscroll in $R$ given by $x_1=x_2=0$. Then~$S\cong\mathbb{P}^1\times\mathbb{P}^1$, and $S$ is contained in $X$.
Furthermore, the normal bundle of $S$ in~$X$ is $\mathcal{O}_{\mathbb{P}^1\times\mathbb{P}^1}(-1,-1)$.
This implies the existence of the following commutative diagram:
\begin{equation}
\label{equation:3-2}
\xymatrix@R=1em{
 &&V&& \\
U_{1}\ar@{->}[dd]_{\psi_{1}}\ar@{->}[rru]^{\gamma_{1}}&&&&U_{2}\ar@{->}[dd]^{\psi_{2}}\ar@{->}[llu]_{\gamma_{2}} \\
&&X_{\mathbb{C}}\ar@{->}[dd]_{\omega}\ar@{->}[lld]_{\phi_1}\ar@{->}[rrd]^{\phi_2}\ar@{->}[uu]_{\alpha}\ar@{->}[ull]_{\beta_{1}}\ar@{->}[urr]^{\beta_{2}}&& \\
\mathbb{P}^1 &&&& \mathbb{P}^1\\
&&\mathbb{P}^1\times\mathbb{P}^1\ar@{->}[llu]^{\pi_1}\ar@{->}[rru]_{\pi_2}&&}%
\end{equation}
where $U_1$ and $U_2$ are smooth 3-folds, the morphisms $\beta_1$ and $\beta_2$ are contractions of the surface $S$ to curves in
these 3-folds, the morphism $\alpha$ is a contraction of the surface $S$ to an isolated ordinary double point of the 3-fold $V$,
the morphism $\psi_1$ is a fibration into del Pezzo surfaces of degree $4$,
and $\psi_2$ is a fibration into quadric surfaces.
This commutative diagram is well known to experts, see the proof of \cite[Theorem~2.3]{TakeuchiNew},
the proof of~\mbox{\cite[Proposition~3.8]{JaPeRa07}}, and the proof of~\mbox{\cite[Lemma~8.2]{CheltsovShramovUMN}}.
Note that $V$ is a Fano 3-fold such that $V$ has non-$\mathbb{Q}$-factorial singularities, $-K_{V}^{3}=16$ and $\mathrm{Pic}(C)=\mathbb{Z}[-K_{V}]$,
and the morphisms $\gamma_1$ and $\gamma_2$ are its two (distinct) small resolutions.

Now, we are ready to prove that $X_{\mathbb{C}}$ is K-polystable. For every point $p\in Z_{\mathbb{C}}$, we have $\delta_p(X_{\mathbb{C}})\leqslant 1$ by assumption that $X_\mathbb{C}$ is not K-polystable and that $p\in Z_{\mathbb{C}}$.
On the other hand, it follows from \cite[Lemma~6.1]{BelousovLoginov2024} that $\delta_p(X_{\mathbb{C}})>1$ for every point $p$
in the $\alpha$-exceptional surface $S$.
Thus, we conclude that $Z_{\mathbb{C}}\cap S=\varnothing$.
Similarly, if $\mathscr{F}$ is a smooth fiber of the fibration into cubic surfaces $\phi_1$,
then it follows from \cite[Lemma~6.2]{BelousovLoginov2024} that $\delta_p(X_{\mathbb{C}})>1$ for every point $p\in \mathscr{F}$,
which implies that $Z_{\mathbb{C}}\cap \mathscr{F}=\varnothing$.
This shows that $\phi_1(Z_{\mathbb{C}})$ is a point in $\mathbb{P}^1$,
and the fiber of $\phi_1$ over this point is singular.

Let $\mathcal{F}$ be the fiber of $\phi_1$ that contains $Z_{\mathbb{C}}$.
Then $\mathcal{F}$ is singular. Then $\mathcal{F}$ is a normal cubic surface by \cite[Lemma~3.2]{BelousovLoginov2024}.
On the other hand, the fiber $\mathcal{F}$ is defined over $\Bbbk$, which implies, in particular, that $\mathcal{F}$ is not a cone.
Hence, we conclude that $\mathcal{F}$ is a singular cubic surface that has Du Val singularities.

Let $\mathcal{C}$ be a general fiber of the induced conic bundle $\phi_2\vert_{\mathcal{F}}\colon \mathcal{F}\to\mathbb{P}^1$.
Then $\mathcal{C}$ is smooth, since $\mathcal{F}$ is normal.
Moreover, if $\mathcal{C}\cap Z_{\mathbb{C}}\ne\varnothing$, then $\delta_p(X_{\mathbb{C}})>1$ for every point $p\in \mathcal{C}\cap Z_{\mathbb{C}}$,
which is impossible, since $\delta_p(X_{\mathbb{C}})\leqslant 1$ for every point $p\in Z_{\mathbb{C}}$.
Thus, we conclude that $\mathcal{C}\cap Z_{\mathbb{C}}=\varnothing$.
This means that $Z_{\mathbb{C}}$ is an irreducible component of a fiber of the conic bundle $\omega$.
But $S$ intersects every irreducible component of every fiber of the conic bundle $\omega$,
which is a contradiction, since we already showed that $Z_{\mathbb{C}}\cap S=\varnothing$.
\end{proof}

\begin{lemma}
\label{lemma:pointless-3-5}
Suppose that $X$ is contained in Family \textnumero 3.5 and $X(\Bbbk)=\varnothing$.
Then $X_{\mathbb{C}}$ is K-polystable.
\end{lemma}

\begin{proof}
The required assertion follows from \cite[\S~5.14]{Book} and \cite{Denisova2024}.
To explain this in details, let us describe the construction and geometry of the geometric model of the 3-fold $X$.
So, for a while, we work over $\mathbb{C}$.

Set $S=\mathbb{P}^1\times\mathbb{P}^1$.
Then $S$ contains a~smooth rational curve $C\subset S$ of degree $(5,1)$ such that $X_{\mathbb{C}}$ can be constructed from $C$ as follows.
Consider the embedding $S\hookrightarrow\mathbb{P}^1\times\mathbb{P}^2$ given by
$$
\big([u:v],[x:y]\big)\mapsto\big([u:v],[x^2:xy:y^2]\big),
$$
and identify $S$ and $C$ with their images in $\mathbb{P}^1\times\mathbb{P}^2$ using this embedding.
Then there exists a birational morphism $\pi\colon X_{\mathbb{C}}\to\mathbb{P}^1\times\mathbb{P}^2$ that blows up $C$.

To describe geometry of the 3-fold $X_{\mathbb{C}}$, let $\mathrm{pr}_1\colon\mathbb{P}^1\times\mathbb{P}^2\to\mathbb{P}^1$ and $\mathrm{pr}_2\colon\mathbb{P}^1\times\mathbb{P}^2\to\mathbb{P}^2$
be the projections to the first and the second factors, respectively.
Then $\mathrm{pr}_2(S)$ is a~smooth conic in $\mathbb{P}^2$.
Let $\widetilde{S}$ be the proper transform on $X_{\mathbb{C}}$ of the surface $S$, let $E$ be the $\pi$-exceptional surface,
and let $H_2=(\mathrm{pr}_2\circ\pi)^*(\mathcal{O}_{\mathbb{P}^2}(1))$.
Then $\widetilde{S}\sim 2H_2-E$.
Set $H_1=(\mathrm{pr}_1\circ\pi)^*(\mathcal{O}_{\mathbb{P}^1}(1))$. Then
$$
-K_{X_{\mathbb{C}}}\sim_{\mathbb{Q}}2H_1+\frac{3}{2}\widetilde{S}+\frac{1}{2}E.
$$
Note that $\widetilde{S}\cong\mathbb{P}^1\times\mathbb{P}^1$ and $\widetilde{S}\vert_{\widetilde{S}}$ is a~line bundle  of degree $(-1,-1)$.
Therefore, there exists a~birational morphism $\varpi\colon X_{\mathbb{C}}\to Y$
such that $Y$ is a~singular Fano 3-fold that has one isolated ordinary double point,
the morphism $\varpi$  contracts $\widetilde{S}$ to the singular point of the 3-fold $Y$, and $-K_Y^3=22$.
Using this, we obtain the following commutative diagram:
$$
\xymatrix@R=1em{
&&\mathbb{P}^1\times\mathbb{P}^2\ar[dll]_{\mathrm{pr}_1}\ar[drr]^{\mathrm{pr}_2}&&\\
\mathbb{P}^1&&&& \mathbb{P}^2\\
&&X_{\mathbb{C}}\ar[ull]_{\phi_1}\ar[urr]^{\phi_2}\ar[uu]_{\pi}\ar[dd]_{\varpi}\ar[dll]_{\sigma_1}\ar[drr]^{\sigma_2}&&\\
V\ar[drr]_{\psi_1}\ar[uu]^{\varphi_1}&&&&U\ar[dll]^{\psi_2}\ar[uu]_{\varphi_2}\\
&&Y&&}
$$
where $V$ and $U$ are smooth weak Fano 3-folds, $\phi_1$ is a~fibration into quartic del Pezzo surfaces,
$\phi_2$ is a~conic bundle, $\sigma_1$ and $\sigma_2$ are birational contractions of the surface $\widetilde{S}$ to smooth rational curves,
$\psi_1$ and $\psi_2$ are small resolutions of the 3-fold $Y$, $\phi_1$ is a~fibration into quintic del Pezzo surfaces,
and $\phi_2$ is a~$\mathbb{P}^1$-bundle.

Now, we are ready to prove that $X_{\mathbb{C}}$ is K-polystable.
Let $p$ be a general point in $Z_{\mathbb{C}}$, and let $F$ be the fiber of the del Pezzo fibration $\phi_1$ that contains $p$.
Then $F$ is a quartic del Pezzo surface with Du Val singularities, and $\delta_p\big(X_{\mathbb{C}}\big)\leqslant 1$
for every point $p\in Z_{\mathbb{C}}$. On the other hand, if $p\in\widetilde{S}$, then $\delta_{p}(X_{\mathbb{C}})>1$ by \cite[Lemma 5.68]{Book}.
Moreover, if $F$ is smooth, then it follows from  \cite[Lemma 5.69]{Book} that $\delta_{p}(X_{\mathbb{C}})>1$.
Thus, we conclude that $p\not\in\widetilde{S}$, and the surface $F$ is singular.
The latter implies that $Z_{\mathbb{C}}\subset F$, since we assume that $p$ is a general point in $Z_{\mathbb{C}}$.
In particular, we see that $F$ is defined over $\Bbbk$, because $Z$ is defined over $\Bbbk$.
On the other hand, if $F$ has only ordinary double points,
then it follows from the proof of \cite[Main Theorem]{Denisova2024} that $\delta_{p}(X_{\mathbb{C}})>1$.
Hence, we conclude that $F$ has a singular point that is not an ordinary double points.
Now, using classification of singular quartic del Pezzo surfaces with Du Val singularities \cite{CorayTsfasman1988},
we see that $F$ has a unique such singular point of $F$ is unique, so it must be defined over $\Bbbk$,
which contradicts the assumption $X(\Bbbk)=\varnothing$.
\end{proof}

\begin{lemma}
\label{lemma:pointless-3-6}
Suppose that $X$ is contained in Family \textnumero 3.6 and $X(\Bbbk)=\varnothing$.
Then $X_{\mathbb{C}}$ is K-polystable.
\end{lemma}

\begin{proof}
The required assertion follows from \cite{Cheltsov2024}.
Indeed, it follows from \cite{MoMu81,MoMu83} that $X_{\mathbb{C}}$ can be obtained by blowing up $\mathbb{P}^3$ along
a disjoint union of a line and a smooth quartic elliptic curve.
Let $L$ be this line in $\mathbb{P}^3$, let $C_4$ be this smooth quartic elliptic curve in $\mathbb{P}^3$ such that $L\cap C_4=\varnothing$,
and let~$\pi\colon X\to\mathbb{P}^3$ be the blowup of these two  curves.
Then we can choose coordinates $x_0$, $x_1$, $x_2$, $x_3$ on $\mathbb{P}^3$ such that
$$
C_4=\big\{x_0^2+x_1^2+\lambda(x_2^2+x_3^2)=0,\lambda(x_0^2-x_1^2)+x_2^2-x_3^2=0\big\}\subset\mathbb{P}^3
$$
for some complex number $\lambda\not\in\{0,\pm 1,\pm i\}$, and
$$
L=\big\{a_0x_0+a_1x_1+a_2x_2=0,b_1x_1+b_2x_2+b_3x_3=0\big\}\subset\mathbb{P}^3
$$
for some $[a_0:a_1:a_2]$ and $[b_1:b_2:b_3]$ in $\mathbb{P}^2$.
Then we have the following commutative diagram:
$$
\xymatrix@R=1em{
&&\mathbb{P}^1\times\mathbb{P}^1\ar@/^1pc/@{->}[drr]^{\mathrm{pr}_2}\ar@/_1pc/@{->}[dll]_{\mathrm{pr}_1}&&\\%
\mathbb{P}^1&&X_{\mathbb{C}}\ar@{->}[d]_{\pi}\ar@{->}[u]_{\eta}\ar@{->}[rr]^{\phi}\ar@{->}[ll]_{\sigma}&&\mathbb{P}^1\\%
&&\mathbb{P}^3\ar@{-->}[urr]_{\varphi}\ar@{-->}[ull]^{\varsigma}&&}
$$
where $\varsigma$ is given by $[x_0:x_1:x_2:x_3]\mapsto[a_0x_0+a_1x_1+a_2x_2:b_1x_1+b_2x_2+b_3x_3]$,
the map $\varphi$ is given~by
$$
[x_0:x_1:x_2:x_3]\mapsto\big[x_0^2+x_1^2+\lambda(x_2^2+x_3^2):\lambda(x_0^2-x_1^2)+x_2^2-x_3^2\big],
$$
the map $\sigma$ is a fibration into quintic del Pezzo surfaces,
$\phi$ is a fibration into sextic del Pezzo surfaces,
the map $\eta$ is a conic bundle, $\mathrm{pr}_1$ and $\mathrm{pr}_2$ are projections to the first and the second factors, respectively.

Suppose that $X_{\mathbb{C}}$ is not K-polystable.
Then, arguing as in the proof of Lemma~\ref{lemma:pointless-3-5}, we see that $X$ contains a~geometrically irreducible curve $Z$ defined over $\Bbbk$
such that $\delta_{p}(X_{\mathbb{C}})\leqslant 1$ for every point $p\in Z_{\mathbb{C}}$.
Let us show that this leads to a contradiction.

Set $H=\pi^*(\mathcal{O}_{\mathbb{P}^3}(1))$. Let $E$ and $R$ be the exceptional surfaces of the blowup $\pi$ such that $\pi(E)=C_4$ and $\pi(R)=L$.
Then the quintic del Pezzo fibration $\sigma$ is given by the pencil $|H-R|$, the sextic del Pezzo fibration $\varphi$ is given by the pencil $|2H-E|$, the conic bundle $\eta$ is given by $|3H-E-R|$.
Note that $\mathrm{Eff}(X)=\langle E, R, H-R, 2H-E\rangle$,
and the Mori cone $\overline{\mathrm{NE}(X)}$ is generated by the classes of curves contracted by the blow up $\pi\colon X\to\mathbb{P}^3$ and the conic bundle $\eta\colon X\to\mathbb{P}^1\times\mathbb{P}^1$.

Fix a general point $p\in Z_{\mathbb{C}}$.
Let $S$ be the surface in the pencil $|H-R|$ that contains $p$, and let $u$ be a non-negative real number.
Then $S$ is a del Pezzo surface with at most Du Val singularities, and the divisor $-K_X-uS$ is pseudo-effective if and only if $u\leqslant 2$.
For~$u\in[0,2]$, let $P(u)$ be the positive part of the Zariski decomposition of the divisor $-K_X-uS$,
and let $N(u)$ be its negative part. Then
$$
P(u)\sim_{\mathbb{R}} \left\{\aligned
&(4-u)H-E+(u-1)R \ \text{ if } 0\leqslant u\leqslant 1, \\
&(4-u)H-E\ \text{ if } 1\leqslant u\leqslant 2,
\endaligned
\right.
$$
and
$$
N(u)= \left\{\aligned
&0\ \text{ if } 0\leqslant u\leqslant 1, \\
&(u-1)R\ \text{ if } 1\leqslant u\leqslant 2,
\endaligned
\right.
$$
which gives $S_{X}(S)=\frac{1}{22}\int\limits_{0}^{2}P(u)^3du=\frac{67}{88}$.
Now, for every prime divisor $F$ over the surface $S$, we set
$$
S\big(W^S_{\bullet,\bullet};F\big)=\frac{3}{(-K_X)^3}\int\limits_0^{\tau}\mathrm{ord}_F\big(N(u)\vert_{S}\big)\big(P(u)\vert_{S}\big)^2du+\frac{3}{(-K_X)^3}\int\limits_0^\tau\int\limits_0^{\infty}\mathrm{vol}\big(P(u)\big\vert_{S}-vF\big)dvdu,
$$
where $(-K_{X})^3=22$. Then, following \cite{AbbanZhuang,Book}, we let
$$
\delta_p\big(S,W^S_{\bullet,\bullet}\big)=\inf_{\substack{F/S\\p\in C_S(F)}}\frac{A_S(F)}{S\big(W^S_{\bullet,\bullet};F\big)},
$$
where the infimum is taken by all prime divisors over the surface $S$ whose center on $S$ contains $p$.
Then it follows from \cite{AbbanZhuang,Book} that
$$
1\geqslant\delta_{p}\big(X_{\mathbb{C}}\big)\geqslant\min\Bigg\{\frac{1}{S_X(S)},\delta_p\big(S,W^S_{\bullet,\bullet}\big)\Bigg\}.
$$
Therefore, since $S_{X}(S)<1$, we conclude that $\delta_p(S,W^S_{\bullet,\bullet})\leqslant 1$.

If $\sigma(Z_{\mathbb{C}})=\mathbb{P}^1$, then $S$ is a general fiber of the del Pezzo fibration $\sigma$, which implies, in particular, that the surface $S$ is smooth.
In this case, we know from \cite{Cheltsov2024} that  $\delta_p(S,W^S_{\bullet,\bullet})>1$, which is a contradiction.
Thus, we conclude that the surface $S$ is singular, and $\sigma(Z_{\mathbb{C}})$ is a point in $\mathbb{P}^1$. This means that $S$ is the unique fiber of the the del Pezzo fibration that contains the curve $Z_{\mathbb{C}}$.
Hence, since $Z$ is defined over $\Bbbk$, the surface $S$ must also be defined over $\Bbbk$.
Now, using classification of singular quintic del Pezzo surfaces with Du Val singularities \cite{CorayTsfasman1988},
we see that the worst singular point of the surface $S$ is unique unless $S$ has exactly two ordinary double points.
Thus, we conclude that $S$ has exactly two ordinary double points,
because otherwise its worst singular point would be defined over $\Bbbk$, which is impossible as $X(\Bbbk)=\varnothing$.

Recall that the divisor $-K_{S}$ is very ample, so we can identify $S$ with its anticanonical image in $\mathbb{P}^5$.
Then it follows from \cite{CorayTsfasman1988} that $S$ contains a unique line that passes through two singular points of the surface $S$.
Denote this line by $\ell$. Note that the line $\ell$ is also defined over $\Bbbk$.
Moreover, we have $\ell^2=0$ and the linear system $|2\ell|$ is a pencil that is free from base points.
Observe that the pencil $|2\ell|$ contains exactly two singular curves: the curve $2\ell$
and another curve, let us denote it by $C$, that is a union of two lines $\ell_1$ and $\ell_2$ such that the intersection $\ell_1\cap\ell_2$ consists of a single point.
This shows that the singular curve $C=\ell_1+\ell_2$ is defined over $\Bbbk$ and, therefore, the point $\ell_1\cap\ell_2$ is also defined over $\Bbbk$,
which contradicts our assumption $X(\Bbbk)=\varnothing$.
\end{proof}

\begin{lemma}
\label{lemma:pointless-3-7}
Suppose that $X$ is contained in Family \textnumero 3.7 and $X(\Bbbk)=\varnothing$.
Then $X_{\mathbb{C}}$ is K-polystable.
\end{lemma}

\begin{proof}
It follows from \cite{MoMu81,MoMu83,Matsuki} that there exists the following diagram
$$
\xymatrix@R=1em{
&X\ar@{->}[ld]_{\pi}\ar@{->}[rd]^{\phi}&&\\%
W&&C}
$$
where $W$ is form of the smooth divisor of degree $(1,1)$ in $\mathbb{P}^2\times\mathbb{P}^2$,
$C$ is a smooth conic in $\mathbb{P}^2$ that is defined over $\Bbbk$,
the morphism $\pi$ is the blowup of a smooth elliptic curve $\mathscr{C}$ that is defined over $\Bbbk$,
and $\phi$ is a morphism such that a general fiber of $\phi_{\mathbb{C}}$ is a smooth del Pezzo surface of degree $6$.

Let $E$ be the $\pi$-exceptional surface, let $p$ be a point in $Z_{\mathbb{C}}$, and let $S$ be the fiber of $\phi_{\mathbb{C}}$ that contains $p$.
Then $E\simeq\mathscr{C}\times C$, the surface $S$ has isolated singularities, and $\delta_p(X_{\mathbb{C}})\leqslant 1$.
Moreover, if $S$ is Du Val,
then it follows from the proof of \cite[Lemma 2.1]{CheltsovDenisovaFujita} that
\begin{equation}
\label{equation:3-7a}
\delta_p(X_{\mathbb{C}})\geqslant\left\{\aligned
&\mathrm{min}\Big\{\frac{16}{11},\frac{16}{15}\delta_p(S)\Big\}\ \text{if $p\not\in E_{\mathbb{C}}$}, \\
&\mathrm{min}\Big\{\frac{16}{11},\frac{16\delta_{p}(S)}{\delta_{p}(S)+15}\Big\}\ \text{if $p\in E_{\mathbb{C}}$}.
\endaligned
\right.
\end{equation}
Furthermore, if $S$ is smooth, then it follows from \cite{Denisova-DP} that
\begin{equation}
\label{equation:3-7b}
\delta_p(S)=\left\{\aligned
&1\ \text{if $p$ is contained in a $(-1)$-curve in $S$}, \\
&\frac{6}{5}\ \text{otherwise}.
\endaligned
\right.
\end{equation}

Suppose that $\phi(Z_{\mathbb{C}})$ is a point in $C_{\mathbb{C}}\simeq\mathbb{P}^1$.
Then $Z_{\mathbb{C}}\subset S$, which implies that $C\simeq\mathbb{P}^1$ and $S$ is defined over $\Bbbk$.
If $S$ is smooth, \eqref{equation:3-7a} and \eqref{equation:3-7b} implies that $p\in E_{\mathbb{C}}$ and $p$ is contained in a $(-1)$-curve in $S$.
Keeping in mind that $p$ is any point in $Z_{\mathbb{C}}$, we conclude that $Z_{\mathbb{C}}\subset E_{\mathbb{C}}\vert_{S}$ and $Z_{\mathbb{C}}$ is a $(-1)$-curve in $S$,
which is impossible, since $E_{\mathbb{C}}\vert_{S}$ is a smooth elliptic curve isomorphic to $\mathscr{C}_{\mathbb{C}}$.
Thus, $S$ is singular.

Recall from \cite{Kojima2020} that $S$ can have at most one non-Du Val singular point.
Hence, since $S$ is defined over $\Bbbk$ and $S$ does contain $\Bbbk$-points,
we conclude that $S$ has Du Val singularities, and it has at least $2$ singular points such that non of them is defined over $\Bbbk$.
Now, using \cite[Proposition~8.3]{CorayTsfasman1988}, we see that $S$ has exactly two isolated ordinary double points.
Moreover, it follows from \cite[Proposition~8.3]{CorayTsfasman1988} that $S$ contains a unique smooth geometrically rational curve $\ell$ such that $\ell^2=-\frac{1}{2}$,
and this curve contains exactly one singular point of $S$. This implies that $\ell$ and this singular point are both defined over $\Bbbk$,
which contradicts our assumption $X(\Bbbk)=\varnothing$.
Thus, we see that $\phi(Z_{\mathbb{C}})$ is not a point in $C_{\mathbb{C}}\simeq\mathbb{P}^1$, so $\phi(Z)=C$.

From now one, we assume that $p$ is a general point of the curve $Z_{\mathbb{C}}$.
Then the surface $S$ is smooth.
As above, using \eqref{equation:3-7a} and \eqref{equation:3-7b}, we see that $p\in E_{\mathbb{C}}$, which immediately gives $Z_{\mathbb{C}}\subset E_{\mathbb{C}}$.
Then $\pi(Z)=\mathscr{C}$.
Indeed, if $\pi(Z_{\mathbb{C}})$ is a point in $\mathscr{C}_{\mathbb{C}}$, this point would be defined over $\Bbbk$,
which would imply $W(\Bbbk)\ne\varnothing$, but $W(\Bbbk)=\varnothing$ by Lemma~\ref{lemma:nonempty}.
Hence, we see that $\pi(Z)=\mathscr{C}$. This easily leads to a contradiction.

Indeed, let $H$ be a divisor in $\mathrm{Pic}(W_{\mathbb{C}})$ such that $-K_{W_{\mathbb{C}}}\sim 2H$,
and let $u$ be a~non-negative real number. Then $-K_{X_{\mathbb{C}}}-uE_{\mathbb{C}}=\pi^*(2H)-(1+u)E_{\mathbb{C}}$,
which implies that $-K_{X_{\mathbb{C}}}-uE_{\mathbb{C}}$ is pseudoeffective if and only if  $u\leqslant 1$,
and for every $u\in[0,1]$, the divisor $-K_{X_{\mathbb{C}}}-uE_{\mathbb{C}}$ is nef.
This gives
$$
S_{X_{\mathbb{C}}}(E_{\mathbb{C}})=\frac{1}{24}\int_{0}^{1}\big(-K_{X_{\mathbb{C}}}-uE_{\mathbb{C}}\big)^3du=\frac{1}{24}\int_{0}^{1}6(2u^3-6u+4)du=\frac{3}{8}.
$$
Thus, it follows from \cite[Corollary 1.110]{Book} that $S(W_{\bullet,\bullet}^{E_{\mathbb{C}}}; Z_{\mathbb{C}})\geqslant 1$, where
$$
S\big(W_{\bullet,\bullet}^{E_{\mathbb{C}}}; Z_{\mathbb{C}}\big)=\frac{3}{24}\int_{0}^{1}\int_0^\infty \mathrm{vol}\Big(\big(-K_{X_{\mathbb{C}}}-uE_{\mathbb{C}}\big)\big\vert_{S}-vZ_{\mathbb{C}}\Big)dvdu.
$$
Recall that $E_{\mathbb{C}}\cong\mathscr{C}_{\mathbb{C}}\times C_{\mathbb{C}}\simeq\mathscr{C}_{\mathbb{C}}\times\mathbb{P}^1$.
Let $\mathbf{s}$ be a~fiber of the projection $\phi_{\mathbb{C}}\vert_{E_{\mathbb{C}}}\colon E_{\mathbb{C}}\to C_{\mathbb{C}}$,
and let $\mathbf{f}$ be a~fiber of the projection $\pi_{\mathbb{C}}\vert_{E_{\mathbb{C}}}\colon E_{\mathbb{C}}\to\mathscr{C}_{\mathbb{C}}$.
Then $Z_{\mathbb{C}}\equiv a\mathbf{s}+b\mathbf{f}$
for some non-negative integers $a$ and $b$. Moreover, we have $a\geqslant 1$, because $\pi(Z_{\mathbb{C}})=\mathscr{C}_{\mathbb{C}}$.

$$
1\leqslant S\big(W_{\bullet,\bullet}^{E_{\mathbb{C}}}; Z_{\mathbb{C}}\big)=S\big(W_{\bullet,\bullet}^{E_{\mathbb{C}}}; a\mathbf{s}+b\mathbf{f}\big)\leqslant S\big(W_{\bullet,\bullet}^{E_{\mathbb{C}}}; \mathbf{s}\big),
$$
where
\begin{align*}
S\big(W_{\bullet,\bullet}^{E_{\mathbb{C}}}; a\mathbf{s}+b\mathbf{f}\big)&=\frac{3}{24}\int_{0}^{1}\int_0^\infty \mathrm{vol}\Big(\big(-K_{X_{\mathbb{C}}}-uE_{\mathbb{C}}\big)\big\vert_{S}-v\big(a\mathbf{s}+b\mathbf{f}\big)\Big)dvdu,\\
S\big(W_{\bullet,\bullet}^{E_{\mathbb{C}}}; \mathbf{s}\big)&=\frac{3}{24}\int_{0}^{1}\int_0^\infty \mathrm{vol}\Big(\big(-K_{X_{\mathbb{C}}}-uE_{\mathbb{C}}\big)\big\vert_{S}-v\mathbf{s}\Big)dvdu.
\end{align*}
Hence, we see that $S(W_{\bullet,\bullet}^{E_{\mathbb{C}}}; \mathbf{s})\geqslant 1$.
But $S(W_{\bullet,\bullet}^{E_{\mathbb{C}}}; \mathbf{s})$ is easy to compute.
Namely, if $v\in\mathbb{R}_{\geqslant 0}$, then
$$
\big(-K_{X_{\mathbb{C}}}-uE_{\mathbb{C}}\big)\big\vert_{E_{\mathbb{C}}}-v\mathbf{s}\equiv (1+u-v)\mathbf{s}+6(1-u)\mathbf{f},
$$
which gives
$$
1\leqslant S\left(W_{\bullet,\bullet}^{E_{\mathbb{C}}};\mathbf{s}\right)=\frac{3}{24}\int_{0}^{1}\int_{0}^{1+u}\big((1+u-v)\mathbf{s}+6(1-u)\mathbf{f}\big)^2dvdu=\frac{3}{24} \int_{0}^{1} \int_{0}^{1+u} 12(1-u)(1+u-v)dvdu=\frac{11}{16}.
$$
The latter is absurd.
\end{proof}

\begin{lemma}
\label{lemma:pointless-3-10}
Suppose that $X$ is contained in Family \textnumero 3.10 and $X(\Bbbk)=\varnothing$.
Then $X_{\mathbb{C}}$ is K-polystable.
\end{lemma}

\begin{proof}
It follows from \cite{Matsuki, MoMu81,MoMu83} that there exists the following diagram
$$
\xymatrix@R=1em{
&X\ar@{->}[ld]_{\pi}\ar@{->}[rd]^{\eta}&&\\%
Q&&S}
$$
where $Q$ is form of a smooth quadric 3-fold,
$S$ is a form of a smooth quadric surface,
$\pi$ is the blowup of a geometrically reducible curve $C$ such that $C_{\mathbb{C}}=C_1+C_2$,
where $C_1$ and $C_2$ are disjoint conics in the quadric 3-fold $Q_{\mathbb{C}}\subset\mathbb{P}^4$,
and $\eta$ is a conic bundle.
Let $\mathscr{C}$ be the discriminant curve of the conic bundle $\eta$.
Then $\mathscr{C}_{\mathbb{C}}$ is a reduced curve in $S_{\mathbb{C}}\simeq\mathbb{P}^1\times\mathbb{P}^1$ that has degree $(2,2)$.

Over $\mathbb{C}$, we can choose coordinates $[x:y:z:t:w]$ on $\mathbb{P}^4$
such that
\begin{align*}
C_1&=\{x=0,y=0,w^2+zt=0\},\\
C_2&=\{z=0,t=0,w^2+xy=0\},
\end{align*}
and one of the following three cases hold:
\begin{equation}
\label{equation:beth}\tag{A}
Q_{\mathbb{C}}=\big\{w^2+xy+zt+a(xt+yz)+b(xz+yt)=0\big\},
\end{equation}
where $(a,b)\in\mathbb{C}^2$ such that $a\pm b\ne \pm 1$, or
\begin{equation}
\label{equation:gimel}\tag{B}
Q_{\mathbb{C}}=\big\{w^2+xy+zt+a(xt+yz)+xz=0\big\},
\end{equation}
where $a\in\mathbb{C}$ such that $a\ne\pm 1$, or
\begin{equation}
\label{equation:daleth}\tag{C}
Q_{\mathbb{C}}=\big\{w^2+xy+zt+xt+xz=0\big\}.
\end{equation}
Moreover, it follows from \cite[\S 5.17]{Book} that $X_{C}$ is K-polystable if and only if we are in case (C).

Over $\mathbb{C}$, we have the following commutative diagram:
$$
\xymatrix{
&&&&X_{\mathbb{C}}\ar@/_1pc/@{->}[dllll]_{\alpha_2}\ar@{->}[dll]_{\gamma_1}\ar@{->}[d]_{\eta_{\mathbb{C}}}\ar@/^1pc/@{->}[drrrr]^{\alpha_1}\ar@{->}[drr]^{\gamma_2}&&&&\\
Y_1\ar@{->}[rr]^{\beta_1}\ar@/_1pc/@{->}[drrrr]_{\pi_1}&&\mathbb{P}^1&&S_{\mathbb{C}}\ar@{->}[ll]_{\mathrm{pr}_1}\ar@{->}[rr]^{\mathrm{pr}_2}&&\mathbb{P}^1&& Y_2\ar@{->}[ll]_{\beta_2}\ar@/^1pc/@{->}[dllll]^{\pi_2}\\%
&&&&Q_{\mathbb{C}}\ar@{-->}[ull]^{\delta_1}\ar@{-->}[urr]_{\delta_2}\ar@{-->}[u]&&&&}
$$
where $\delta_1$ is the map given by $[x:y:z:t:w]\mapsto[x:y]$,
the map $\delta_2$ is given by $[x:y:z:t:w]\mapsto[z:t]$,
the maps $\pi_1$ and $\pi_2$ are blowups of the quadric $Q_{\mathbb{C}}$ along the smooth conics $C_1$ and $C_2$, respectively,
$\alpha_1$~and $\alpha_2$ are blowups of the proper transforms of~these conics, \mbox{respectively},
$\beta_1$ and $\beta_2$~are fibrations into quadric surfaces,
$\gamma_1$ and $\gamma_2$ are fibrations into sextic del Pezzo surfaces,
$\mathrm{pr}_1$ and $\mathrm{pr}_2$ are natural projections of $S_{\mathbb{C}}\simeq\mathbb{P}^1_{x,y}\times\mathbb{P}^1_{z,t}$ to its factors,
and the map $Q_{\mathbb{C}}\dasharrow S_{\mathbb{C}}$ is given by
$$
[x:y:z:t:w]\mapsto([x:y],[z:t]).
$$
Here and below, we identified $S_{\mathbb{C}}=\mathbb{P}^1_{x,y}\times\mathbb{P}^1_{z,t}$ with coordinates $([x:y],[z:t])$.

Over $\mathbb{C}$, the equation of the curve $\mathscr{C}_{\mathbb{C}}$ can be computed as follows.
If we are in case (A), it is given by
$$
a^2\big(x^2t^2+y^2z^2\big)+2ab\big(xyz^2+xyt^2+ztx^2+zty^2\big)+b^2\big(x^2z^2+y^2t^2\big)+2\big(a^2+b^2-2\big)yzxt=0.
$$
If we are in case (B), the curve $\mathscr{C}_{\mathbb{C}}$ is given in $S_{\mathbb{C}}=\mathbb{P}^1_{x,y}\times\mathbb{P}^1_{z,t}$ by the following equation:
$$
a^2t^2x^2+(2a^2-4)xyzt+2atzx^2+a^2y^2z^2+2ayz^2x+z^2x^2=0.
$$
If $a\ne 0$, the curve $\mathscr{C}_{\mathbb{C}}$ is irreducible that has a node at $([0:1],[0:1])$.
If $a=0$, then
$$
\mathscr{C}_{\mathbb{C}}=\{zx(zx-4yt)=0\}\subset\mathbb{P}^1_{x,y}\times\mathbb{P}^1_{z,t},
$$
so $\mathscr{C}_{\mathbb{C}}$ is a~union of curves of degrees $(0,1)$, $(1,0)$, $(1,1)$,
and
$$
\mathrm{Sing}\big(\mathscr{C}_{\mathbb{C}}\big)=\big\{([0:1],[0,1]), ([1:0],[0:1]), ([0:1],[1:0])\big\}.
$$
Finally, if we are in case (C), the curve $\mathscr{C}_{\mathbb{C}}$ is given by $x(t^2x+2txz-4tyz+xz^2)=0$,
so $\mathscr{C}_{\mathbb{C}}$ splits as a~union of a~curve of degree $(0,1)$ and a~smooth curve of degree $(2,1)$.

Now, we are ready to prove the assertion of the lemma.
Suppose the Fano 3-fold $X_{\mathbb{C}}$ is not K-polystable.
Then either we are in case (B), or we are in case (C).
Let us show that $X$ has a $\Bbbk$-point in both cases. We~will do this geometrically.

First, we consider case (C). In this case, we have $\mathscr{C}=L+Z$,
where $L$ and $Z$ are smooth geometrically rational curves in $S$ such that both of them are defined over $\Bbbk$,
$L_{\mathbb{C}}=\{x=0\}\subset S_{\mathbb{C}}\simeq\mathbb{P}^1\times\mathbb{P}^1$ and $Z_{\mathbb{C}}=\{t^2x+2txz-4tyz+xz^2=0\}$.
Set $\mathscr{H}=\pi_*(\eta^*(L))$.
Then $-K_Q\sim 2\mathscr{H}$, and the linear system $|\mathscr{H}|$ gives an embedding $Q\hookrightarrow\mathbb{P}^4$,
which implies that $Q$ is a pointless quadric 3-fold in $\mathbb{P}^4$.
Observe also that the surface $\mathscr{H}_{\mathbb{C}}$ is cut out on $Q_{\mathbb{C}}$ by the hyperplane $\{x=0\}$,
which implies that $\mathscr{H}_{\mathbb{C}}$ is a quadric cone with one singular point.
This shows that $\mathrm{Sing}(\mathscr{H}_{\mathbb{C}})$ is defined over $\Bbbk$ and, in particular, $Q(\Bbbk)\ne\varnothing$,
so that $X(\Bbbk)\ne\varnothing$ by Lemma~\ref{lemma:nonempty}, which contradicts to our assumption.

Thus, we are in case (B). Suppose that $a=0$.
Then $\mathscr{C}$ is reducible. Namely, we have $\mathscr{C}=\Delta+\Delta^\prime$
where $\Delta$ is a geometrically irreducible curve such that $\Delta_{\mathbb{C}}=\{zx-4yt=0\}\subset S_{\mathbb{C}}\simeq\mathbb{P}^1\times\mathbb{P}^1$,
$\Delta^\prime$ is a geometrically reducible curve such that $\Delta^\prime_{\mathbb{C}}=L_1+L_2$ for $L_1=\{z=0\}$ and $L_2=\{x=0\}$.
This implies that $S$ is a quadric surface in $\mathbb{P}^3$,
and the curves $\Delta$ and $\Delta^\prime$ are its hyperplane sections.
Moreover, the conic bundle $\eta_{\mathbb{C}}$ has exactly three non-reduced fibers: the fibers over the singular points of the curve~$\mathscr{S}_{\mathbb{C}}$.
These are the fibers over the points $L_1\cap\Delta$, $L_2\cap\Delta$, $L_1\cap L_2$.
Denote them by $F_1$, $F_2$, $F_3$, respectively.
Then $F_1=2\ell_1$, $F_2=2\ell_2$, $F_3=2\ell_3$, where $\ell_1$, $\ell_2$, $\ell_3$ are irreducible smooth curves.
Moreover, on $Q_{\mathbb{C}}\subset\mathbb{P}^4$, the curves $\pi_{\mathbb{C}}(\ell_1)$, $\pi_{\mathbb{C}}(\ell_2)$, $\pi_{\mathbb{C}}(\ell_3)$ are lines
that can be described as follows:
\begin{align*}
\pi_{\mathbb{C}}(\ell_1)&=\{x=0,t=0,w=0\},\\
\pi_{\mathbb{C}}(\ell_2)&=\{y=0,z=0,w=0\},\\
\pi_{\mathbb{C}}(\ell_3)&=\{x=0,z=0,w=0\}.
\end{align*}
Note that the hyperplane section $\{w=0\}\cap Q_{\mathbb{C}}$ is the unique hyperplane section of $Q_{\mathbb{C}}$ that contains all these three lines.
Thus, this hyperplane section is defined over $\Bbbk$, since the one-cycles $\ell_1+\ell_2$ and $\ell_3$ are defined over $\Bbbk$.
Therefore, as above, we see that $Q$ is a smooth quadric 3-fold in $\mathbb{P}^3$, and this quadric 3-fold that contains a line that is defined over $\Bbbk$ --- the image of the curve $\ell_3$.
This shows that $Q(\Bbbk)\ne\varnothing$, so $X(\Bbbk)\ne\varnothing$ by Lemma~\ref{lemma:nonempty}, which contradicts our assumption.

Hence, we see that $a\ne 0$.
Then $\mathscr{C}_{\mathbb{C}}$ is geometrically irreducible, and it has one singular point.
Moreover, the restrictions $\mathrm{pr}_1\vert_{\mathscr{C}_{\mathbb{C}}}\colon \mathscr{C}_{\mathbb{C}}\to\mathbb{P}^1$
and $\mathrm{pr}_2\vert_{\mathscr{C}_{\mathbb{C}}}\colon \mathscr{C}_{\mathbb{C}}\to\mathbb{P}^1$ are double covers such that each of them is ramified in
two points away from the singular point of the curve $\mathscr{C}_{\mathbb{C}}$.
These these ramification points are
$$
\big([-a:1],[1:0]\big), \big([a-a^3:1],[1-a^2:a]\big), \big([1:0],[-a:1]\big), \big([1-a^2:a],[a-a^3:1]\big).
$$
Note that the union of these four points is a zero-cycle in $\mathscr{C}$ that is defined over $\Bbbk$.
Moreover, one can check that the fibers of the conic bundle $\eta_{\mathbb{C}}$ over these four points are reducible reduced conics,
and the images of their singular points in $Q_{\mathbb{C}}$ via $\pi_{\mathbb{C}}$ are the following four points:
$$
[0:0:1:0:0], [a-a^3:1:a^2-1:-a:0], [1:0:0:0:0], [a^2-1:-a:a-a^3:1:0].
$$
Again, the union of these four points is a zero-cycle in $Q$ that is defined over $\Bbbk$.
One can check that these four points in $\mathbb{P}^4$ are in linearly general position, since $a\ne\pm 1$.
Thus, there exists a unique hyperplane section of the quadric $Q_{\mathbb{C}}$ that contains these four points,
which implies that it is also defined over $\Bbbk$. This implies that $Q$ is a smooth quadric 3-fold in $\mathbb{P}^4$.
Now, we let $F$ be the fiber of the conic bundle $\eta$ over the singular point of the curve $\mathscr{C}$.
Observe that $F$ is defined over $\Bbbk$, and $F=2\ell$, where $\ell$ is a geometrically irreducible curve in $X$ that is also defined over $\Bbbk$.
Then $\pi(\ell)$ is a line in $Q$, which implies that $Q(\Bbbk)\ne\varnothing$, so that $X(\Bbbk)\ne\varnothing$ by Lemma~\ref{lemma:nonempty},
which is a contradiction.
\end{proof}

\begin{lemma}
\label{lemma:pointless-3-12}
Suppose that $X$ is contained in Family \textnumero 3.12 and $X(\Bbbk)=\varnothing$.
Then $X_{\mathbb{C}}$ is K-polystable.
\end{lemma}

\begin{proof}
It follows from \cite{MoMu81,MoMu83,Matsuki} and Lemma~\ref{lemma:nonempty} that there exists
a birational morphism $\pi\colon X\to U$ such that $U$ is a pointless form of $\mathbb{P}^3$,
and $\pi$ is the blowup of two disjoint smooth geometrically irreducible and geometrically rational curves $L$ and $C$ such that $-K_{U}\cdot L=4$ and $-K_{U}\cdot C=12$.
Over $\mathbb{C}$, the curve $L_{\mathbb{C}}$ is a line in $U_{\mathbb{C}}\simeq\mathbb{P}^3$, and $C_{\mathbb{C}}$ is a twisted cubic.
Let $f\colon V\to U$ be the blowup of the curve $L$, and let $\widetilde{C}$ be the strict transform of the curve $C$ on  $V$.
Then there exists a Sarkisov link
$$
\xymatrix@R=1em{
&V\ar@{->}[ld]_{f}\ar@{->}[rd]^{g}&&\\%
U&&Z}
$$
where $Z$ is a conic in $\mathbb{P}^2$, and $g_{\mathbb{C}}$ is a $\mathbb{P}^2$-bundle over $Z_{\mathbb{C}}\simeq\mathbb{P}^1$.
Moreover, the map $g_{\mathbb{C}}$ induces a finite morphism $\omega\colon \widetilde{C}\to Z$ of degree $3$.
Furthermore, it follows from \cite{Denisova} that $X_{\mathbb{C}}$ is not K-polystable if and only if the triple cover $\omega_{\mathbb{C}}\colon \widetilde{C}_{\mathbb{C}}\to Z_{\mathbb{C}}$  has a unique ramification point of ramification index $3$.
Thus, if $X_{\mathbb{C}}$ is not K-polystable, then the set $\widetilde{C}(\Bbbk)$ is not empty --- it contains the ramification point of index $3$ of the finite morphism $\omega$,
so $X(\Bbbk)\ne\varnothing$ by Lemma~\ref{lemma:nonempty}.
So, if $X(\Bbbk)=\varnothing$, then $X_{\mathbb{C}}$ is K-polystable.
\end{proof}

\begin{lemma}
\label{lemma:pointless-3-13}
Suppose that $X$ is contained in Family \textnumero 3.13 and $X(\Bbbk)=\varnothing$.
Then $X_{\mathbb{C}}$ is K-polystable.
\end{lemma}

\begin{proof}
Let $V$ be the complete intersection in $\mathbb{P}^2\times\mathbb{P}^2\times\mathbb{P}^2$
that is given by the following system of equations:
$$
\left\{\aligned
&x_1y_1+x_2y_2+x_3y_3=0,\\
&y_1z_1+y_2z_2+y_3z_3=0, \\
&x_1z_2+x_2z_1+x_2z_3-x_3z_2-2x_3z_3=0,
\endaligned
\right.
$$
where $([x_1:x_2:x_3],[y_1:y_2:y_3],[z_1:z_2:z_3])$ are coordinates on the product $\mathbb{P}^2_{x_1,x_2,x_3}\times\mathbb{P}^2_{y_1,y_2,y_3}\times\mathbb{P}^2_{z_1,z_2,z_3}$.
If $X_{\mathbb{C}}$ is not K-polystable, then it follows from \cite[\S~5.19]{Book} that $X$ is a $\Bbbk$-form of $V$.
Let us show that any $\Bbbk$-form of $V$ has a $\Bbbk$-point. Set
\begin{align*}
W_{x,y}&=\{x_1y_1+x_2y_2+x_3y_3=0\}\subset\mathbb{P}^2_{x_1,x_2,x_3}\times\mathbb{P}^2_{y_1,y_2,y_3},\\
W_{y,z}&=\{y_1z_1+y_2z_2+y_3z_3=0\}\subset\mathbb{P}^2_{y_1,y_2,y_3}\times\mathbb{P}^2_{z_1,z_2,z_3},\\
W_{x,z}&=\{x_1z_2+x_2z_1+x_2z_3-x_3z_2-2x_3z_3=0\}\subset\mathbb{P}^2_{x_1,x_2,x_3}\times\mathbb{P}^2_{z_1,z_2,z_3}.
\end{align*}
Then $W_{x,y}$, $W_{y,z}$, $W_{x,z}$ are smooth,
and natural projections of $\mathbb{P}^2_{x_1,x_2,x_3}\times\mathbb{P}^2_{y_1,y_2,y_3}\times\mathbb{P}^2_{z_1,z_2,z_3}$ to its factor
induce birational morphisms $\pi_{x,y}\colon V\to W_{x,y}$, $\pi_{y,z}\colon V\to W_{y,z}$, $\pi_{x,z}\colon V\to W_{x,z}$.
Let $E_{x,y}$, $E_{y,z}$, $E_{x,z}$ be the exceptional surfaces of the morphisms $\pi_{x,y}$, $\pi_{y,z}$, $\pi_{x,z}$, respectively.
Then
\begin{align*}
E_{x,y}&=\{x_1y_3-x_2y_2+2x_3y_2-x_3y_3=0, x_1y_1-x_2y_2-x_3y_1=0, x_2y_1-x_2y_3-2x_3y_1=0\}\cap W_{x,y},\\
E_{y,z}&=\{y_2z_2+2y_2z_3+y_3z_1+y_3z_3=0, y_1z_1+y_1z_3-y_2z_2=0, y_1z_2+2y_1z_3+y_3z_2=0\}\cap W_{y,z},\\
E_{x,z}&=\{x_2z_3-x_3z_2=0, x_1z_2-x_2z_1=0, x_1z_3-x_3z_1=0\}\cap W_{x,z}.
\end{align*}
The divisors $E_{x,y}$, $E_{y,z}$, $E_{x,z}$ generate the cone of effective divisors of the 3-fold $V_{\mathbb{C}}$.
Over $\mathbb{C}$, we have
$$
E_{x,y}\cap E_{y,z}\cap E_{x,z}=\big([1:0:0],[0:1:0],[0:0:1]\big).
$$
This shows that every $\Bbbk$-form of $V$ contains a $\Bbbk$-point ---
the unique singular point of multiplicity $3$ of the unique effective reduced divisor that splits over $\mathbb{C}$ into the sum of three surfaces that generate the cone of effective divisors of the geometric model of the $\Bbbk$-form.
\end{proof}

\begin{lemma}
\label{lemma:pointless-4-13}
Suppose that $X$ is contained in Family \textnumero 4.13 and $X(\Bbbk)=\varnothing$.
Then $X_{\mathbb{C}}$ is K-polystable.
\end{lemma}

\begin{proof}
Using \cite{MoMu81,MoMu83,Matsuki}, we see that there is a~birational morphism $\pi\colon X\to S\times Z$
such that $S$ is a possibly pointless form of $\mathbb{P}^1\times\mathbb{P}^1$, $Z$ is a conic in $\mathbb{P}^2$,
and $\pi$ is the blowup of a smooth geometrically irreducible and geometrically rational curve $C$
such that $C_{\mathbb{C}}$ is a curve of degree $(1,1,3)$ in $S_{\mathbb{C}}\times Z_{\mathbb{C}}\simeq\mathbb{P}^1\times\mathbb{P}^1\times\mathbb{P}^1$.
Over $\mathbb{C}$, one can choose coordinates $([x_0:x_1],[y_0:y_1],[z_0:z_1])$ on $\mathbb{P}^1\times\mathbb{P}^1\times\mathbb{P}^1$ such that
$C_{\mathbb{C}}$ is given by one of the following two equations:
$$
x_0y_1-x_1y_0=x_0^3z_0+x_1^3z_1+\lambda\big(x_0x_1^2z_0+x_0^2x_1z_1\big)=0
$$
for some $\lambda\in\mathbb{C}\setminus\{\pm 1,\pm 3\}$, or
\begin{equation}
\label{equation:4-13-unstable}
x_0y_1-x_1y_0=x_0^3z_0+x_1^3z_1+x_0x_1^2z_0=0.
\end{equation}
Moreover, it follows \cite[\S~5.12]{Book} that $X_{\mathbb{C}}$ is always K-semistable,
and $X_{\mathbb{C}}$ is not K-polystable if and only if the curve $C_{\mathbb{C}}$ can be given by the equation~\eqref{equation:4-13-unstable}.
Note that in this (non-K-polystable) case, the natural triple cover $C_{\mathbb{C}}\to Z_{\mathbb{C}}$ has a unique ramification point with ramification index $3$,
which implies that this point is defined over $\Bbbk$ and, in particular, the set $C(\Bbbk)$ is not empty.
Hence, if $X_{\mathbb{C}}$ is not K-polystable, then $X(\Bbbk)\ne\varnothing$ by Lemma~\ref{lemma:nonempty},
which contradicts our assumption.
\end{proof}

\section{Examples of pointless smooth Fano 3-folds}
\label{section:pointless}

In this section, we provide examples of pointless smooth Fano 3-folds in the families studied in Section\,\ref{section:pointless-3-folds}.

\begin{example}
\label{example:1-9}
It follows from \cite{SV00} that there exists a unique pointless $\mathbb{Q}$-form $V$ of the five-dimensional homogeneous space $G_2/P$
of the exceptional simple algebraic group of type $G_2$ by a maximal parabolic subgroup $P$.
By \cite[Theorem 3.1]{LandsbergManivel}, we can describe $V$ as the subvariety of the Grassmannian $\mathrm{Gr}(2,V_7)$ parametrizing
planes on which the octonionic multiplication is identically zero, where $V_7$ is a seven-dimensional vector space of imaginary octonions.
Note that $V$ is a $G_2$-torsor  defined by the unique non-zero pure Rost symbol $\{-1,-1,-1\}$ in the Milnor K-theory $K^M_3(\mathbb{Q})/2$ modulo $2$.
It follows from \cite{Mukai} that $V_{\mathbb{C}}$ is a smooth Fano 5-fold, and its Picard group $\operatorname{Pic}(V_{\mathbb{C}})$
is generated by a divisor $H$ with $-K_{V_{\mathbb{C}}}\sim 3H$ and $H^5=18$.
We also know that $|H|$ gives an embedding $V_{\mathbb{C}}\hookrightarrow \mathbb{P}^{13}$.
Since the class $H$ is $\operatorname{Gal}(\overline{\mathbb{Q}}\slash\mathbb{Q})$-invariant, we also have an embedding $V\hookrightarrow U$ into a $\Bbbk$-form $U$ of  $\mathbb{P}^{13}$.
Set $D=K_U\vert_V-5K_V$. Then $D$ is defined over $\mathbb{Q}$ and $D_{\mathbb{C}}\sim H_{\mathbb{C}}$,
so $|D|$ gives an embedding of $V$ into $\mathbb{P}^{13}$. Now we take $X=D_1\cap D_2$ where $D_1$ and $D_2$ are general divisors in $|D|$.
Then $X$ is a pointless $\mathbb{Q}$-form of a smooth Fano 3-fold belonging to Family \textnumero 1.9.
\end{example}

\begin{example}
\label{example:1-10}
It follows from \cite[Theorem 1.1]{Kollar-Olaf} that there exists a non-empty connected family of smooth prime Fano 3-folds $X$ of degree $22$
defined over $\mathbb{R}$ such that $X_{\mathbb{C}}$ is smooth Fano 3-fold in Family \textnumero 1.10 and $X(\mathbb{R})=\varnothing$.
\end{example}
\begin{example}
\label{example:2-5}
Let $V$ be a smooth cubic 3-fold in $\mathbb{P}^4$ defined over $\mathbb{Q}$ without $\mathbb{Q}$-points \cite{Mordell},
let $C$ be an intersection of $V$ with any codimension two linear subspace such that $C$ is smooth,
and let $\pi\colon X\to V$ be the blowup of the curve $C$.
Then $X(\mathbb{Q})=\varnothing$, and $X_{\mathbb{C}}$ is a smooth Fano 3-fold in Family \textnumero 2.5.
To present an explicit example of $V$ and $C$, let
$$
Y=\big\{x_1^3+2x_2^3+4x_3^3+x_1x_2x_3+7(x_4^3+2x_5^3+4x_6^3+x_4x_5x_6)=0\big\}\subset\mathbb{P}^5,
$$
where $x_1,x_2,x_3,x_4,x_5,x_6$ are coordinates on $\mathbb{P}^5$. Then $Y$ is a smooth cubic 4-fold defined over $\mathbb{Q}$,
which does not contain $\mathbb{Q}$-points \cite{DaiXu}.
Indeed, the congruence $x_1^3+2x_2^3+4x_3^3+x_1x_2x_3\equiv 0\ \mathrm{mod}\ 7$ has only trivial solutions,
which easily implies that $Y$ does not have points over $\mathbb{Q}_{7}$, so it does not have $\mathbb{Q}$-points either.
Now, we can let $V=Y\cap\{x_6=0\}$ and $C=V\cap\{x_4=x_5=0\}$.
\end{example}

\begin{example}
\label{example:2-10}
Let $V$ be a pointless smooth complete intersection of two quadrics in $\mathbb{P}^5$ defined over $\mathbb{R}$.
For instance, let
$$
V=\Big\{\sum_{i=1}^6 x_i^2 = \sum_{i=1}^6 a_i x_i^2 =0\Big\}\subset\mathbb{P}^5,
$$
where $a_1,a_2,a_3,a_4,a_5,a_6$ are real numbers such that $a_i\neq a_j$ whenever $i \neq j$,
and  $x_1,x_2,x_3,x_4,x_5,x_6$ are coordinates on $\mathbb{P}^5$.
Now, we take $C$ to be an intersection of $V$ with any codimension two linear subspace such that $C$ is smooth,
e.g.\ $C=\{x_0=x_1=0\}\cap V$, and let $\pi\colon X\to V$ be the blowup of the curve $C$.
Then $X(\mathbb{R})=\varnothing$, and $X_{\mathbb{C}}$ is a smooth Fano 3-fold in Family \textnumero 2.10.
\end{example}

\begin{example}
\label{example:2-12}
Explicit examples of real pointless smooth Fano 3-folds whose geometric models belong to Family 2.12 have been constructed in \cite{CheltsovLiMauPinardin}.
\end{example}

\begin{example}
\label{example:2-13}
Over $\mathbb{R}$, let
$$
C=\{x^6+x^4y^2+x^2y^4+y^6+z^2=0\}\subset\mathbb{P}(1_{x},1_{y},3_{z}),
$$
and let $\phi\colon \mathbb{P}(1_{x},1_{y},3_{z})\to\mathbb{P}^4$ given by $[x:y:z]\mapsto[x^3:x^2y:xy^2:y^3:z]$.
Then $C$ is a smooth pointless real hyperelliptic curve of genus $2$,
and $\phi(C)\simeq C$ is a curve of degree $6$ that is contained in the smooth pointless real quadric 3-fold
$$
Q=\big\{x_1^2+x_2^2+x_3^2+x_4^2+x_5^2=0\big\}\subset\mathbb{P}^4,
$$
where $x_1,x_2,x_3,x_4,x_5$ are projective coordinates on $\mathbb{P}^4$.
Let $\pi\colon X\to Q$ be the blowup of the curve $\phi(C)$.
Then $X_{\mathbb{C}}$ is a smooth Fano 3-fold in Family \textnumero 2.13, and $X(\mathbb{R})=\varnothing$ by Lemma~\ref{lemma:nonempty}.
\end{example}

\begin{example}
\label{example:2-16}
Over $\mathbb{R}$, let
$$
C=\{x_1^2+x_2^2+x_3^2=0,x_5=0,x_6=0,x_6=0\}\subset \mathbb{P}^5,
$$
and let $V$ be the complete intersection of two quadrics in $\mathbb{P}^5$ that is given by
$$
\left\{\aligned
&x_1^2+x_2^2+x_3^2+x_4^2+x_5^2+x_6^2=0, \\
&1983x_1x_4+1973x_2x_5+1967x_3x_6=0,
\endaligned
\right.
$$
where  $x_1,x_2,x_3,x_4,x_5,x_6$ are coordinates on $\mathbb{P}^5$.
Then both $C$ and $V$ are smooth and pointless.
Let $\pi\colon X\to V$ be the blowup of the conic $C$.
Then $X_{\mathbb{C}}$ is a smooth Fano 3-fold in Family \textnumero 2.16, and $X(\mathbb{R})=\varnothing$ by Lemma~\ref{lemma:nonempty}.
\end{example}


\begin{example}
\label{example:2-19}
Let $U$ be the unique real form of $\mathbb{P}^3$ that has no real points.
Then $\mathrm{Pic}(U)=\mathbb{Z}[Q]$ with $-K_U\sim 2Q$,
and $U$ contains a twisted line $L$, that is, $L_{\mathbb{C}}$ is a line in $U_{\mathbb{C}}\simeq\mathbb{P}^3$.
Let $S$ be a general surface in $|Q|$ containing $L$.
Then $S_{\mathbb{C}}\simeq\mathbb{P}^1\times\mathbb{P}^1$,
since otherwise $S_{\mathbb{C}}$ would be a quadric cone whose vertex yields an $\mathbb{R}$-point in $U$.
Since $S$ contains $L$, it follows that $S$ is isomorphic to $\mathbb{P}^1\times C$ for some pointless conic $C$.
Now we let $\Gamma$ be a general curve in $|-K_S+L|$,
and let $X\to U$ be the blowup of $U$ along $\Gamma$.
Then $X_{\mathbb{C}}$ is a smooth Fano 3-fold in Family \textnumero 2.19,
and it follows from Lemma~\ref{lemma:nonempty} that $X(\mathbb{R})=\varnothing$, as $U(\mathbb{R})=\varnothing$.
\end{example}

\begin{example}
\label{example:2-21}
Let $C$ be the conic $\{x^2+y^2+z^2=0\}\subset\mathbb{P}^2$,
where $x$, $y$, $z$ are projective coordinates on $\mathbb{P}^2$.
Then $C$ is smooth and without $\mathbb{R}$-points.
Let $\phi\colon \mathbb{P}^2\to\mathbb{P}^5$ be the second Veronese embedding given by $[x:y:z] \mapsto [x^2:y^2:z^2:xy:xz:yz]$.
Then
$$
\phi(C)=\big\{F_1=F_2=F_3=F_4=F_5=F_6=x_1+x_2+x_3=0\big\}\subset\mathbb{P}^5,
$$
where $F_1=x_1x_2-x_4^2$, $F_2=x_1x_3-x_5^2$, $F_3=x_2x_3-x_6^2$, $F_4=x_1x_6-x_4x_5$, $F_5=x_2x_5-x_4x_6$, $F_6=x_3x_4-x_5x_6$,
and $x_1,x_2,x_3,x_4,x_5,x_6$ are coordinates on $\mathbb{P}^5$.
Let
$$
\widetilde{Q}=\Big\{ \sum_{i=0}^5 x_i^2=0\Big\}=\big\{(x_1+x_2+x_3)^2 -2 (F_1+F_2+F_3)=0\big\}\subset\mathbb{P}^5.
$$
Then $\widetilde{Q}(\mathbb{R})=\varnothing$ and $\phi(C)\subset \widetilde{Q}$.
Let $H=\{x_1+x_2+x_3=0\}$ and $Q=H\cap \widetilde{Q}$.
Then $Q$ is a smooth pointless quadric hypersurface in $H\simeq\mathbb{P}^4$ containing $\phi(C)$.
Thus, blowing up $Q$ along $\phi(C)$, we obtain a pointless smooth Fano 3-fold, over $\mathbb{R}$, whose geometric model is contained in Family \textnumero 2.21.
\end{example}

\begin{example}
\label{example:2-24}
Let $S$ be a pointless $\mathbb{Q}$-form of $\mathbb{P}^2$, and let $V=S\times S$. Then it follows from \cite[Chapter~7]{Collio} that $\mathrm{Pic}(V)$ contains a line bundle $L$ such that $L_{\mathbb{C}}$ is a divisor of degree $(1,-1)$ on $V_{\mathbb{C}}\simeq\mathbb{P}^2\times \mathbb{P}^2$. Let $X$ be a general divisor in the linear system $|L+\pi_2^*(-K_S)|$, where $\pi_2\colon V\to S$ is the projection to the second factor. Then $X$ is smooth, and $X_{\mathbb{C}}$ is a smooth Fano 3-fold in Family \textnumero 2.24. By construction, we have $X(\mathbb{Q})=\varnothing$, because $V$ does not have points in $\mathbb{Q}$.
\end{example}

\begin{example}
\label{example:3-2}
Let $C$ be a pointless conic in $\mathbb{P}^2$ over $\mathbb{R}$,
and let $\mathcal{E}$ be the restriction of the tangent bundle of  $\mathbb{P}^2$ to $C$.
Then it follows from \cite{BiswasNagaraj} that $\mathcal{E}$ is an indecomposable vector bundle on $C$ and $\mathcal{E}_{\mathbb{C}}$
splits as $\mathcal{O}_{\mathbb{P}^1}(3)\oplus \mathcal{O}_{\mathbb{P}^1}(3)$ on $C_{\mathbb{C}}\simeq\mathbb{P}^1$. Set
$$
V=\mathbb{P}\big(\mathcal{O}_{C}(-K_C)\oplus \mathcal{O}_{C}(-K_C)\oplus\mathcal{E}\otimes\mathcal{O}_{C}(K_C)\big),
$$
let $\eta\colon V\to C$ be the natural projection, let $M$ be the tautological vector bundle on $V$,
and let $X$ be a general divisor in the linear system $|3M-\pi^*(-K_C)|$. Then $X$ is smooth, and it follows from  \cite[\S~11]{CheltsovPrzyjalkowskiShramov} that $X_{\mathbb{C}}$ is a smooth Fano 3-fold in Family \textnumero 3.2.
We have $X(\mathbb{R})=\varnothing$, since $C(\mathbb{R})=\varnothing$.
\end{example}

\begin{example}
\label{example:3-5}
Let $C$ be a pointless real conic in $\mathbb{P}^2$,
let $S=C\times C$, let $\Delta$ be the diagonal curve in $S$,
and let $B$ be a general curve in the linear system $|\Delta+\mathrm{pr}_2^*(-2K_C)|$, where $\mathrm{pr}_2\colon S\to C$ is the projection to the second factor. Then $B_{\mathbb{C}}$ is a divisor of degree $(1,5)$ on $S_{\mathbb{C}}\simeq\mathbb{P}^1\times\mathbb{P}^1$.
Now, we identify $S$ with a surface in $C\times \mathbb{P}^2$ via the embedding $C\hookrightarrow\mathbb{P}^2$ of the second factor of $S$ and regard $B$ as a curve in $C\times \mathbb{P}^2$. Let $\pi\colon X\to C\times\mathbb{P}^2$ be the blowup of the curve $B$. Then $X(\mathbb{R})=\varnothing$, and $X_{\mathbb{C}}$ is a smooth Fano 3-fold in Family \textnumero 3.5.
\end{example}

\begin{example}
\label{example:3-6}
In the notations and assumptions of Example~\ref{example:2-19},
let $Q_1$ and $Q_2$ be two general surfaces in $|Q|$, let $C=Q_1\cap Q_2$, and let $\pi\colon X\to U$ be the blowup of the curves $L$ and $C$.
Then $X(\mathbb{R})=\varnothing$, and $X_{\mathbb{C}}$ is a smooth Fano 3-fold in Family \textnumero 3.6.
\end{example}

\begin{example}
\label{example:3-7}
Let $S$ be a $\mathbb{Q}$-form of $\mathbb{P}^2$ with no $\mathbb{Q}$-points, and let $S^\prime$ be the pointless $\mathbb{Q}$-form of $\mathbb{P}^2$ whose class in the Brauer group of $\mathbb{Q}$ is the inverse of the class of $S$. Set $V=S\times S^\prime$. Then $\mathrm{Pic}(V)$ contains a divisor $D$ such that $D_{\mathbb{C}}$ is a divisor of degree $(1,1)$ on $V_{\mathbb{C}}\simeq\mathbb{P}^2\times\mathbb{P}^2$. Let $Y_1$, $Y_2$, $Y_3$ be general divisors in $|D|$, set $C=Y_1\cap Y_2\cap Y_3$, and let $\pi\colon X\to Y_1$ be the blowup of the curve $C$.
Then $X(\mathbb{Q})=\varnothing$, and $X_{\mathbb{C}}$ is smooth Fano 3-fold in Family \textnumero 3.7.
\end{example}

\begin{example}
\label{example:3-10}
Let $Q$ be a pointless real smooth quadric 3-fold in $\mathbb{P}^4$, let $\Pi_1$ and $\Pi_2$ be general disjoint two-dimensional linear subspaces in $\mathbb{P}^4$, and let $\pi\colon X\to Q$ be the blowup of the conics $Q\cap\Pi_1$ and $Q\cap\Pi_2$. Then $X(\mathbb{R})=\varnothing$, and $X_{\mathbb{C}}$ is a smooth Fano 3-fold in Family \textnumero 3.10.
\end{example}

\begin{example}
\label{example:3-12}
In the notations and assumptions of Example~\ref{example:2-19},
let $S_1$ and $S_2$ be two general surfaces in $|Q|$ that contain the twisted line $L$. Then it follows from Example~\ref{example:2-19} that $Q_1\cdot Q_2=L+C$
where $C$ is a smooth geometrically rational curve such that $C_{\mathbb{C}}$ is a twisted cubic curve in $U_{\mathbb{C}}\simeq\mathbb{P}^3$. Let $L^\prime$ be a twisted line in $U$ such that $L^\prime_{\mathbb{C}}\cap C_{\mathbb{C}}=\varnothing$, and let $\pi\colon X\to U$ be the blowup of the curves $L^\prime$ and $C$.
Then, by construction, $X(\mathbb{R})=\varnothing$, and $X_{\mathbb{C}}$ is a smooth Fano 3-fold in Family \textnumero 3.12.
\end{example}

\begin{example}
\label{example:3-13}
Let $Q$ be the real smooth pointless quadric in $\mathbb{P}^4$ given by
$$
x^2+y^2+z^2+t^2+w^2=0,
$$
where $x,y,z,t,w$ are coordinates on $\mathbb{P}^4$.
Let $S$ be the hyperplane section of $Q$ that is cut out by $w=0$.
Then $S_{\mathbb{C}}$ contains conjugated lines $L_1=\{w=0,x=iy,z=it\}$ and
$L_2=\{w=0,x=-iy,z=-it\}$, and the curve $L_1+L_2$ is defined over $\mathbb{R}$.
Let $\alpha\colon\widetilde{Q}\to Q$ be the blowup of the curve $L_1+L_2$.
Then we have the following commutative diagram:
$$
\xymatrix@R=1em{
&\widetilde{Q}\ar@{->}[ld]_{\alpha}\ar@{->}[rd]^{\beta}&&\\%
Q\ar@{-->}[rr]_{\chi}&& W}
$$
where $W$ is a smooth Fano 3-fold with $W_{\mathbb{C}}$ a divisor of degree $(1,1)$ in $\mathbb{P}^2\times\mathbb{P}^2$,
$\beta$ is a birational morphism that contracts the strict transform of the surface $S$ on the threefold $\widetilde{Q}$ to a smooth curve in $W$,
and $\chi$ is a birational map.
Now, let $C_2$ be the conic in $Q$ that is cut out by the plane $\{x+t=0,y+z=0\}$.
The conic $C_2$ is disjoint from the curve $L_1+L_2$. Set $C=\chi(C_2)$, which is a smooth curve in $W$, where $\mathrm{pr}_1(C_{\mathbb{C}})$ and $\mathrm{pr}_2(C_{\mathbb{C}})$ are conics in $\mathbb{P}^2$,
and the induced morphisms $C_{\mathbb{C}}\to\mathbb{P}^2$ gives isomorphisms $C_{\mathbb{C}}\simeq\mathrm{pr}_1(C_{\mathbb{C}})$
and $C_{\mathbb{C}}\simeq\mathrm{pr}_2(C_{\mathbb{C}})$, where $\mathrm{pr}_1\colon W_{\mathbb{C}}\to\mathbb{P}^2$ and
$\mathrm{pr}_2\colon W_{\mathbb{C}}\to\mathbb{P}^2$ are projections to the first and the second factors of $\mathbb{P}^2\times\mathbb{P}^2$, respectively.
Thus, if we blowup $W$ along $C$, we obtain a pointless 3-fold over $\mathbb{R}$, whose geometric model is a smooth Fano 3-fold in Family \textnumero 3.13.
\end{example}

\begin{example}
\label{example:4-13}
Let $Q=C\times C$, where $C$ is a pointless conic defined over $\mathbb{R}$ and set $V=Q\times C$. Consider a divisor $Z\subset Q$ such that $Z_\mathbb{C}$ has degree $(1,1)$ on $Q_\mathbb{C}\cong \mathbb{P}^1\times\mathbb{P}^1$.
Let $S=\mathrm{pr}_1^*(Z)\subset V$, where $\mathrm{pr}_1\colon V \to Q$ is the projection onto the first factor, so that $S\cong Z\times C$. Observe that $S_\mathbb{C}\cong Z_\mathbb{C}\times C_\mathbb{C}\cong \mathbb{P}^1\times\mathbb{P}^1$. As both $Z$ and $C$ are pointless, and $\mathbb{P}^1$ has only one nontrivial form, we conclude that $Z\cong C$. Once again, let $D\subset S$ be a divisor so that $D_\mathbb{C}$ has degree $(1,1)$ on $S_\mathbb{C}\cong \mathbb{P}^1\times\mathbb{P}^1$. Let $B$ be a general curve in $|D+\pi_1^*(-K_Z)|$ where $\pi_1\colon S \to Z$ is the natural projection. Note that $B_\mathbb{C}$ is a divisor on $S_\mathbb{C}\cong \mathbb{P}^1\times\mathbb{P}^1$ of degree $(3,1)$, hence $B_\mathbb{C}$ is a curve on $V_\mathbb{C}$ of degree $(1,1,3)$. Let $\pi\colon X \to V$ be the blowup of $V$ along $B$. Then we obtain a pointless 3-fold over $\mathbb{R}$, whose geometric model is a smooth Fano 3-fold in Family \textnumero 4.13.
\end{example}

\medskip
\noindent
\textbf{Acknowledgements.} We thank J\'er\'emy Blanc, Jean-Louis Colliot-Th\'el\`ene, Adrien Dubouloz,
Kento Fujita, Laurent Manivel, Alexander Merkurjev, Alena Pirutka, Evgeny Shinder, Ronan Terpereau, Yuri Tschinkel, Alexander Vishik, and Ziquan Zhuang for fruitful discussions.
The bulk of this paper was worked out during a two-week CIRM RIR program in Luminy, and a two months stay of the second author in Saitama University.
We thank the staff in both institutions for facilitating high quality working environments for research.
Ivan Cheltsov has been supported by JSPS Invitational Fellowships for Research in Japan (S24067), EPSRC grant EP/Y033485/1,
and Simons Collaboration grant \emph{Moduli of Varieties}. Hamid Abban has been supported by EPSRC grant EP/Y033450/1 and a Royal Society International Collaboration Award ICA$\backslash$1$\backslash$231019. Takashi Kishimoto has been supported by JSPS KAKENHI Grant Number 23K03047.

\end{document}